\documentclass[11pt]{article}

\usepackage[utf8]{inputenc}
\usepackage[dvips]{graphicx}
\usepackage{calc}
\usepackage{color}
\usepackage{amsmath}
\usepackage{amssymb}
\usepackage{amscd}
\usepackage{amsthm}
\usepackage{amsbsy}
\usepackage{delarray}
\usepackage{enumerate}
\usepackage[T1]{fontenc}
\usepackage{inputenc}
\usepackage{enumerate}
\usepackage[all]{xy}
\usepackage{hyperref}
\usepackage{mathtools}
\usepackage{float}
\usepackage{pdfsync}
\usepackage{textcomp}
\usepackage{transparent}

\begin{document}

\setlength{\parindent}{5mm}
\renewcommand{\leq}{\leqslant}
\renewcommand{\geq}{\geqslant}
\newcommand{\N}{\mathbb{N}}
\newcommand{\sph}{\mathbb{S}}
\newcommand{\Z}{\mathbb{Z}}
\newcommand{\R}{\mathbb{R}}
\newcommand{\C}{\mathbb{C}}
\newcommand{\F}{\mathbb{F}}
\newcommand{\g}{\mathfrak{g}}
\newcommand{\h}{\mathfrak{h}}
\newcommand{\RN}{\mathbb{R}^{2n}}
\newcommand{\ci}{c^{\infty}}
\newcommand{\derive}[2]{\frac{\partial{#1}}{\partial{#2}}}
\renewcommand{\S}{\mathbb{S}}
\renewcommand{\H}{\mathbb{H}}
\newcommand{\eps}{\varepsilon}
\newcommand{\B}{\mathbf{B}}
\newcommand{\A}{\mathcal{A}}
\newcommand{\cF}{\mathcal{F}}
\newcommand{\ind}{\mathrm{L}}
\newcommand{\Conj}{\mathrm{Conj}}
\newcommand {\sign}{\mathrm {sign}}

\newcommand{\BX}{\mathbf{B}(X)}

\newcommand{\pfrac}[2]{\frac{\partial #1}{\partial #2}}
\newcommand{\Image}{{\rm Im}}
\newcommand{\Ker}{\rm Ker}
\newcommand{\Hom}{{\rm Hom}}
\newcommand{\hHom}{{\mathcal Hom}}
\newcommand{\Coker}{{\rm Coker}}
\newcommand{\Coim}{{\rm Coim}}
\newcommand{\codim}{{\rm codim}}

\newcommand{\Ob}{{\rm Ob}}
\newcommand{\Mor}{{\rm Mor}}
\newcommand{\Fun}{{\rm Fun}}
\newcommand{\Nat}{{\rm  Nat}}
\newcommand{\Hh}{{\mathcal H}}
\newcommand{\I}{{\mathcal I}}
\newcommand{\cI}{{\mathcal I}^{\bullet}}
\newcommand{\cH}{{\mathcal H}^{\bullet}}
\newcommand{\cK}{{\mathcal K}^{\bullet}}
\newcommand{\cV}{{\mathbf V}^{\bullet}}
\newcommand{\T}{{\mathcal T}}
\newcommand{\J}{{\mathcal J}}
\newcommand{\M}{{\mathcal M}}
\renewcommand{\L}{{\mathcal L}}
\renewcommand{\O}{{\mathcal O}}
\newcommand \id {{\rm id}}
\newcommand \Cat [1] {{\rm{\bf #1}}}
\newcommand{\supp}{{\rm supp}}

\DeclarePairedDelimiter{\ceil}{\lceil}{\rceil}

\theoremstyle{plain}
\newtheorem{theo}{Theorem}
\newtheorem*{theo*}{Theorem}
\newtheorem{prop}[theo]{Proposition}
\newtheorem{lemma}[theo]{Lemma}
\newtheorem{definition}[theo]{Definition}
\newtheorem*{notation*}{Notation}
\newtheorem*{notations*}{Notations}
\newtheorem{corol}[theo]{Corollary}
\newtheorem{conj}[theo]{Conjecture}
 \newtheorem{question}{Question}
\newtheorem*{conj*}{Conjecture}
\newtheorem*{claim*}{Claim}
\newtheorem{claim}[theo]{Claim}

\newcounter{numexo}[section]
\newtheorem{exo}[theo]{Exercice}
\newtheorem{exos}[numexo]{Exercices}

\newenvironment{demo}[1][]{\addvspace{8mm} \emph{Proof #1.
    ~~}}{~~~$\Box$\bigskip}

\newlength{\espaceavantspecialthm}
\newlength{\espaceapresspecialthm}
\setlength{\espaceavantspecialthm}{\topsep} \setlength{\espaceapresspecialthm}{\topsep}

\newtheorem{exple}[theo]{Example}
\renewcommand{\theexple}{}
\newenvironment{example}{\begin{exple}\rm }{\hfill $\blacktriangleleft$\end{exple}}

\newenvironment{remark}[1][]{\refstepcounter{theo} 
\vskip \espaceavantspecialthm \noindent \textsc{Remark~\thetheo
#1.} }%
{\vskip \espaceapresspecialthm}

\newenvironment{remarks}[1][]{\refstepcounter{theo} 
\vskip \espaceavantspecialthm \noindent \textsc{Remarks~\thetheo
#1.} }%
{\vskip \espaceapresspecialthm}

\def\bb#1{\mathbb{#1}} \def\m#1{\mathcal{#1}}

\def\del{\partial}
\def\co{\colon\thinspace}
\def\Homeo{\mathrm{Homeo}}
\def\Hameo{\mathrm{Hameo}}
\def\Diffeo{\mathrm{Diffeo}}
\def\Symp{\mathrm{Symp}}
\def\Sympeo{\mathrm{Sympeo}}
\def\Id{\mathrm{Id}}
\newcommand{\norm}[1]{||#1||}
\def\Ham{\mathrm{Ham}}
\def\Hamtilde{\widetilde{\mathrm{UHam}}}
\def\UHam{\mathrm{UHam}}

\def\Crit{\mathrm{Crit}}
\def\Spec{\mathrm{Spec}}
\def\EssSpec{\mathrm{EssSpec}}
\def\dbot{d_{\mathrm{bot}}}
\def\Leb{\mathrm{Leb}}

\newcommand{\bigslant}[2]{{\raisebox{.2em}{$#1$}\left/\raisebox{-.2em}{$#2$}\right.}}

\def\barcodes{\bigslant{\mathrm{Barcodes}}{\sim}}

\definecolor{claude}{rgb}{0.8,0,0}\newcommand{\claude}{\color{claude}}
\definecolor{sobhan}{rgb}{0,.6,0}\newcommand{\sobhan}{\color{sobhan}}
\definecolor{fred}{rgb}{0,0,0.8}\newcommand{\fred}{\color{fred}}


\title{Barcodes and area-preserving homeomorphisms }
\author{Fr\'ed\'eric Le Roux, Sobhan Seyfaddini, Claude Viterbo}
\date{\today}

\maketitle
\begin{abstract} 
In this paper we use the theory of barcodes as a new  tool for studying dynamics of area-preserving homeomorphisms. We will show that the barcode of a Hamiltonian diffeomorphism of a surface depends continuously on the diffeomorphism, and furthermore define barcodes for Hamiltonian homeomorphisms.

Our main dynamical application concerns the notion of {\it weak conjugacy}, an equivalence relation which arises naturally in connection to $C^0$ continuous conjugacy invariants of Hamiltonian homeomorphisms. We show that for a large class of Hamiltonian homeomorphisms with a finite number of fixed points, the number of fixed points, counted with multiplicity, is a weak conjugacy invariant.  The proof relies, in addition to the theory of barcodes, on techniques from surface dynamics such as Le Calvez's theory of transverse foliations.

In our exposition of barcodes and persistence modules, we present a proof of the Isometry Theorem which incorporates Barannikov's theory of simple Morse complexes.
 \end{abstract}

\tableofcontents


\section{Introduction and main results}
Our goal in writing this paper is to use a new set of tools to study dynamics of area-preserving homeomorphism. Floer homology has played an important role in studying dynamical features of Hamiltonian diffeomorphisms. However, it is not well-defined for non-smooth objects such as area-preserving homeomorphisms. As we will see in this article, \emph{barcodes} provide us with a medium through which one can define and effectively apply Floer theory for studying  dynamics of area-preserving homeomorphisms. 

  A barcode $\B= \{I_j\}_{j \in \N}$ is a countable collection of intervals (or bars) of the form $I_j = (a_j, b_j], a_j \in \mathbb{R}, b_j \in \mathbb{R}\cup \{ +\infty\}$ which satisfy certain finiteness assumptions.   Using Hamiltonian Floer theory one can associate, in a canonical manner, a  barcode  $\B(\phi)$ to every Hamiltonian diffeomorphism $\phi$ which encodes a significant amount of information about the Floer homology of $\phi$:  it completely characterizes the filtered Floer complex associated to $\phi$ up to quasi-isomorphism, and hence it subsumes all filtered Floer theoretic invariants of $\phi$. Barcodes have recently found several interesting applications in PDE, under the guise of Barannikov complex (see \cite{LNV}), and in symplectic topology; see for example the article of Polterovich and Shelukhin \cite{PolShel}.

   We will show in this article that the barcode $\B(\phi)$ depends continuously, with respect to the uniform topology, on $\phi$, and moreover, $\B(\phi)$ is well-defined even when $\phi$ is a Hamiltonian homeomorphism. Of course, when $\phi$ is a diffeomorphism, the barcode $\B(\phi)$ may be interpreted as the (filtered) Floer homology  of $\phi$.  

The barcode of a Hamiltonian homeomorphism $\phi$  is defined via  a limiting process and so the nature of the information it carries about the dynamics of $\phi$ is by no means clear.   To extract dynamical information from $\B(\phi)$, namely about the fixed points of $\phi$, we will use Le Calvez's theory of transverse foliations for dynamical systems on surfaces. 

\subsection{Rokhlin property and Hamiltonian homeomorphism}
Before presenting our results on barcodes, we will describe our main dynamical applications.

\medskip

 Let $(\Sigma, \omega)$ be a closed surface equipped with an area form $\omega$ and denote by $\overline{Ham}(\Sigma, \omega)$ the $C^0$ closure of Hamiltonian diffeomorphisms of $(\Sigma, \omega)$.  This is often referred to as the group of \emph{Hamiltonian homeomorphisms} of $(\Sigma, \omega)$. Although much is known about groups of  volume preserving homeomorphisms (see \cite{fathi}), the algebraic structure of this group is shrouded in mystery:  For example, it is not known whether it is simple or not and in the case where $\Sigma = S^2$ we do not know if it admits any homogeneous quasimorphisms; see \cite{muller-oh, leroux, EPP}.  Investigating questions of this nature, F. B\'eguin, S. Crovisier, and the first author of this paper were lead to ask if there exists $\theta \in \overline{Ham}(\Sigma, \omega)$ whose conjugacy class is dense in $\overline{Ham}(\Sigma, \omega)$?

The question of whether a given topological group has dense conjugacy classes  has been of interest in ergodic theory; see \cite{glasner-weiss01, glasner-weiss08} and the references therein.   Glasner and Weiss,  inspired by influential results of Halmos and Rokhlin, refer to groups with dense conjugacy classes as \emph{Rokhlin} groups.   Examples of Rokhlin groups include the group of measure-preserving automorphisms of a standard measure space equipped with the weak topology (see  p. 5 of \cite{glasner-weiss08}), the group of unitary operators of a separable, infinite-dimensional Hilbert space equipped with the strong topology, and  the group of homeomorphisms, which are isotopic to the identity, for  any even dimensional sphere equipped with the topology of uniform convergence; see  \cite{glasner-weiss08} for further details on these examples. 

Despite the last example of the previous paragraph, it turns out that the symplectic nature of $\overline{Ham}(\Sigma, \omega)$  forces a certain form of rigidity on its conjugacy classes: Indeed, $\overline{Ham}(\Sigma, \omega)$ is not Rokhlin.   In the case of surfaces with  genus greater than one, this is a consequence of Gambaudo-Ghys \cite{gambaudo-ghys} who proved existence of continuous homogeneous quasimorphisms on $\overline{Ham}(\Sigma, \omega)$.  The case of the surface of genus one, follows from a combination of results from Entov-Polterovich-Py \cite{EPP} and Gambaudo-Ghys \cite{gambaudo-ghys} (the latter constructs homogeneous quasimorphism on the group, while the former proves their continuity).  In the case of $S^2$, the negative answer was provided by the second  author in \cite{Sey12}.  In each case, the proof involves constructing continuous conjugacy invariants.  One can not easily associate a dynamical interpretation to these invariants, particularly in the case of $S^2$ where the invariant is the so-called spectral norm which is derived from Hamiltonian Floer theory and  whose construction is rather complicated.  

In this article, we introduce, via the theory of barcodes,  a very simple invariant, whose definition involves a \emph{weighted count of fixed points}, to separate closures of conjugacy classes.  In fact, we will see that our weighted sum of fixed points, and more generally barcodes, are invariants of a natural equivalence relation which we will refer to as \emph{weak conjugacy}:  This  is the strongest Hausdorff equivalence relation on $\overline{\Ham}(M,\omega)$  which is weaker than the conjugacy relation.

The weak conjugacy relation is  characterized by the following universal property: $f$ is weakly conjugate to $g$  if and only if $\theta (f) = \theta(g)$ for any continuous function $\theta : \overline{\Ham}(M, \omega) \to Y$ such that $\theta$ is invariant under conjugation and $Y$ is a Hausdorff topological space. See Definition \ref{def:almost_conj} in Section \ref{sec:almost_conj} for further details. 

It is evident from the above characterization of weak conjugacy that if $\overline{\Ham}(M, \omega)$ possessed a dense conjugacy class then the weak conjugacy relation would be trivial, {\it i.e.}\ any $f$ and $g$ would be weakly conjugate.  We remark  that $f,g$ are weakly  conjugate if they satisfy the following criterion:  there exist $h_1, \ldots, h_N \in \overline{\Ham}(M, \omega)$ such that $h_1= f,  h_N = g$, and $\overline{\mathrm{Conj}}(h_i) \cap \overline{\mathrm{Conj}}(h_{i+1}) \neq \emptyset$; here  $\overline{\mathrm{Conj}}$ stands for closure of conjugacy class.

Observe that any continuous conjugacy invariant is a weak conjugacy invariant as well.  As a consequence, the notion of weak conjugacy arises naturally in settings where one needs to consider continuous conjugacy invariants.  This is for example the case in the study of mapping class group actions on the circle; see the article by K. Mann and M. Wolff \cite{Mann-Wolff} and the references therein.

\subsubsection{Weak conjugacy and the Lefschetz index}
In order to formulate our results, we will need to introduce a few notions.

Let $x$ be an isolated fixed point of a homeomorphism $f$ of a manifold of dimension $k$.  Recall that the Poincar\'e-Lefschetz index of the fixed point $x$, which we will denote by $\ind(f, x)$,  is defined as follows:  Let $U$ be a chart centered at $x$ and denote by $S$ a small sphere (of codimension $1$ in $U$) which is  centered at $x$ as well, oriented as the boundary of its interior in $U$.  For $S$  sufficiently small, the formula
$$x \mapsto \frac{f(x)-x}{\Vert f(x)- x \Vert}$$ yields a well-defined map from $S$ to the unit sphere in $\R^k$ and $\ind(f,x)$ is the degree of this map. For further details on this index, we refer the interested reader to \cite{katok_hassel, hatcher}.

Let $(\Sigma, \omega)$ be a closed and connected symplectic surface.  Given $f \in \overline{\Ham}(\Sigma, \omega)$, we will denote the set of contractible fixed points of $f$ by $Fix_c(f)$; recall that a fixed point $x$ of $f$ is said to be contractible if there exists an isotopy $I_t$, where $t \in [0,1], I_0 = \Id, I_1 = f, I_t \in \overline{\Ham}(\Sigma, \omega)$, such that the loop $I_t(x)$ is contractible.  This notion is well-defined, i.e. does not depend on the choice of the isotopy, because the proof of the Arnold conjecture implies that for a loop in $\Ham (M, \omega)$ based at the identity all trajectories are contractible
(In our situation,  if $\Sigma \neq \S^2$, then $\overline{\Ham}(\Sigma, \omega)$ is simply connected; see \cite{pol_book} for a proof).  

We are now ready to state  our first result.

\begin{theo}\label{theo:lef_index_invariance}
Suppose that $f,g \in \overline{\Ham}(\Sigma, \omega)$ are smooth and have finitely many contractible fixed points.  If $f$ and $g$ are weakly conjugate, then $$\sum_{x\in Fix_c(f)} |\ind(f,x)| = \sum_{x\in Fix_c(g)} | \ind(g,x)|.$$ 
\end{theo}

We will refer to $\sum_{x\in Fix_c(f)} |\ind(f,x)|$ as the \emph{absolute Lefschetz number} of $f$. A few remarks are in order.  First, observe that one can immediately conclude from the above that $\overline{\Ham}(\Sigma, \omega)$ is not Rokhlin:  As mentioned earlier, we must produce $f, g$ which are not weakly conjugate.  It is very easy to produce  $f$ and $g$ with different absolute Lefschetz numbers, and hence they cannot be weakly conjugate, by just creating a pair of fixed point of index $\pm 1$. 

Second, we should mention that if $f$ is non-degenerate, then every fixed point has index $\pm 1$ and so the absolute Lefschetz number is just the total number of fixed points of $f$.  However, the above result would not be true if we were to replace the absolute Lefschetz number by the total number of fixed points.  Indeed, it is possible to produce $f,g$ which are weakly conjugate but do not have the same number of fixed points.

Third, observe that if $f$ and $g$ are weakly conjugate, then so are $f^p$ and $g^p$ for every $p\in \Z$.  Hence, if $f,g$ have finitely many periodic points of period $p$, then the absolute Lefschetz numbers of $f^p$ and $g^p$ coincide as well.

Finally, we remark that our proof of Theorem \ref{theo:lef_index_invariance} yields the following refinement: We denote by $\Spec(f)$ the set of action values of a given Hamiltonian diffeomorphism $f$.  Recall that $\Spec(f)$ is well-defined up to a shift.  Now, suppose that $f\in \Ham(M, \omega)$ has finitely many fixed points.  For every value $a \in \Spec(f)$ define $\ind(f,a):= \sum |\ind(f,x)|,$ where the sum is taken over all $x \in Fix_c(f)$ whose action is $a$.  Define 
$ \widetilde{\Spec}(f) : = \{(a, \ind(f,a)): a \in \Spec(f)\}.$   If  $f,g$ are as in the statement of Theorem \ref{theo:lef_index_invariance}, then $\widetilde{\Spec}(f) = \widetilde{\Spec}(g)$. 
 In fact, one can even refine this latter statement to take into account Conley-Zehnder indices of fixed points with a given action.  These statements are immediate consequences of Theorem \ref{theo:barcode_inv} and Proposition \ref{prop:lefindex=rank}.

\medskip  
Our next result is a generalization of Theorem \ref{theo:lef_index_invariance} to the setting where $f,g$ are not assumed to be smooth.   As will be explained in the next section, the removal of the smoothness assumption gives rise to significant complications.  Our strategy for dealing with these complications passes through the following notion:    Consider $f \in \overline{Ham}(\Sigma, \omega)$ with finitely many fixed points. We say that $f$ is \emph{smoothable} if there exists a sequence $f_i \in Ham(\Sigma,\omega )$ which converges uniformly to $f$ and such that $Fix_{c}(f_i)=Fix_{c}(f)$.   

\begin{theo}\label{theo:lef_index_invariance_homeos}
Suppose that $f,g \in \overline{\Ham}(\Sigma, \omega)$ are smoothable and have finitely many contractible fixed points.  If $f$ and $g$ are weakly conjugate, then $$\sum_{x\in Fix_c(f)} |\ind(f,x)| = \sum_{x\in Fix_c(g)} | \ind(g,x)|.$$ 
\end{theo}

We conjecture that every $f \in \overline{Ham}(\Sigma, \omega)$ with finitely many fixed points is smoothable, and we verify  this  for a large class of homeomorphisms.    It follows from our results that a non-smoothable $f$ (which conjecturally does not exist) would be dynamically very complicated near its (contractible) fixed point set.   In Section \ref{sec:Ham-homeo-smoothable}, we establish a precise criterion for smoothability which, in particular, implies the following statement.

\begin{theo}\label{theo:smoothable-easy}
Suppose that $f \in \overline{\Ham}(\Sigma, \omega)$ has a finite number of contractible fixed points. Assume that there does not exist $x \in Fix_c(f)$  which is accumulated by periodic orbits of every period. Then, $f$ is smoothable.
\end{theo}

The above is an immediate consequence of Theorem \ref{theo:smoothable}.

\subsubsection{A homeomorphism which is not weakly conjugate to any diffeomorphism}

It is well-known that  homeomorphisms could be dynamically  more complicated than diffeomorphisms.  The result below on the weak conjugacy relation tells us that barcodes are capable of detecting the wilder dynamics of  homeomorphisms.  
 \begin{theo}\label{theo:almost_conj2}
   There exists $f \in \overline{\Ham}(\Sigma, \omega)$ which is not weakly conjugate to any Hamiltonian diffeomorphism.  In particular, the closure of the conjugacy class of $f$ contains no Hamiltonian diffeomorphisms.
 \end{theo}

\subsection{Barcodes as invariants of Hamiltonian homeomorphisms}
We will now explain how barcodes enter our story by presenting a brief outline of the proofs of the results mentioned in the previous section.    Here, in order to avoid the complications which arise in the case of $\S^2$, we will focus on surfaces of positive genus.  

As we will see in Section \ref{sec:floer}, using Hamiltonian Floer theory, one can associate a barcode $\B(H)$  to every Hamiltonian $H \in C^{\infty}(\S^1 \times \Sigma).$\footnote{In fact, Hamiltonian Floer theory may be used to construct barcodes on symplectic manifolds far more general than surfaces; see \cite{UZ}.}  It turns out that if $H, G$ are two Hamiltonians the time-$1$ maps of whose flows coincide, then there exists a constant $c$ such that $\B(G) = \B(H) + c$  where $\B(H)+c $ is the barcode obtained from $\B(H)$ by shifting each of the bars in $\B(H)$ by $c$.  This implies that we have a well-defined map $\B: \Ham(\Sigma, \omega)  \rightarrow \widehat{\mathcal B},$ where $\widehat{\mathcal B}$ denotes  the space of barcodes modulo the equivalence relation $\B_1 \sim \B_2$ if $\B_2 = \B_1 + c$ for some constant $c$.  

Now, $\widehat{\mathcal B}$ may be equipped with a natural distance $\dbot$ which is called the bottleneck distance.    Consider the mapping of metric spaces $$\B: (\Ham(\Sigma, \omega), d_{C^0}) \rightarrow (\widehat{\mathcal B}, \dbot), $$ where $d_{C^0}$ denotes the $C^0$ distance.    We prove in Theorem \ref{theo:cont_barcodes_genus}, and Remark \ref{rem:cont_total_barcode}, that the above mapping is continuous and, moreover, extends continuously to $ \overline{\Ham}(\Sigma, \omega)$.  This is what allows us to associate barcodes to Hamiltonian homeomorphisms. This is generalized to aspherical manifolds in \cite{BHS18}. 

It follows from the aforementioned continuity of barcodes, and standard properties of Hamiltonian Floer theory, that if $f,g \in \overline{\Ham}(\Sigma, \omega)$ are weakly conjugate,  then $\B(f) = \B(g)$; see Theorem \ref{theo:barcode_inv}.  Clearly, to prove Theorems \ref{theo:lef_index_invariance} \& \ref{theo:lef_index_invariance_homeos} it is sufficient to show that the absolute Lefschetz number of $f \in \overline{\Ham}(\Sigma, \omega)$ is an invariant of its barcode $\B(f)$.  This is achieved in Theorems \ref{theo:lef_endpoints} \& \ref{theo:lef_endpoints_homeos}: When $f$ is smooth, we use local Floer homology to show that the absolute Lefschetz number of $f$ is simply the total number of endpoints (counted with multiplicity) of the bars of $\B(f)$.  We show that the same conclusion continues to hold when $f$ is not smooth but is smoothable; we do so by proving that for $f$ smoothable, one can find smooth $f_n$'s converging to $f$ such that $\B(f_n) = \B(f)$.  This will then allow us to prove Theorem \ref{theo:lef_index_invariance_homeos} as a consequence of Theorem \ref{theo:lef_index_invariance}.  Without the assumption of smoothability, we have little control over  $\B(f_n)$ and so cannot say much about $\B(f)$.  

\medskip

We end this part of the introduction by giving a sketch of the proof of Theorem \ref{theo:almost_conj2}.  As we will see, if $h$ is a Hamiltonian diffeomorphism, then the endpoints of the bars in  $\B(h)$  correspond to actions of fixed points of $h$.  Now, the set of actions of fixed points of a Hamiltonian diffeomorphism (on a closed surface other than $\S^2$)  is always compact.  To prove the theorem, we simply produce  a Hamiltonian homeomorphism $f$  with the property that the set of endpoints of its barcode $\B(f)$ is not bounded.

\subsection{Smoothability, transverse foliations and local dynamics}
We will end the introduction of our paper with a few words on the proof of Theorem \ref{theo:smoothable-easy}.  This proof relies on techniques  from surface dynamics.  Namely, Le Calvez's theory of transverse foliations for dynamical systems of surfaces as described in~\cite{lecalvez05}, and the notion of local rotation set for an isolated fixed point which was introduced by the third author in \cite{leroux13}.

The local rotation set is a topological conjugacy invariant associated to a germ of orientation preserving surface homeomorphism near a fixed point. It is a closed  interval of $[-\infty, +\infty]$, defined modulo translation by integers, which captures the amount of asymptotic rotation of orbits around the fixed point. Furthermore, when the fixed point $p$ is isolated and the homeomorphism $f$ is area preserving, the local rotation set is a subset of $[0,1]$. The precise content of Theorem~\ref{theo:smoothable}, which implies Theorem \ref{theo:smoothable-easy}, is the following:  if the local rotation set is a \emph{proper} subset of $[0,1]$, then there is an area-preserving homeomorphism $g$ that is smooth near $p$, coincides with $f$ outside some small neighborhood of $p$, and has the same fixed point set as $f$.

The proof of the theorem is divided into two cases. It is known that the Lefschetz index $L(f,p)$ of an isolated fixed point for an area preserving homeomorphism is less or equal to $1$. 
When $L(f,p) <1$, we show that, after a first small perturbation, there exists a small disk $D$ containing $p$ which is in \emph{canonical position} for $f$, which mainly means that $f(D) \cap D$ is connected. This first step makes a fundamental use of Le Calvez's transverse foliations. The canonical position is then used to approximate $f$ by surgery, replacing the restriction of $f$ to $D$ by a smooth model. This approach is inspired by the technique designed by Schmitt and used by Slaminka to remove index $0$ fixed points (see~\cite{schmitt79,slaminka93}, even though some arguments in these papers are rather hard to follow).
When $L(f,p) =1$ we use a different strategy. In this case Le Calvez's transverse foliation provides a coordinates system for which the ``$\theta$'' polar coordinate is essentially increasing along every trajectory of some isotopy from the identity to $f$. Assuming the local rotation set is included in $[0,1)$, we refine Le Calvez's result, finding a coordinates system in which, in addition, at least one ray $\theta=\text{constant}$ is mapped by $f$ to another ray. This is the content of the ``iterated leaf lemma'' (Lemma~\ref{lem:iterated-leaf} below), which is of independent interest. This property allows us to make a series of perturbations by pushing points in the direction positively transverse to the rays, adding no new fixed points and ending up with a map that admits a periodic ray of period 2. This  map is finally (and easily) smoothed into a map that is a periodic 2 rotation near $p$.

\subsection*{Organization of the paper}  
In Section \ref{sec:per-barann-barcodes}, we  present a brief review of persistence modules, Barannikov complexes, barcodes and the metrics which are naturally associated to them.  In Section \ref{sec:floer}, after presenting a brief review of Hamiltonian Floer theory, we explain how one can associate barcodes to Hamiltonian diffeomorphisms via Hamiltonian Floer homology.  In Section \ref{sec:barcodes_surfaces} we prove our main results on $C^0$ continuity of barcodes.  Section \ref{sec:proof_invariance_Lef} contains the proofs of Theorems \ref{theo:lef_index_invariance} and \ref{theo:lef_index_invariance_homeos}.  Theorem \ref{theo:almost_conj2} is proven in Section \ref{sec:almost_conj}.

In Section \ref{sec:Ham-homeo-smoothable}, we state and prove a more general version of Theorem \ref{theo:smoothable-easy}; see Theorem \ref{theo:smoothable}.  This section contains a brief review of the necessary background from dynamical systems such as LeCalvez's theory of transverse foliations and the notion of local rotation set for isolated fixed points of area-preserving homeomorphisms.

\subsection*{Acknowledgments}
Subsequent to the announcement of the results of this article, Theorem \ref{theo:cont_barcodes_genus} on  continuity of barcodes was generalized to higher dimensional aspherical symplectic manifolds 
by Buhovsky, Humili\`ere and the second author (see~\cite{BHS18}). Their methods are very different than ours.   Thus, Theorem 1  generalizes to the aspherical situation, but with $\vert L(f,x)\vert$ replaced by $r(f,x)$, the rank of local Floer homology of $x$. 

Continuity of barcodes could also be proven by combining the results of the articles \cite{shel} and \cite{Sey12}.  This fact was brought to the attention of the second author by Egor Shelukhin in the course of a private communication.   This is the strategy used in \cite{BHS18}.

   We would like to thank Lev Buhovsky, Sylvain Crovisier, Viktor Ginzburg, Vincent Humili\`ere,  Patrice Le Calvez, and Maxime Wolff for helpful conversations.  
 
This paper was partially written during the second author's stay at the Institute for Advanced Study who greatly benefited from the lively research atmosphere of the IAS and would like to thank the members of the School of Mathematics for their warm hospitality.

\section{Preliminaries on Persistence modules, Barannikov complexes and barcodes}\label{sec:per-barann-barcodes}
   In this section, we introduce persistence modules,  barcodes and Barannikov modules and the natural metrics associated to each of these objects.  The main goal here is to prove that the above objects define equivalent categories and that the natural correspondences between them are isometries; see Proposition \ref{prop:isometry}.  
   
  Persistence modules and barcodes have been studied extensively in the topological data analysis community; see \cite{Cohen-Steiner2007, chazal2009} and the references therein.  The isometric correspondence between barcodes and persistence modules is well-known and is referred to   as the Isometry Theorem.  What seems to be less well-known  is the correspondence between Barcodes and persistence homology, on the one hand, and  Barannikov's simple Morse complex  \cite{Barann}, on the other hand. This aspect of Proposition \ref{prop:isometry} is of rather folkloric nature and some features of it have already been exploited in the symplectic community; see for example \cite{PushChek, LNV, laudenbach, UZ}.
  
  We should add that our presentation of Barannikov's work is more closely aligned with the presentation in \cite{LNV} rather than Barannikov's original article \cite{Barann}.
  
  \medskip
   
  For the rest of this section, we consider vector spaces over a fixed field $\mathbb F$. It is important to keep the same field all along as some Morse/Floer homological constructions of the following section (see Section \ref{sec:floer}) depend on the choice of $\mathbb F$.
 
\subsection{Persistence modules}\label{sec:persistence}
We begin by defining persistence modules, their morphisms and the interleaving distance.

\begin{definition} \label{def:per-mod} A {\bf persistence module} $\mathbf V$  is a family $(V_s)_{s\in \mathbb R }$ of vector spaces equipped with morphisms $i_{s,t}:V_s \longrightarrow V_t$, for $s\leq t$, satisfying: 
\begin{enumerate}
\item For all $s\in \R$ we have $i_{s,s}=\Id$ and for every $s\leq t\leq u$ we have  $i_{t,u}\circ i_{s,t}=i_{s,u}$, 
\item There exists a finite subset $ F \subset {\mathbb R} $, often referred to as the spectrum of   $\mathbf V$, such that  $i_{s,t}$ is an isomorphism whenever $s,t$ belong to the same connected component of $ {\mathbb R} \setminus F$,
\item For all $t\in\R$, $\varinjlim_{s<t}V_s=V_t$ ; equivalently, for fixed $t$, $i_{s,t}$ is an isomorphism for $s<t$ sufficiently close to $t$.
\end{enumerate}
\end{definition}

  Here are two examples of persistence modules:  
First, consider an interval $I$ of the form $(a, b]$  and define $Q_s(I) = \F$, if $s\in I$, and $Q_s(I) = \{0\}$, if $s\notin I$.  Then, $Q_s(I)$ is a persistence module, with $i_{s,t}$ equal to $\Id$ if $s,t \in I$ and $0$ otherwise.
Second, let $f: M\rightarrow \R$ be a continuous function on a closed manifold $M$.  The family of vector spaces $V_s := H(M^s)$, the singular homology of $M^s := \{x\in M : f(x) < s\}$, is a persistence module, where the $i_{s,t}$'s are induced by inclusion of $M^{s}$ into $M^t$.  
\medskip

We denote by $\mathcal P$ the category of persistence modules. An element in $Mor( \mathbf V, \mathbf W)$  is given by a family $u_s: V_s \longrightarrow W_s$ commuting with $i_{s,t}$  and $j_{s,t}$.
$$
\xymatrix {V_{s}\ar[r]^{u_s}\ar[d]^{i_{s,t}}&W_{s} \ar[d]^{j_{s,t}}\\ V_t \ar[r]^{u_t} & W_t }
$$

  If $\mathbf V$ is a persistence module, we denote by $\tau_a\mathbf V $ the persistence module defined by $(\tau_aV)_s=V_{a+s}$. For $a>0$, we have a natural morphism $i_a: \mathbf V \longrightarrow \tau_a\mathbf V$ induced by $i_{s,s+a}$. 
  
  \medskip 
  
  The set of persistence modules can be equipped with the so-called interleaving pseudo-distance.

\begin{definition} 
Let ${\mathbf V}=(V_s)_{s\in {\mathbb R} }$ and ${\mathbf W}=(W_s)_{s\in {\mathbb R} }$ be two persistence modules. The pseudo-distance $d_{int}(\mathbf V, \mathbf W)$,  called the {\bf interleaving} distance, is defined as the infimum of the set of $ \varepsilon $ such that there are morphisms $\varphi_s: V_s \longrightarrow W_{s+ \varepsilon }$ and $\psi_s: W_s \longrightarrow V_{s+ \varepsilon }$  compatible with the $i_{s,t}, j_{s,t}$ 

$$\xymatrix{V_{s- \varepsilon }\ar[r]^{\varphi_{s- \varepsilon }}\ar[d]^{i_{s- \varepsilon ,t- \varepsilon }}& W_{s }\ar[d]^{j_{s ,t }} \ar[r]^{\psi_{s }} & V_{s+ \varepsilon }\ar[d]^{i_{s+ \varepsilon ,t+ \varepsilon} }\ar[r]^{\varphi_{s+ \varepsilon }}& W_{s+ 2 \varepsilon }\ar[d]^{j_{s+2 \varepsilon ,t+2 \varepsilon }}\\ V_{t- \varepsilon } \ar[r]^{\varphi_{t- \varepsilon }}&W_{t } \ar[r]^{\psi_t}& V_{t+ \varepsilon } \ar[r]^-{\varphi_{t+ \varepsilon }} & W_{t+2 \varepsilon } & \\ 
}$$
such that $\psi_s\circ \varphi_{s- \varepsilon } =i_{s- \varepsilon , s+ \varepsilon }$ and $\varphi_{s+ \varepsilon }\circ \psi_s= j_{s,s+2\varepsilon }$.   In other words, $\varphi: \mathbf V \longrightarrow \tau_ \varepsilon \mathbf W$ and $\psi: \mathbf W \longrightarrow \tau_ \varepsilon \mathbf V$ are such that $\varphi\circ \tau_ \varepsilon \psi = j_{2 \varepsilon }, \psi\circ  \tau_ \varepsilon \varphi=i_{2 \varepsilon }$. 
\end{definition}Note that the pseudo-distance takes value in $[0,+\infty]$.
\begin{prop} [See \cite{chazal2009}]
The interleaving distance defines a pseudometric on the set of persistence modules. Furthermore, two persistence modules at distance zero are isomorphic.  
\end{prop} 
\begin{proof}  
We  leave the proof of the fact that $d_{int}$ is a pseudo-distance as an exercise for the reader and will only sketch the proof of the second statement.

Let us assume that $d_{int}(\mathbf V, \mathbf W) = 0$. Thus,  we have morphisms $\varphi_ \varepsilon : \mathbf V \longrightarrow \tau_ \varepsilon \mathbf W$ and $\psi_ \varepsilon : \mathbf W \longrightarrow \tau_ \varepsilon \mathbf V$ such that $\varphi_ \varepsilon \circ \tau_ \varepsilon \psi_ \varepsilon = j_{2 \varepsilon }, \psi_ \varepsilon \circ  \tau_ \varepsilon \varphi_ \varepsilon =i_{2 \varepsilon }$, where $\varepsilon > 0$ maybe picked to be as small as one wishes.   We must show that $\mathbf V \simeq \mathbf W$. 
The idea is simple: as $ \varepsilon $ converges to zero, $i_{2 \varepsilon }$ and $j_{2 \varepsilon }$ ``converge'' to the identity morphism, and so  we would like to show that $\varphi_ \varepsilon , \psi_ \varepsilon $ ``converge'' to $\varphi_0, \psi_0$ and get that $\varphi_0\circ \psi_0=\psi_0\circ \varphi_0=\id$.

Fix any $s\in \R$ which is neither in the spectrum of $\mathbf V$ nor in the spectrum of $\mathbf W$.  For $ \varepsilon $ small enough, the vector spaces $W_{s+ \varepsilon }$ and $V_{s + \varepsilon }$ are stationary in the sense  that $ W_s \simeq W_{s+\varepsilon}$ and $V_s \simeq V_{s+\varepsilon}$. Thus, we can  identify $\varphi_ \varepsilon: V_s \longrightarrow W_{s+ \varepsilon} $ and $\psi_\varepsilon : W_s \longrightarrow V_{s+ \varepsilon} $ to  maps $\varphi_0: V_s \longrightarrow W_s$  and $\psi_0: W_s \longrightarrow V_s$, respectively.  It is not difficult to see that $\varphi_0\circ \psi_0=\psi_0\circ \varphi_0=\id$.  One can further check, since $V_s=V_{s- \varepsilon }$ for $s$ in the spectrum and $ \varepsilon $ small enough,  that  for $s$  in the spectrum of  $\bf V$, or  $\bf W$, one can still define  $\varphi_0: V_s \longrightarrow W_s$,  $\psi_0: W_s \longrightarrow V_s$ and that the two maps continue to be inverse to one another.
\end{proof}

Note that for any $t\in \R$, there exists $\varepsilon> 0$ such that  $i_{s,u}:V_s \longrightarrow V_u$ is an isomorphism if $s, u \in (t-\varepsilon,t]$ or if $s, u \in (t, t+\varepsilon)$.  Pick $t^- \in (t- \varepsilon, t]$ and $t^+ \in (t, t + \varepsilon)$ and let $j(t)=\dim (\Ker (i_{t^-,t^+}))+ \codim (\Image (i_{t^-,t^+}))$.  Observe that $j(t)$ is zero except for $t\in F$.  We say that $\mathbf V$ is {\bf generic} if  $j(t) \leq1$ for all $t\in \R$.

\begin{exo} \label{exo-10}
Prove that generic persistence modules are dense for the interleaving distance. 
\end{exo} 

\begin{remarks} 
\begin{enumerate}

\item The notion of persistence modules can be translated in term of sheaves, as in \cite{kashiwara-schapira2}. We endow $ {\mathbb R} $ with the (non-separated) topology for which open sets are the sets $]s, +\infty[$, and denote by $\widetilde{\mathbb R}$ this topological space,  then the $\mathbf V ( ]s,+\infty[)=V_s$ defines a presheaf on $ \widetilde{\mathbb R} $. This will be a sheaf provided $V_t=V_{t^+}=\lim_{s>t}V_s$, which we assume.  
 We moreover assume this sheaf is constructible, that is there is a finite subset of $ {\mathbb R} $, $F$,  such that 
$i_{s,t}$ is an isomorphism whenever $s,t$ belong to the same connected component of $ {\mathbb R} \setminus F$. 
The sheaf is flabby hence injective,  if and only if $i_{s,t}$ is onto for all $s<t$. 
We shall not insist on the sheaf-theoretic point of view, for which we refer to \cite{kashiwara-schapira2}.

\item
One can generalize the notion of persistence modules to that of persistence complexes, where $V_s$ is graded in a manner compatible with the filtration, i.e. $V_s=\oplus_pV_s^p$, and the maps respect the grading and differentials. This corresponds to the abelian category of complexes of constructible sheaves.  
\end{enumerate}
\end{remarks}

\subsection{Barannikov modules}\label{sec:barannikov}
This section is dedicated to the introduction of Barannikov modules.  Before presenting the definition, let us  recall that a filtration on a vector space $C$ is a family $C^s$, $s \in \R $, of vector subspaces of $C$  which is increasing ({\it i.e.\ }$C^s \subset C^t$ if $s<t$) and such that $\displaystyle C^s = \cup_{t<s} C^t$.  We will furthermore require that $C^s = C$ for $s>>0$  and $C^{s_1}= C^{s_2}$ for $s_1, s_2 <<0$.  A filtration on a chain complex $(C, \partial)$ is a filtration of $C$ such that $\partial(C^s) \subset C^s$.

\begin{definition}
A {\bf Barannikov module}, also called a {\bf simple module} or a {\bf simple chain complex}, is a  finite dimensional filtered chain complex $\mathbf C = (C, \partial)$ which is endowed with a  preferred basis {\bf $B$} such that:

\begin{enumerate} 
\item $C$ admits a decomposition of the form $C=C_+\oplus C_-\oplus C_0$ which  is compatible with the filtration, i.e. $C^s=C^s_+\oplus C^s_-\oplus C^s_0$, where $C^s_+,C^s_-,C^s_0$ are the filtrations induced on $C_+, C_-,C_0$.

\item The preferred basis $ B$ is compatible with the decomposition and the filtration in the sense that  $C^s_+, C^s_-,C^s_0$ are generated by $B \cap C^s_+, B \cap  C^s_-,  B \cap C^s_0$, respectively. We will denote  $ B_+ =  B \cap C_+,  B_- =  B \cap C_-, B_0 = B \cap C_0$.  

\item The differential $\partial$ gives a bijection from $ B_+$ to $B_-$, and $\partial(B_-) = \partial(B_0) = \{0\}.$
\end{enumerate}
\end{definition}  

Let us emphasize that $\partial$ yields an isomorphism between $C_+$ and $C_-$ sending basis elements to basis elements and that $\partial$ vanishes on $C_-\oplus C_0$.

\medskip 

We will denote by  $\mathcal C$  the category of Barannikov modules.  
 Given two Barannikov modules $\mathbf C = (C,\partial_C)$ and $\mathbf D = (D,\partial_D)$, an element in  $Mor( \mathbf C, \mathbf D)$  is a chain map which respects the filtration and the preferred bases. 
 
 Given a Barannikov module $\mathbf C = (C, \partial)$ and any $a \in \R$ we define the Barannikov module $\tau_a \mathbf C$  by shifting the filtration by $a$.  For $a>0$, we have  a natural map $i_a : (C_*,\partial_C) \longrightarrow (\tau_aC_*,\partial_C)$.
 
 \medskip

As we will now see, the set of Barannikov modules may be equipped with a natural metric.

\begin{definition} 
The distance between two simple modules, $\mathbf C=(C,\partial_C)$ and $\mathbf D=(D,\partial_D)$ is defined as the infimum of the set of $ \varepsilon $ such that there are filtration preserving  morphisms $\varphi: (C,\partial_C) \longrightarrow (\tau_ \varepsilon D,\partial_D)$ and $\psi: (D,\partial_D) \longrightarrow (\tau_ \varepsilon C,\partial_C)$ such that $\psi \circ \varphi =j_{ 2 \varepsilon }$ and $\varphi \circ \psi = i_{ 2 \varepsilon }$. 
We denote it by $d_{s}(\mathbf C, \mathbf D)$. 
\end{definition} 

We will now define maps $H : \mathcal C \rightarrow \mathcal P$ and $\gamma: \mathcal P \rightarrow \mathcal C$, which will  be proven to be isometries, between the two categories of Barannikov and persistence modules.

\begin{prop} \label{prop:barann-to-persistence}
Let $\mathbf C = (C,\partial_C)$ be a simple chain complex.  Then, the family of vector spaces $V_s :=H(C^s,\partial_C)$, the homology of $C^s$, equipped with the maps $i_{s,t}: V_s \longrightarrow V_t$  induced by  the inclusion map $ C^s \longrightarrow C^t$, is a persistence module.  
\end{prop}
We leave the proof of this proposition to the reader.  Note that for $s$ large enough, we have $V_s=H(C^s,\partial_C)= C_0$. 

\medskip The map $\gamma: \mathcal P \rightarrow \mathcal C$ is constructed/defined below.  We will not prove in detail that $\gamma$ is well-defined as the proof, albeit in a different setting, can be found in \cite{LNV}; see Section 2 therein. 

\begin{definition} \label{Def-1.10}
Let  $\mathbf V=(V_s)_{s\in {\mathbb R} }$ be a generic persistence module.  We define a simple module $\mathbf C = \gamma(\mathbf V)$ as follows:

The preferred basis $B$ consists of the set of $t\in \R$ such that $j(t)=1$; this is precisely the spectrum of $\mathbf V$.  We define $C$ to be the span, over $\mathbb{F}$, of $B$.  It is clear that $C$ is a filtered vector space.  

We define $B_+$, the preferred basis for $C_+$, to be the set of $t \in B$ such that $\dim (\Ker (i_{t^-,t^+}))=1$ (and so  $\codim (\Image (i_{t^-,t^+}))=0$).  Note that $B \setminus B_+$ consists of $t\in B$ such that $\codim (\Image (i_{t^-,t^+}))=1$ (and so $ \dim (\Ker (i_{t^-,t^+}))=0$).

We will next define $B_-$, the preferred basis for $C_-$, and the differential $\partial: B_+ \rightarrow B_-$.  Consider $t\in B_+$.  One can show that there exists a unique $s \in  B \setminus B_+$ satisfying the following property:  Let  $x\in V_{t^-}$ represent a non-zero element in $ \Ker (i_{t^-,t^+})$.  The element $x$ is in the image of $i_{s^+,t^-}: V_{s^+} \rightarrow V_{t^-}$ but $x$ is not in the image of $i_{s^-,t^-}: V_{s^-} \rightarrow V_{t^-}$.  We set $\partial(t) = s$ and define $B_-$ to be the set of all such $s$.

Lastly, we define $B_0$, the preferred basis for $C_0$, by $B_0 := B \setminus (B_+ \cup B_-)$ and we set $\partial = 0$ on $B_- \cup B_0$.
\end{definition} 


We will extend this definition to the non-generic persistence modules in Section \ref{sec:isometry}.

\subsection{Barcodes}\label{sec:barcodes}
We introduce in this section barcodes and the bottleneck distance.  Let us begin by introducing some preliminary notions.  A family  of intervals, $\mathfrak B$, is a list of intervals of the form $((a_j,b_j])_{j\in \{1, \ldots, n\}}$, where $-\infty \leq  a_j\leq b_j\leq + \infty$. It is convenient for us to allow trivial intervals of the form $(a,a]$ and so we do permit them in our families.  We allow a segment to appear multiple times and we identify two families which may be  obtained from each other by a permutation.  We say two families  are equivalent if 
removing all singletons $(a,a]$ from them yields the same family. For example, the families $\mathfrak B_1 = ((-2,-2], (3,3], (-1,0], (0,1])$,  $\mathfrak B_2 = ( (-1,0], (-5,-5], (0,1])$  and  $\mathfrak B_3 = ( (0,1],$ $ (-1,0])$ are equivalent.  

\begin{definition} \label{def:barcode}
A barcode $\mathbf B$ is the equivalence class of a family of intervals $\mathfrak B$.   
\end{definition}  

 Let us emphasize that in the above definition the same interval can appear multiple times  and that the order in which we list the intervals is irrelevant.   Note also that adding  or removing  empty intervals of the type $(c,c]$ to a barcode does not change it.  

A barcode represented by a family of segments where no two intervals have a common finite endpoint is called a {\bf generic} barcode. 

\medskip

Let $a\leq b, c \leq d$ be four elements of  $\R \cup \pm \infty$. We set $d((a,b],(c,d])= \max \{ \vert c-a \vert , \vert d-b \vert \}$.  Note that if $c=d = \frac{a+b}{2}$, then $d((a,b],(c,d]) = \frac{b - a}{2}$. 
\begin{definition}
Let $\mathbf B_1, \mathbf B_2$ be barcodes and take representatives $\mathfrak B_1=(I_j^1)_{ j\in A_1}$, $\mathfrak B_2=(I_k^2)_{k \in A_2}$.  The {\bf bottleneck distance} distance between $\mathbf B_1, \mathbf B_2$, denoted by  $d_{bot}(\mathbf B_1,\mathbf B_2)$,  is the infimum of the set of $ \varepsilon $ such that there is a bijection $\sigma$ between two subsets $A'_1, A'_2$ of $A_1, A_2$ with the property that  $d(I_j^1, I_{\sigma(j)}^2 ) \leq \varepsilon $ and all the remaining intervals $I_j^1, I_k^2$ for $j\in A_1\setminus A'_1, k \in A_2\setminus A'_2$, have length less than $ 2\varepsilon $.   \end{definition}  

Note that the bottleneck distance is an {\bf extended metric} as  it takes value in ${\mathbb R}_+\cup \{+\infty\}$.

\begin{remark} 
An equivalent definition, is that $d_{bot}(\mathbf B_1, \mathbf B_2) \leq \varepsilon $ if and only if there are representatives $\mathfrak B_1, \mathfrak B_2$ of $\mathbf B_1, \mathbf B_2$ and a bijection $\sigma$ between the segments of $\mathfrak B_1$ and $\mathfrak B_2$, such that $d(I_j^1,I_{\sigma(j)}^2)\leq \varepsilon $.
\end{remark}

 \begin{remark} Note that if $\mathbf{B_1, B_2}$ consist of one interval each, say $I_1 = [x_1, y_1]$ and $I_2 = [x_2, y_2]$, respectively, then $$d_{bot}(\mathbf B_1, \mathbf B_2) = \min( \max(\frac{y_1 -x_1}{2}, \frac{y_2-x_2}{2}), \max(\vert x_2 - x_1 \vert, \vert y_2 - y_1 \vert)).$$
\end{remark}

\medskip

 We denote by $\mathcal B$ the category of barcodes. We will now give a description of $Mor(\mathbf B_1, \mathbf B_2)$:  Let  $I = (a,b]$,  $J= (c,d]$ be two non-trivial  intervals, i.e.\ $a<b$ and $c<d$; we write $I \leqslant J$ if  $a \leq c$ and $b \leq d$.  First, suppose that $\mathbf{B_1, B_2}$ (have representative which) consist of one interval each, say $I_1 = (x_1, y_1]$ and $I_2 = (x_2, y_2]$, respectively.  Then, there exists a non-trivial morphism from $\mathbf B_1$ to $\mathbf{B_2}$ if and only if $I_2 \leqslant I_1$.  In other words, $Mor(\mathbf B_1, \mathbf B_2) =  \F$ if $I_2 \leqslant I_1$ and  $Mor(\mathbf B_1, \mathbf B_2) =  \{0\}$, otherwise. If either of $I_1, I_2$ is a trivial interval, then we set   $Mor(\mathbf B_1, \mathbf B_2) =  \{0\}$, even if $I_1 = I_2$.
 
  A morphism between $\mathbf B_1$ and $\mathbf B_2$ is a collection of morphisms between the intervals constituting them.  More precisely,  suppose that $\mathbf B_1=(I^1_j)_{j\in A_1},  \mathbf B_2=(I^2_j)_{j\in A_2}$.  Note that some of the intervals may be trivial.   A morphism from $\mathbf{B_1}$ to $\mathbf{B_2}$ is represented by a map $j \mapsto S_j$, from $A_1$ to the set of subsets of $A_2$, such that $I^1_j \leq  I^2_{k}$ for $k \in S_j$.   Equivalently, and more concisely, $Mor(\mathbf{B_1}, \mathbf{B_2}) = \oplus Mor(I_j^1, I_k^2)$ where $j \in A_1$ and $k\in A_2$.

%

\begin{remark} 
For a real number $a$, we define $\tau_a (\mathbf B)$ to be  the barcode obtained by shifting  all the intervals of $\mathbf B$ by $-a$, i.e.\ if $\mathbf{B} = ((x_j, y_j])_{j\in A}$, then $\tau_a \mathbf B = ((x_j - a, y_j - a])_{j\in A} $.   There exists a canonical morphism $i_a: \mathbf B \longrightarrow \tau_a \mathbf B$ which we will  describe in the the case when $\mathbf B$ consists of a single interval $(x,y]$, leaving the description in the case of more general barcodes to the reader.  First, suppose that the  interval is non-trivial.  Then, for $0 \leq  a < y-x$ we have  canonical morphism $i_a: \mathbf B \longrightarrow \tau_a \mathbf B$, corresponding to the unique morphism $(x,y] \longrightarrow (x-a,y-a]$.   For other values of $a$, or when $\mathbf{B}$ consists of a trivial interval, the morphism $i_a$ is zero as $Mor(\mathbf{B}, \mathbf{\tau_a \mathbf B}) =\{0\}$.  

Note also that given a morphism $u: \mathbf{B_1}\rightarrow \mathbf{B_2}$, there exists a corresponding morphism $\tau_a u: \tau_a \mathbf{B_1} \rightarrow \tau_a \mathbf{B_2}$.

\end{remark} 

The bottleneck distance admits the following characterization.

\begin{prop} \label{Prop-1.19}
The bottleneck distance $d_{bot}(\mathbf B_1,\mathbf B_2)$ is given by the infimum of the set of $\delta >0$ with the property that there are morphisms $u: \mathbf B_1 \longrightarrow \tau_\delta \mathbf B_2, v: \mathbf B_2 \longrightarrow \tau_\delta \mathbf B_1$ such that $\tau_\delta u\circ v =i^2_{2\delta},  \tau_\delta v\circ u =i^1_{2\delta}$. 
\end{prop} 
\begin{proof} We will only provide a sketch of the proof and leave the details to the reader.

We leave it to the reader to check that the proposition is true when $\mathbf{B_1, B_2}$ consist of single non-trivial intervals.

Suppose that $u,v, \delta$ are as in the statement of the proposition.  We will show that this implies that $d_{bot}(\mathbf B_1,\mathbf B_2)< \delta$.  We may assume that $\mathbf{B_1, B_2}$ have no intervals of length less than $2\delta$.  
Indeed, one can check that we may remove from $\mathbf{B_1, B_2}$ intervals of length less than $2\delta$ and modify $u,v$ so that we still have $\tau_\delta u\circ v =i^2_{2\delta},  \tau_\delta v\circ u =i^1_{2\delta}$ ; furthermore, this does not affect the inequality  $d_{bot}(\mathbf B_1,\mathbf B_2)< \delta$.

Now, if there are no intervals of length less than $2\delta$, then $i^1_{2\delta}, i^2_{2\delta}$ yield bijections among the bars of $\mathbf{B_1, B_2}$, respectively. Using this, one could show that the maps $u,v$ correspond to a bijection $\sigma$ between the intervals of $\mathbf{B_1}$ and the intervals of $\mathbf{B}_2$.  Furthermore, we have $d(I^1_j, I^2_{\sigma(j)})\leq \delta$.  This implies that $d_{bot}(\mathbf B_1,\mathbf B_2)< \delta$.

Next, suppose that $d_{bot}(\mathbf B_1,\mathbf B_2)< \delta$.  This means that we can find a bijection $\sigma$ between the intervals of $\mathbf{B}_1$ and $\mathbf{B_2}$ which are of length greater than two delta.  Using $\sigma$, one can easily construct $u, v$ satisfying the statement of the proposition.  
\end{proof} 


\medskip

We end this section by introducing the completion of the space of barcodes for the bottleneck distance.  Note that in Definition \ref{def:barcode} of a barcode we restricted ourselves to families of intervals consisting of a finite number of intervals.  One could also consider families consisting of infinitely many intervals and define an infinite barcode to be the equivalence class of one such family.  (The equivalence relation in consideration is, of course, the one introduced before Definition \ref{def:barcode}.)

\begin{prop}\label{prop:completion_barcodes} 
The completion of $\mathcal B$ for the bottleneck distance consists of the set of infinite barcodes $\mathbf B$ which satisfy the following finiteness property: for any positive $ \varepsilon $,  only finitely many intervals of $\B$ have length greater than $ \varepsilon $. We denote by $\overline {\mathcal B}$ the set of these barcodes. 
\end{prop} 
We leave the proof of the above proposition to the reader.  We should add that barcodes satisfying similar finiteness conditions as above have appeared in \cite{chazal2016} where they are referred to as q-tame barcodes.

\subsection{The functorial triangle}\label{sec:isometry}

We now define the following commutative triangle of functors relating the categories of persistence modules, Barannikov modules and barcodes.
$$
\xymatrix{{\mathcal P}\ar@/^0.5pc/[rrdd]^{\beta} \ar@/^0.5pc/[rrrr]^{\gamma}&&&& {\mathcal C}\ar@/^0.5pc/[llll]^{H} \\
&&&&\\
&&{\mathcal B}\ar[uurr]_{\alpha}\ar@/^0.5pc/[uull]^{Q}&&
} 
\hskip 1cm
\xymatrix{(V_s)_{s\in {\mathbb R} }\ar@/^0.5pc/[rrdd]^{\beta} \ar@/^0.5pc/[rrrr]^{\gamma}&&&& (C^s)_{s\in {\mathbb R} }\ar@/^0.5pc/[llll]^{H} \\
&&&&\\
&&{\mathbf B}\ar[uurr]_{\alpha}\ar@/^0.5pc/[uull]^{Q}&&
}
$$

The functors appearing in the above diagrams are defined as follows:

 \begin{enumerate} 
 \item The functor $H$ was defined in Section \ref{sec:barannikov}.  Recall that it associates to a Barannikov module $\mathbf{C} = (C, \partial_C)$ the persistence module $V_s= H(C^s,\partial_C)$.

\item The functor $\alpha$ associates to a barcode  $\mathbf B = \{ (a_j,b_j] \}_{j\in \{1,...,n\}}$  the simple module $C({\mathbf B})$, whose generators (\emph{i.e.}\  elements of the preferred basis) are as follows: If  $a_j < b_j< \infty$, then $a_j, b_j$ are both generators. If $b_j = \infty$, then $a_j$ is a generator, and if $a_j =b_j$, then neither one is a generator.  
We define the differential by  $\partial_C b_j=a_j, \partial_C a_j=0$. 

\item
We begin by defining the functor $\beta$ for a generic persistence module $\mathbf{V } = (V_s)_{s\in {\mathbb R} }$.  

 Denote by $B$ the spectrum of $\mathbf{V}$.  Consider the set of $t\in B $ such that $\dim (\Ker (i_{t^-,t^+}))=1$ (and so  $\codim(\Image (i_{t^-,t^+}))=0$).   Label its elements $b_1, \ldots, b_n$.  For each $b_j$, there exists a unique  $a_j\in \R$ with the following property:   Let  $x\in V_{b_j^-}$ represent a non-zero element in $ \Ker (i_{b_j^-,b_j^+})$.  The element $x$ is in the image of $i_{a_j^+,t^-}: V_{a_j^+} \rightarrow V_{b_j^-}$ but $x$ is not in the image of $i_{a_j^-,b_j^-}: V_{a_j^-} \rightarrow V_{b_j^-}$.  One can easily check that $a_j$ has the property that $\codim(\Image (i_{a_j^-,a_j^+}))= 1$ and thus each $a_j$ belongs to the spectrum of $\mathbf{V}$.  We label the remaining elements of the spectrum of $\mathbf{V}$ by $\{c_1, \ldots, c_m\}$.  The barcode $\beta(\mathbf{V})$ consists of the  list of intervals:  $\left( (a_j, b_j], (c_k, \infty] \right)$, where $1\leq j\leq n$ and  $1\leq k\leq m$.
 
Let us explain how to extend $\beta$ to the set of all persistence modules.
 One can  check that $\beta$, defined as above, commutes with the shift morphisms $\tau_a$, and as a consequence it is $1$-Lipschitz.  Since, according to Exercise \ref{exo-10},  generic persistence modules are dense for the interleaving distance, we can extend $\beta$ to $\mathcal P$.



\item We define $Q:=H\circ \alpha$.  Let us point out that for a  non-trivial interval $I= (a, b]$  we have $Q(I)_s = \F$, if $s\in I$, and $Q_s(I) = \{0\}$, if $s\notin I$.

 \item The functor $\gamma$ is defined as $\gamma:= \alpha\circ \beta $.    This definition of $\gamma$ does coincide with Definition \ref{Def-1.10} in the generic case.   
\end{enumerate} 

The  definitions presented above only describe the action of the functors at the level of objects.  We leave the description of how the functors should act on morphisms to the reader.

 
\begin{prop} \label{prop:isometry}
The functors defined above satisfy the following properties:
\begin{enumerate} 
\item $\beta \circ H\circ \alpha = \Id_{\mathcal{B}},$
\item Each of the maps $\beta, H, \alpha$ is an isometry. 
\end{enumerate} 

\end{prop} 

\begin{proof} 
To prove that $\beta \circ H\circ \alpha = \Id_{\mathcal{B}}$, it is sufficient to verify this for barcodes which consist of a single interval which can be checked quite easily.   

Now we present the proof of the second statement.  We have a functor shifting the filtration by $a$ on each of the three categories which we denote by the same notation $\tau_a$. We also have natural maps from an object to the $a$-shift of the object, which are denoted by the same notation $i_a$. In both  $\mathcal P$ and $\mathcal B$, the distance between two objects $X,Y$  is given as the smallest $\delta$ such that there are morphisms $u : X \longrightarrow Y$ and $v: Y \longrightarrow X$ such that $v\circ \tau_\delta u = i_{2\delta}, u\circ \tau_\delta v= j_{2\delta}$; note that here we are relying on Proposition \ref{Prop-1.19}.  So the fact that the functors $\beta$ and $H\circ \alpha$ are isometries follows from the fact that they commute with $\tau_a, i_a$ combined with the first part of the proposition.  One can easily verify that $\alpha: \mathcal{B}\rightarrow \mathcal{C}$ is an isometry, which in turn implies that $H$ is also an isometry.
\end{proof} 

\begin{remark} \label{rem:isometry}
\begin{enumerate} 
\item  There exists another, somewhat simpler, proof which avoids Proposition \ref{Prop-1.19}:  One could directly check that $\beta, H$, and $\alpha$ are contractions. Since the composition of these maps is the identity they must all be isometries. 
\item Since we have isometries, the completion $\overline {\mathcal B}$ yields the completion of $\mathcal P, \mathcal C$ that we denote $\overline{\mathcal P}, \overline { \mathcal C}$. 

\item It follows from the above proposition that $Q, \gamma$ are also isometries as they are compositions of isometries.  Furthermore, we also have $H\circ \gamma = \Id_{\mathcal{P}},  \gamma\circ H = \Id_{\mathcal{C}}, Q \circ \beta = \Id_{\mathcal{P}}$.  We will verify here that $H\circ \gamma = \Id_{\mathcal{P}}$ and leave the rest to the reader:  We know that $\beta\circ H\circ \alpha=\Id_{\mathcal B}$, hence composing with $\beta$ on the right, $ \beta \circ H\circ \gamma=\beta$, in other words we have $H\circ \gamma =\Id_{\mathcal C}$ provided $\beta$ is injective. But $\beta$ is an isometry, and so is obviously injective.
\end{enumerate} 
\end{remark} 

\begin{corol}[Isometry Theorem \cite{chazal2016, Cohen-Steiner2007}]
The map $\beta$ is an isometry. Thus the interleaving distance between two persistence modules coincides with the bottleneck distance of their barcodes. 
\end{corol}
\begin{corol} [Structure Theorem, \cite{zomorodian-carlsson}]  Each persistence module can be written as $\mathbf V=\bigoplus_{I\in S}  Q(I)$ where $S$ is a list of intervals. 
\end{corol} 
\begin{proof} 
Indeed, for a barcode $\mathbf B$ with representative $(I_{j})_{j \in A}$ we have
 $Q(\mathbf B)= \oplus_{j \in A} Q(I_{j})$ and $Q$ is surjective, so we get the result. 
\end{proof} 
%
%
%

\subsection{Filtered homology}

We would like to state explicitly certain consequences of Proposition \ref{prop:isometry} which will be used in the following sections of our article.

\begin{prop} \label{prop:filtered_hom}
Let $\mathbf V$ be a persistence module and $\mathbf C= \gamma (\mathbf V)$, where $\gamma$ is as defined in the previous section.  Then, for all $t\in {\mathbb R} $ we have an isomorphism  $H(C^t, \partial_B) \longrightarrow V_t$. 
\end{prop} 
\begin{proof} 
This is simply stating that $H\circ \gamma = \Id_{\mathcal{P}}$ which we proved in Remark \ref{rem:isometry}.
\end{proof}

\begin{corol}
Suppose that $f$ is a  Morse function on $M$ and denote $M^t:= \{x \in M: f(x) <t\}$.  Let $\mathbf{V} $ be the persistence module $(H(M^t))$   and let $\mathbf{C}$ be the  associated Barannikov module.  We have 
$H(C^t,\partial_C)=H(M^t)$. 
\end{corol}

\begin{remark}
 An analogous statement holds for Floer homology. This plays a crucial role in the following sections of our article.

\end{remark}
\section{Barcodes via Hamiltonian Floer theory}\label{sec:floer}
  The goal of this section is to explain how one may use Hamiltonian Floer theory to associate  barcodes to Hamiltonian diffeomorphisms.  In the case of surfaces other than the sphere, we will have maps
   $$\B_j: \Ham (\Sigma, \omega) \rightarrow \widehat{\mathcal B},$$ for each $j \in \Z$. The index $j$ corresponds to the Conley-Zehnder index and $\widehat{\mathcal B}$ denotes the space of barcode considered up to shift.  In the case of the sphere, we will have similar maps  $$\B_j: \UHam (\S^2, \omega) \rightarrow \widehat{\mathcal B},$$ which are only well-defined on the universal cover of $\Ham(\S^2, \omega)$.  Since this is a two sheeted covering, for each $\varphi \in \Ham(\S^2, \omega)$, and each $ j \in \Z$, we will have two barcodes $\B_j(\tilde \varphi_1), \B_j(\tilde \varphi_2)$.  Lastly, we will also introduce the \emph{total barcode} $\B$ which will combine all the $\B_j$'s into a single barcode.  
   
  We should point out that one can associate barcodes to Hamiltonian diffeomorphisms on symplectic manifolds more general than what we consider here; see \cite{UZ}.

  \medskip
  
  We will begin by introducing some of our conventions and recalling some basic symplectic geometry.   Suppose that $M$ is equipped with a symplectic form $\omega$.  A symplectic diffeomorphism is a diffeomorphism $\theta: M \to M$ such that $\theta^* \omega = \omega$. The set of all symplectic diffeomorphisms of $M$ is denoted by $\Symp(M, \omega)$.  Hamiltonian diffeomorphisms provide an important class of examples of symplectic diffeomorphisms. Recall that a smooth Hamiltonian $H: \S^1 \times M \to \R$ gives rise to a time-dependent vector field $X_H$ which is defined via the equation: $\omega(X_H(t), \cdot) = -dH_t$.  The Hamiltonian flow of $H$, denoted by  $\varphi^t_H$, is by definition the flow of $X_H$.  A Hamiltonian diffeomorphism is a diffeomorphism which arises as the time-one map of a Hamiltonian flow.  The set of all Hamiltonian diffeomorphisms, which is denoted by $\Ham(M, \omega)$, forms a normal subgroup of $\Symp(M, \omega)$. 

The inverse of a Hamiltonian flow  $(\varphi_H^t)^{-1}$ is itself a Hamiltonian flow whose generating Hamiltonian is given by $\bar{H}(t,x) = -H(t,\varphi^t_H(x))$.  Given two Hamiltonian flows $\varphi^t_H, \varphi^t_K$ the composition $\varphi_H^t \varphi_K^t$ is also a Hamiltonian flow, generated by $H \# K(t,x) := H(t,x)+K(t,(\varphi^t_H)^{-1}(x))$.

Given a Hamiltonian $H$, we define
$$\displaystyle \| H \| = \int_{\S^1} [ \max_{x \in M} H(t,\cdot) - \min_{x \in M} H(t, \cdot)] \; dt.$$  Recall that the Hofer distance \cite{hofer, lalonde-mcduff} between $\phi, \psi \in \Ham(M, \omega)$ is defined to be $d_{Hofer}(\phi, \psi):= \inf\{\|H - G\| : \phi = \varphi^1_H, \psi = \varphi^1_G\}.$

We will denote by $\UHam(M, \omega)$ the universal of cover of $\Ham(M, \omega)$.  An element $\tilde \varphi$ of $\UHam (M, \omega)$  consists of  the homotopy class of  a Hamiltonian path $\{\varphi^t_H\}_{0 \leq t \leq 1}$, relative to its endpoints which are $\Id$ and $\varphi^1_H$.

 \subsection{A brief review of Hamiltonian Floer theory}\label{sec:floer_review}
  
  We will now briefly recall the aspects of Hamiltonian Floer theory which will be needed to construct barcodes for Hamiltonian diffeomorphisms.   Throughout this section, unless otherwise stated, $(M, \omega)$  will denote a $2n$ dimensional, closed, connected and monotone symplectic manifold.  The latter condition means that there exists $\lambda \geq 0$ such that $\omega|_{\pi_2} = \lambda c_1|_{\pi_2}$.  All closed surfaces are examples of monotone symplectic manifolds.\footnote{Symplectic manifolds on  which $\omega|_{\pi_2} =0= c_1|_{\pi_2}$ are usually called symplectically aspherical.  All closed surfaces other than the sphere fall into this category.  Note that we are treating symplectically aspherical manifolds as monotone.}

   We denote by $\Omega_0(M)$ the space of contractible loops in $M$.  The Novikov covering of $\Omega_0(M)$ is defined by 
   $$\tilde{\Omega}_0(M) = \frac{ \{ [z,u]: z \in \Omega_0(M) , u: D^2 \rightarrow M , u|_{\partial D^2} = z \}}{[z,u] = [z', u'] \text { if } z=z' \text{ and } \omega (\bar{u} \# u') = 0},$$  
   where $\bar{u} \# u'$ is the sphere obtained by gluing $u$, with its orientation reversed, to $u'$ along their common boundary.  The disc $u$ in $[z, u]$, is referred to as the capping disc of the loop $z$.     
   
   We should point out that in the case of surfaces of non-zero genus, since $\pi_2(M) = 0$, we have $\tilde{\Omega}_0(M) = \Omega_0(M)$.  More generally, in the case of the sphere and other monotone manifolds, $\tilde{\Omega}_0(M)$ is a covering space of $\Omega_0(M)$ whose group of deck transformations is given by $\Gamma:= \frac{\pi_2(M)}{\ker(c_1)} = \frac{\pi_2(M)}{ \ker([\omega])} \simeq  \Z$.  An element   $A$ of $ \Gamma$  acts on $\tilde{\Omega}_0(M)$ by $A  [z,u] = [z, u \# A]$.
   
\medskip

\noindent \textbf{The action functional and its spectrum.}  
 Recall that the action functional $\mathcal{A}_H: \tilde{\Omega}_0(M) \rightarrow \mathbb{R}$, associated to a Hamiltonian $H$, is defined by
   $$\mathcal{A}_H([z,u]) =  \int_{0}^{1} H(t,z(t))dt \text{ }- \int_{D^2} u^*\omega.$$
Note that $\mathcal{A}_H([z,u \# A ]) = \mathcal{A}_H([z,u ]) - \omega(A),$    for every $A \in \Gamma$.

The set of critical points of $\mathcal{A}_H$, denoted by $\Crit(\mathcal{A}_H)$, consists of equivalence classes,  $[z,u] \in \tilde{\Omega}_0(M)$, such that $z$ is a $1$--periodic orbit of the Hamiltonian flow $\varphi^t_H$.  We will often refer to such  $[z,u]$ as capped $1$--periodic orbits of $\varphi^t_H$.

The action spectrum of $H$, denoted by $\Spec(H)$, is the set of critical values of $\mathcal{A}_H$; it has Lebesgue measure zero.   It turns out that the action spectrum $\Spec(H)$ is independent of $H$ in the following sense: If  $H'$ is another Hamiltonian such that $\varphi^1_H = \varphi^1_{H'}$, then  there exists a constant $C \in \R$ such that  $\Spec(H) = \Spec(H') + C,$ where $\Spec(H') + C$ is the set obtained from $\Spec(H')$ by adding the value $C$ to every element of $\Spec(H')$. It then follows that we can define the action spectrum of a Hamiltonian diffeomorphism $\phi$ by setting $\Spec(\phi) = \Spec(H)$, where $H$ is any Hamiltonian such that $\phi = \varphi^1_H$.  Of course, $\Spec(\phi)$ is well-defined up to a shift by a constant.

\medskip

\noindent \textbf{The Conley-Zehnder index.}  We say a diffeomorphism $\varphi$ is non-degenerate if its graph intersects the diagonal in $M \times M$ transversely.  A Hamiltonian $H$ is called non-degenerate if the time-1 map of its flow $\varphi^1_H$ is non-degenerate.  When $H$ is non-degenerate the set $\Crit(\A_H)$ can be indexed by the  Conley-Zehnder index, $\mu_{cz}: \Crit(\mathcal{A}_H) \rightarrow \mathbb{Z}$, which satisfies the following properties:
\begin{enumerate}
\item For every critical point $p$ of  a $C^2$-small Morse function $f$, we have $i_{Morse}(p) = \mu_{cz}([p, u_p]),$  where $u_p$ is a trivial capping disc and $i_{Morse}$ denotes the Morse index.  

  \item  For every $A \in \pi_2(M)$, we have \begin{equation} \label{eq:cZ_capping}
\mu_{cz}([z,u\#A]) = \mu_{cz}([z,u]) - 2 c_1(A),
\end{equation} where $u\#A$ denotes the capping disc obtained by attaching the sphere $A$ to the disc $u$.
\end{enumerate}

\noindent \textbf{Filtered Floer homology.}  Fix a ground field $\F$ and let $H$ be a non-degenerate Hamiltonian on $M$.  Being non-degenerate means that at each fixed point $x$ of the time--$1$ map $\varphi^1_H$, the derivative $D_x\varphi^1_H$ does not have $1$ as an eigenvalue.  

 For $t \in (-\infty, \infty] \setminus \Spec(H)$, we define $CF_m^t(H)$ to be the $\F$--vector space of formal linear combinations of the form
$$\displaystyle \alpha = \sum_{[z,u] \in \Crit(\A_H)} {\alpha_{[z,u]}[z,u]} ,$$
where $\alpha_{[z,u]}\in \F, \mu_{cz}([z,u]) = m$, and $\A_H([z,u]) < t$.  Furthermore, $\alpha$ is required to satisfy the following finiteness criterion:  for any given $c \in \R$, the number of terms such that $ \A_H([z,u]) > c$ and $\alpha_{[z,u]} \neq 0$ must be finite\footnote{Note that on surfaces of positive genus this finiteness criterion is void as the set $\Crit(A_H)$ is itself finite.}.  

The vector space $CF_*^t(H)$ is equipped with the Floer boundary map  $\partial: CF_*^t(H) \rightarrow CF_{*-1}^t(H)$ which counts $L^2$ negative gradient trajectories of $\mathcal{A}_H$:  these are maps $u(s,t): \R \times S^1 \rightarrow M$ satisfying the PDE $\partial_s u + J(t,u)(\partial_t u + X_H) = 0,$ where $J$ denotes an almost complex structure on $M$ which is compatible with $\omega$.  It is a theorem of Floer that $\partial^2 = 0$. The homology of the chain complex $(CF_*^t(H), \partial)$ is referred to as the \emph{filtered Floer homology} of $H$ and is denoted by $HF_*^t(H)$.  

More generally, filtered Floer homology groups may be defined for any  interval of the form $(a,b)$, where $-\infty \leq a,b \leq \ \infty$ are not in  $\Spec(H)$:  $HF_*^{(a,b)}(H)$ is defined to be the homology of the quotient complex $CF_*^{(a,b)}(H) = CF_*^b(H)/CF_*^a(H)$.  Let us remark that the filtered Floer homology groups do not depend on the choice of the almost complex structure $J$.

For our purposes, we will be needing the following property of filtered  Floer homology groups:  Suppose that two Hamiltonian paths $\{\varphi^t_{H_0} \}_{0 \leq t \leq 1}$ and $\{\varphi^t_{H_1} \}_{0 \leq t \leq 1}$  are homotopic relative endpoints in $\Ham(M, \omega)$, \emph{i.e.}\ they represent the same element of $\UHam (M, \omega)$.  Then, there exists a constant $c \in \R$ such that 

\begin{equation}\label{eq:homotopy_inv}
HF_*^t(H_0) = HF_*^{t+c} (H_1), \forall t \in \R.
\end{equation}

In other words, for $\tilde \varphi \in \UHam(M, \omega)$, the filtered homology groups $HF_*^t(\tilde \varphi)$ are well-defined up to a shift by a constant.
 
 We should mention that the filtered Floer homology groups are well-defined for degenerate $H$ as well: given $a, b \notin \Spec(H)$, one simply defines $HF_*^{(a,b)}(H)$ to be $HF_*^{(a,b)}(\tilde H)$ for $\tilde H$ which is non-degenerate and sufficiently $C^2$ close to $H$.  The definition does not depend on the choice of $\tilde H$.
 
Note that the action of $\Gamma$ on  $\tilde \Omega_0(M)$, which we described above, induces natural isomorphisms on Floer homology:  Given $A \in \Gamma$, the action $[z, u] \mapsto [z, u \# A]$ induces an isomorphism of filtered chain complexes $CF_*^t(H)$ and $CF_{*-2 c_1(A)}^{t- \omega(A)}(H).$    

\subsection{Barcodes for Hamiltonians}\label{sec:barcodes_for_hamiltonians}
Fix an integer $j$ and let $H$ be a strongly non-degenerate Hamiltonian.  Being strongly non-degenerate means that $H$ is  non-degenerate and no two if its capped 1--periodic orbits have the same action.  As explained in \cite{pol-shel-stojis}, the family of vector spaces $HF_j^s(H)$ forms a persistence module where the morphisms  $i_{s,t}: HF_j^s(H) \rightarrow HF_j^t(H)$ of Definition \ref{def:per-mod} are induced by the inclusion of $CF_j^s(H) \hookrightarrow CF_j^t(H)$.

Applying the map $\gamma$ of Section \ref{sec:barannikov} to this persistence module, we obtain a corresponding Barannikov complex which we will denote by $(BC^s_j(H), \partial_B)$.  Let us emphasize that the generators of this complex are given by \emph{actions}, $\A_H([z,u)]$, of capped 1--periodic orbits $[z,u] \in \Crit(\A_H)$ whose Conley-Zehnder indices are either $j$ or $j+1$.  The boundary map, which is given by the construction carried out in Section \ref{sec:barannikov}, sends the action of an orbit of index $j+1$ to the action of an orbit of index $j$ and it is zero on  actions  orbits of index $j$.

Next, applying the map $\beta$ of Section \ref{sec:isometry} to the persistence module $HF_j^s(H)$, we obtain a corresponding barcode which we will denote by $\B_j(H)$.   The lower ends of the bars in $\B_j(H)$ are actions of orbits of index $j$ while the upper ends of the bars are actions of orbits of index $j+1$.
It is easy to see that, since $H$ is strongly non-degenerate, the action of an orbit appears as one endpoint of exactly one bar in exactly one of the barcodes $\B_j(H)$. 

It can be shown, via standard Floer theoretic arguments (see e.g.  Equation (4) in \cite{PolShel}), that barcodes are 1-lipschitz with respect to Hofer distance, namely
\begin{equation}\label{eq:continuity}
\dbot(\B_j(H), \B_j(G)) \leq \| H - G\|.
\end{equation}

The above inequality allows us to define the barcode of any smooth, or even continuous, function $H:\S^1 \times M \rightarrow \R$.  Indeed, for an arbitrary smooth, or continuous, $H$ take $H_i$ to be a sequence of  non-degenerate Hamiltonians such that $\|H - H_i\| \to 0 $, and define $\B_j(H) := \lim \B_j(H_i)$ where the limit is taken with respect to the bottleneck distance. Note that $\B_j(H)$ belongs to the completion of the space of barcodes $\overline {\mathcal B}$ which we described in Section \ref{sec:barcodes}; see Proposition \ref{prop:completion_barcodes}.  Hence, for each integer $j$, we obtain  a map $$\B_j: C^{\infty}(\S^1 \times M) \rightarrow \overline {\mathcal B}$$ which continues to satisfy  Equation \eqref{eq:continuity}.  The rest of this section is dedicated to describing  some of the properties of the barcodes  $\B_j(H)$.

\medskip

\noindent \textbf{Spectrality:}  If $H$ is non-degenerate, the set of endpoints of $\B_j(H)$ forms a subset of $\Spec(H)$.  In fact, this statement continues to hold for any $H \in C^{\infty}(\S^1 \times M)$.  This  can  be proven  by writing $H$ as the limit, in $C^2$ topology, of a sequence of non-degenerate Hamiltonians $H_i$ and applying the continuity and spectrality properties to the $H_i$'s.

\medskip

\noindent \textbf{Symplectic Invariance:} For any $\psi \in \Symp(M, \omega)$ and any $H \in C^{\infty}(\S^1 \times M)$, we have
\begin{equation} \label{eq:conj_inv1}
\B_j(H \circ \psi) = \B_j(H).
\end{equation}
The above follows from the fact that, for non-degenerate $H$, the filtered Floer complexes $CF_*^t(H)$ and $CF_*^t(H \circ \psi)$ are isomorphic.  See, for example, \cite{PolShel} for further details.

\medskip 

\noindent  \textbf{Periodicity:}  Let $N, \tau$ denote the minimal non-negative generators of $c_1(\\ \pi_2(M))$, $\omega (\pi_2(M))$, respectively. 
\begin{equation} \label{eq:periodicity}
\B_j(H ) = \B_{j-2N}(H)- \tau,
\end{equation}
where $\B_{j-2N}(H)- \tau $ is the barcode obtained from  $\B_{j-2N}(H)$ by shifting each of its bars by  $-\tau$.   The above follows immediately from the last paragraph of Section \ref{sec:floer_review}.  Let us point out that in the case of surfaces this property is of interest only for $S^2$ as  $N = \tau = 0$ for other  surfaces.  In the case of the sphere, $N =2$ and the symplectic form may be normalized to impose $\tau = 1$.

\subsubsection{The total barcode of a Hamiltonian}
Let $H$ denote a non-degenerate Hamiltonian and define the \emph{total barcode} of $H$ to be $$\B(H) :=  \sqcup_j \B_j(H).$$  
This is not a standard barcode in the sense of Definition \ref{def:barcode} because it has infinitely many bars. In fact, we can see from the periodicity property that if $I$ is a bar in $\B(H)$, then so are  the shifted bars $I + n \tau$ for every integer $n$.  Nevertheless, we can still work with these barcodes:  The definition of the bottleneck distance easily extends to this class of barcodes.  Note that $\dbot(\B(H), \B(G)) \leq \sup_j \{ \dbot(\B_j(H), \B_j(G))\}$.   One can easily check that the total barcode satisfies Inequality \eqref{eq:continuity} and so it can be defined for any $H$.  Moreover, $\B(H)$ satisfies the spectrality and symplectic invariance properties from the previous section.   Lastly, observe that the periodicity property translates to the following shift invariance property:
\begin{equation}\label{eq:shift_inv}
\B(H) = \B(H) + \tau,
\end{equation}
where $\tau$ is the non-negative generator of $\omega(\pi_2(M))$.

We should emphasize that in the case of aspherical manifolds, such as surfaces of positive genus, $\B_j(H)$ is non-trivial for only finitely many values of $j$.  Note also that, on aspherical manifolds,  the family of vector spaces $\oplus_j HF_{j}^t(H)$ is a persistence module and so as in the previous section we can associate a barcode to this persistence module.  This barcode will be exactly the total barcode $\B(H)$.

\medskip

The following  property of the total barcode will be used in Section \ref{sec:proof_invariance_Lef}.

\begin{prop}\label{prop:endpoints_in_interval}
Let $(a,b)$ denote an interval whose endpoints are not in the spectrum of $H$.  The number of bars of $\B(H)$ which have exactly one endpoint in the interval $(a,b)$ is given by the rank of $HF_*^{(a,b)}(H)$.   

In particular, if $c\in \Spec(H)$ is isolated, then the total number of bars with one end-point at value $c$ is given by the rank of $HF_*^{(c -\varepsilon,c + \varepsilon)}(H)$ for sufficiently small $\varepsilon > 0$.
\end{prop}
\begin{proof}
Observe that a sufficiently $C^2$--small perturbation of $H$ does not change the total number of bars which have exactly one endpoint in the interval $(a,b)$:  Indeed,  such perturbation might create new bars but both endpoints of these bars will be either inside $(a,b)$ or outside of it.  On the other hand, $HF_*^{(a,b)}(H)$ is by definition $HF_*^{(a,b)}(\tilde H)$ where $\tilde H$ is sufficiently $C^2$ close to $H$.  We conclude that we may suppose that $H$ is non-degenerate and that no two of its  periodic orbits have the same action. 

First, suppose that $\Sigma \neq \S^2$.  In that case, the barcode $\B(H)$ is a standard (finite) barcode.  We leave it to the reader to check that in this case the result follows from Proposition  \ref{prop:filtered_hom}.  

Now, suppose that $\Sigma = \S^2$.  Since $\B(H)$ is not a standard barcode, in the sense of Definition \ref{def:barcode}, we cannot immediately apply  Proposition \ref{prop:isometry}.\footnote{Although the proposition does not apply directly, 
one can check that its proof may be modified to encompass more general barcodes such as the one being considered here.} We will reduce to the case of finite barcodes by noting that the set of periodic orbits of $H$ whose actions are in $(a, b)$ is finite.  Let $N,M$ denote the minimal, and respectively maximal, Conley-Zehnder indices of periodic orbits whose actions are in the interval $(a,b)$.  Define the barcode $\B'(H) = \sqcup \B_j(H)$ where $N-1 \leq j \leq M$; recall that the action of a periodic orbit of index $j$ is the endpoint of a bar in either $B_j(H)$ or $B_{j-1}(H)$. We leave it to the reader check that any bar in $\B(H)$ which has an endpoint in $(a,b)$ appears in $\B'(H)$.   

The barcode $\B'(H)$  is the barcode of the persistence module $V_t = \oplus_j HF_j^t \\(H)$, where $N-1 \leq j \leq M$ .  It follows, once again from Proposition \ref{prop:isometry}, that the number of bars in $\B'(H)$ with exactly one endpoint in $(a,b)$ coincides with the rank of $\oplus_j HF_j^{(a,b)}(H)$, where $N-1 \leq j \leq M$.  But the latter is exactly the rank of $HF_*^{(a,b)}(H)$.  This completes the proof.
\end{proof}

Let us point out a consequence of Equation \eqref{eq:shift_inv}   which will play an important role in our story: Given a barcode $\B$ and an interval $I \subset \R$, denote by $\# \mathrm{Endpoints}(\B)\cap I$ the total number of endpoints of bars of $\B$ which are in the interval $I$.  We should emphasize that, here, we count endpoints with their multiplicities, \emph{i.e.}\ if the same value appears as an endpoint of $k$ different bars, then it is counted $k$ times.  

\begin{corol}\label{corol:endpoints_shift_inv}
Suppose that $\B$ is a barcode which is invariant under shift by $\tau$, i.e.\  $\B + \tau = \B$.  Then, for any $c\in \R$ we have  $$\# \mathrm{Endpoints}(\B) \cap [0, \tau) = \# \mathrm{Endpoints}(\B + c) \cap [0, \tau).$$ 
\end{corol}

\subsection{Barcodes for Hamiltonian diffeomorphisms}\label{sec:barcodes_for_ham_diffeos}
Recall that given a barcode $\B  \in \overline {\mathcal B} $ and $c\in \mathbb{R}$ , we have defined $\B + c$ to be the barcode obtained from $\B$ by shifting each of its intervals by $c$.  Let $\sim$ denote the equivalence relation on $\overline {\mathcal B}$  given by  $\B_1 \sim \B_2$  if  $\B_2 = \B_1 +c$ for some $c \in \mathbb{R}$; we will denote the quotient space by $\widehat{\mathcal B}$.  The bottleneck distance defines a distance on $\widehat{\mathcal B}$ which we will continue to denote by $\dbot$. 

It is an immediate consequence of Equation \eqref{eq:homotopy_inv} that if  two Hamiltonian paths $\{\varphi^t_{H_0} \}_{0 \leq t \leq 1}$ and $\{\varphi^t_{H_1} \}_{0 \leq t \leq 1}$  are homotopic relative to endpoints in $\Ham(M, \omega)$, and so they represent the same element of $\UHam (M, \omega)$, then $\B_j(H_0) = \B_j(H_1)$ in $\widehat{\mathcal B}$. Hence, we obtain a map, which we will continue to denote by the same symbol, $$\B_j: \UHam (M, \omega) \rightarrow \widehat{\mathcal B}.$$

The barcode $\B_j(\tilde \varphi)$, where $\tilde \varphi \in \UHam (M, \omega)$, inherits appropriately restated versions of the properties listed in the previous section.  
\medskip
Of course, we may also define the total barcode  
$$\B: \UHam(M, \omega) \rightarrow \widehat{\mathcal{B}}.$$  Indeed, $\B(\tilde \varphi)$ is simply $\B(H)$ considered as a barcode up to shift, where $H$ is any Hamiltonian whose flow represents $\tilde \varphi$.

\begin{remark}\label{rem:not_shifted}
Alternatively to our approach in this article, one could define $\B_j(\tilde \varphi), \B(\tilde \varphi)$ to be $\B_j(H), \B(H)$ where $H$ is a \emph{mean-normalized} Hamiltonian whose flow is a representative of $\tilde \varphi$.  Being mean-normalized means $\int_0^1 \int_M H \omega^n = 0$.  This defines $\B_j(\tilde \varphi), \B(\tilde \varphi)$ without any ambiguity  as a barcode, as opposed to a barcode up to shift, and so one obtains  maps $\B_j, \B: \UHam (M, \omega) \rightarrow \bar{\mathcal B}.$  This is the  manner in which barcodes are defined in \cite{PolShel}, and in fact, it is a more natural approach from the point of view of Hofer geometry.  However, this approach is not suitable for our purposes as it yields barcodes which are not continuous in the uniform topology; see Remark \ref{rem:cont_shift_necessary}.
\end{remark}

\medskip  
As in the previous section, denote by $\tau$ the minimal positive generator of $\omega(\pi_2(M))$.  
\begin{definition}
For any $\tilde \varphi \in \UHam(M, \omega)$, we define $$\# \mathrm{Endpoints}(\B(\tilde \varphi)) \cap [0, \tau) = \# \mathrm{Endpoints}(\B(H)) \cap [0, \tau),$$ where $H$ is any Hamiltonian whose flow represents $\tilde \varphi$.
\end{definition}

The above definition is well-defined as a consequence of Corollary \ref{corol:endpoints_shift_inv}.\footnote{The quantity$ \# \mathrm{Endpoints}(\B(\tilde \varphi)) \cap [0, \tau)$ is, in fact, well-defined for $\varphi \in \Ham(M, \omega)$, i.e. $ \# \mathrm{Endpoints}(\B(\tilde \varphi_1)) \cap [0, \tau) =  \# \mathrm{Endpoints}(\B(\tilde \varphi_2)) \cap [0, \tau)$ for any two lifts  $\tilde \varphi_1,  \tilde \varphi_2$  of $\varphi$ to $\UHam(M, \omega)$.  This fact, which will not be used in our paper, is a consequence of Proposition \ref{prop:rank_endpoints} and Remark \ref{rem:rank_endpoints_general}.}

\subsubsection{Barcodes on surfaces of positive genus }  \label{sec:barcode_ham_diffeos_surfaces_genus}
Let us first consider the case of a closed surface $\Sigma$ which is of positive genus.   It is well-known that, in this case, the fundamental group of $\Ham(\Sigma, \omega)$ is trivial. Therefore, $\UHam (\Sigma, \omega) = \Ham(\Sigma, \omega)$ and hence, $\B_j: \Ham (\Sigma, \omega) \rightarrow \widehat{\mathcal B}$ is defined directly for Hamiltonian diffeomorphisms.   

As an immediate consequence of Equation \eqref{eq:continuity} we wee that 
\begin{equation}\label{eq:hofer_cont}
d_{bot}(\B_j(\phi), \B_j(\psi)) \leq d_{Hofer}(\phi, \psi).
\end{equation}

The symplectic invariance property described  in Section \ref{sec:barcodes_for_hamiltonians} translates to the following conjugacy invariance property: For any $\psi \in \Symp(\Sigma, \omega)$ and any $\phi \in \Ham(M, \omega)$, 

\begin{equation}\label{eq:conj_inv2}
\B_j(\psi^{-1} \phi \psi) = \B_j(\phi).
\end{equation}
This follows from Equation \eqref{eq:conj_inv1} and the fact that  $\varphi^t_{H\circ 
\psi} = \psi^{-1} \varphi^t_H \psi$ for any Hamiltonian $H$.  

\medskip 
Clearly,  the total barcode is also well-defined  directly for Hamiltonian diffeomorphisms.  Hence, we have a map $\B: \Ham (\Sigma, \omega) \rightarrow \widehat{\mathcal B}$ which also satisfies Equations \eqref{eq:hofer_cont} and  \eqref{eq:conj_inv2}.

\subsubsection{Barcodes on the sphere}
  In the case of the sphere, the maps $\B_j : \UHam(\S^2, \omega) \rightarrow \widehat{\mathcal B}$ do not descend to $\Ham(\S^2, \omega)$.  This is because the fundamental group of $\Ham(\S^2, \omega)$ is non-trivial.  In fact, it is $\Z / 2\Z$ and is generated by a full rotation around the North-South axis of the sphere; see for example \cite{pol_book}.  We will denote this rotation by $\mathrm{Rot}$.  Hence, we see that every element $\varphi \in \Ham(\S^2, \omega)$ has two lifts $\tilde \varphi_1, \tilde \varphi_2  \in \UHam(\S^2, \omega)$ where $\tilde \varphi_2 = \mathrm{Rot} \circ \tilde \varphi_1.$  Therefore, for each $\varphi \in \Ham(\S^2, \omega)$ we have  barcodes $\B_j(\tilde{\varphi_1}), \B_j(\tilde{\varphi_2})$ as well as the total barcodes $\B(\tilde{\varphi_1}), \B(\tilde{\varphi_2})$.  These maps satisfy appropriately restated versions of Equations \eqref{eq:hofer_cont} \& \eqref{eq:conj_inv2}. 
  
 Lastly, from Equation \eqref{eq:periodicity} we obtain the following property:  
  \begin{equation}\label{eq:periodicity_sphere}
  \B_j(\tilde \varphi) = \B_{j-4}(\tilde \varphi).
  \end{equation}
  Note that the shift by $\tau$ disappears as we're considering barcodes up to shift.

\section{Continuity of barcodes on surfaces}\label{sec:barcodes_surfaces}
In this section we prove our main results on continuity of barcodes and their extension to Hamiltonian homeomorphisms.  The results of this section allows to effectively define Hamiltonian Floer homology for Hamiltonian homeomorphisms of surfaces.

Before giving precise statements of these results we will introduce some of our conventions.  Let $d$ be a Riemannian distance on a closed manifold $M$. Given two maps $\phi, \psi \co M \to M,$ we   denote
$$d_{C^0}(\phi,\psi)=\sup_{x\in M}d(\phi(x),\psi(x)).$$
The $C^0$ topology on $\Ham(M, \omega)$ is the topology induced by $d_{C^0}$.  We will denote by $\overline{\Ham}(M, \omega)$ the $C^0$ closure of $\Ham(M, \omega)$ taken inside the group of homeomorphisms of $M$.  In the case where $M$ is a surface, $\overline{\Ham}(M, \omega)$ coincides with the set of  area-preserving homeomorphisms of $M$ with vanishing flux or, equivalently, mean rotation vector; for a proof of this fact see \cite{fathi80}. We will refer to elements of $\overline{\Ham}(M, \omega)$ as Hamiltonian homeomorphisms.

\subsection*{Surfaces of positive genus}
  Here is our main result concerning continuity of barcodes on surfaces other than the sphere.  
\begin{theo}\label{theo:cont_barcodes_genus}
 Let $(\Sigma, \omega)$ denote a closed symplectic surface other than the sphere.  For each integer $j$, the mapping 
 $$\B_j: (\Ham(\Sigma, \omega), d_{C^0}) \rightarrow (\widehat{\mathcal B}, \dbot)$$
 is continuous. Furthermore, $\B_j$ extends continuously to $\overline{\Ham}(\Sigma, \omega)$.
\end{theo}

\begin{remark}\label{rem:cont_total_barcode}
The total barcode map $\B: (\Ham(\Sigma, \omega), d_{C^0}) \rightarrow (\widehat{\mathcal B}, \dbot)$ is also continuous and extends continuously to $\overline{\Ham}(\Sigma, \omega)$.  Note that this is not an immediate consequence of the previous theorem.  The proof of this fact is identical to the proof of the above theorem and so we will omit it. 
\end{remark}

\begin{remark}\label{rem:cont_shift_necessary}
In our approach, the barcodes $\B_j(\varphi), \B(\varphi)$ are well-defined up to a shift by a constant.  As mentioned in Remark  \ref{rem:not_shifted}, one can remove this ambiguity by working with mean-normalized Hamiltonians.  This would yield a map $\B_j: \Ham(\Sigma, \omega) \rightarrow\overline{\mathcal B}$.

However, for the above theorem, it is absolutely crucial to consider barcodes up to shift, \emph{i.e.}\ the space $\widehat{\mathcal B}$.  Indeed, the maps $\B_j, \B: (\Ham(\Sigma, \omega), d_{C^0}) \rightarrow (\overline{\mathcal B}, \dbot)$ is not continuous.  This remark also applies to Theorem \ref{theo:cont_barcodes_sphere}.
\end{remark}

\subsection*{The sphere}
In the case of the sphere $\Ham(\S^2, \omega)$ and $\overline{\Ham}(\S^2, \omega)$ are the groups of area-preserving diffeomorphisms and homeomorphisms of $\S^2$, respectively.  The fundamental group of each of these groups is $\Z/2\Z$ and it is generated by the full rotation around the North-South axis of the sphere.   We will denote by $\UHam(\S^2, \omega)$ and $\overline{\UHam}(\S^2, \omega)$  the universal covers of these groups.  Being covering spaces, they naturally inherit the $C^0$ topology from $\Ham(\S^2, \omega)$ and $\overline{\Ham}(\S^2, \omega)$.  As we will explain now, the $C^0$ topology on $\UHam(\S^2, \omega)$ and $\overline{\UHam}(\S^2, \omega)$ may equivalently be defined via a natural lift of $d_{C^0}$:  Take $\tilde \varphi_1, \tilde \varphi_2 \in \UHam(\S^2, \omega)$ and define $$\tilde{d}_{C^0}(\tilde \varphi_1, \tilde \varphi_2) : = \inf\{ { \max_{0\leq t \leq 1}} d_{C^0}(\varphi_1^t, \varphi_2^t)\},$$
where the infimum is taken over all paths $\{\varphi_1^t\}_{0\leq t \leq 1},  \{\varphi_2^t\}_{0\leq t \leq 1}$ which represent $\tilde \varphi_1, \tilde \varphi_2$ in $\UHam(\S^2, \omega)$. Of course, this definition extends $\overline{\UHam}(\S^2, \omega)$.   Here is our main result concerning continuity of barcodes on the sphere.
 
\begin{theo}\label{theo:cont_barcodes_sphere}
  For each integer $j$, the mapping 
 $$\B_j: (\UHam(\S^2, \omega), \tilde{d}_{C^0}) \rightarrow (\widehat{\mathcal B}, \dbot)$$
 is continuous. Furthermore, $\B_j$ extends continuously to $\overline{\UHam}(\S^2, \omega)$.
\end{theo}



\begin{remark}  Once again, the total barcode $\B$ is also continuous and extends continuously to $\overline{\UHam}(\S^2, \omega)$.   \end{remark}

\medskip

The rest of this section is dedicated to the proofs of Theorems \ref{theo:cont_barcodes_genus} \& \ref{theo:cont_barcodes_sphere}.  Since the proofs of the two theorems are quite similar, we will  give a detailed proof of Theorem \ref{theo:cont_barcodes_genus} and will only indicate what must be modified to obtain the proof of Theorem \ref{theo:cont_barcodes_sphere}.

\subsection{Proofs of Theorems \ref{theo:cont_barcodes_genus} \& \ref{theo:cont_barcodes_sphere}} \label{sec:proof_cont_genus}
We begin with the proof of Theorem \ref{theo:cont_barcodes_genus}.  It is a consequence of the following two theorems.

%

\begin{theo}\label{theo:at_identity}
   For any $\varepsilon > 0$, there exists $\delta > 0$ with the following property: If $\psi \in \Ham(\Sigma, \omega)$ and $d_{C^0}(\psi, \Id) < \delta$, then $$\dbot(\B_j(\psi), \B_j(\Id)) < \varepsilon,$$
  for any $j \in \Z$.
\end{theo}

\begin{theo}\label{theo:not_identity}
Suppose that $\Id \neq \eta \in \overline{Ham}(\Sigma, \omega)$.  For any $\varepsilon > 0$, there exists $\delta > 0$, depending on $\eta$, with the following property: if $\phi, \psi \in \Ham(\Sigma, \omega)$,  $d_{C^0}(\phi, \eta) < \delta$, and $d_{C^0}(\psi, Id) < \delta$, then $$\dbot(\B_j(\phi \psi), \B_j(\phi)) < \varepsilon,$$
for any $j \in \Z$.
\end{theo}
The above two theorems are proven below in Sections  \ref{sec:proof_at_identity} and \ref{sec:proof_not_deintity}.  Let us explain why Theorem \ref{theo:cont_barcodes_genus} follows from the above two results.  Indeed, it is not difficult to see that Theorem \ref{theo:at_identity} proves the continuity of $\B_j: (\Ham(\Sigma, \omega), d_{C^0}) \rightarrow (\widehat{\mathcal{B}}, \dbot)$ at the identity and Theorem \ref{theo:not_identity} proves its continuity at every other Hamiltonian diffeomorphism. Furthermore, Theorem \ref{theo:not_identity} implies that the map $\B_j$ extends continuously to $\overline{\Ham}(\Sigma, \omega)$.  To see this take $\eta \in \overline{\Ham}(\Sigma, \omega)$ and let $\phi_i$ be a sequence in $\Ham(\Sigma, \omega)$ converging uniformly to $\eta$.  For any $\varepsilon >0$, there exists a positive integer $N$ such that if $N \leq i, k$, then $d_{C^0}(\phi_i, \eta) < \delta$ and $d_{C^0}( \phi_i^{-1}\phi_k, Id) < \delta$.  It then follows from Theorem \ref{theo:not_identity} that $\dbot(\B_j(\phi_i), \B_j(\phi_k)) < \varepsilon$.  This, of course, implies that the sequence of barcodes $\B_j(\phi_i)$ is convergent in $\widehat{\mathcal B}$, and hence, we can define $\B_j(\eta)$ to be the limit of this sequence.  We leave it to the reader to check that $\B_j(\eta)$ does not depend on the choice of the approximating sequence and that the extension of $\B_j$ to $\overline{\Ham}(\Sigma, \omega)$ is continuous.

\bigskip
The proof of Theorem \ref{theo:cont_barcodes_sphere} is very  similar to that of Theorem \ref{theo:cont_barcodes_genus}.   We first need to introduce a bit of  notation.  Recall that the identity in $\Ham(\S^2, \omega)$ has two lifts to $\UHam(\S^2, \omega)$:  One represented by the homotopy class of the constant path based at the identity, which we will continue to denote by $\Id$, and a second represented by the homotopy class of the non-trivial loop in $\pi_1(\Ham(\S^2, \omega))$, which we will denote by $\mathrm{Rot}$.   One must begin by  proving the following two results which are analogous to Theorems \ref{theo:at_identity} \& \ref{theo:not_identity}.  

\begin{theo}\label{theo:at_identity_sphere}
 Let $\tilde \psi$ denote an element of $\UHam(\S^2, \omega)$.  For any $\varepsilon > 0$, there exists $\delta > 0$ such that 
  \begin{enumerate}

 \item If  $\tilde d_{C^0}(\tilde \psi, \Id) < \delta$, then $\dbot(\B_j(\tilde \psi), \B_j(\Id)) < \varepsilon,$
  for any $j \in \Z$.
  \item If $\tilde d_{C^0}(\tilde \psi, \mathrm{Rot}) < \delta$, then $\dbot(\B_j(\tilde \psi), \B_j(\mathrm{Rot})) < \varepsilon,$
  for any $j \in \Z$.
  
  \end{enumerate}
\end{theo}

\begin{theo}\label{theo:not_identity_sphere}
Let  $\tilde \eta \in \overline{\UHam}(\S^2, \omega)\setminus \{\Id, \mathrm{Rot}\}$.  For any $\varepsilon > 0$, there exists $\delta > 0$, depending on $\eta$, with the following property: if $\tilde \phi, \tilde \psi \in \UHam(\S^2, \omega)$,  $\tilde d_{C^0}(\tilde \phi, \tilde \eta) < \delta$ and $\tilde d_{C^0}(\tilde \psi, Id) < \delta$, then $$\dbot(\B_j(\tilde \phi  \tilde \psi), \B_j(\tilde \phi)) < \varepsilon,$$
for any $j \in \Z$.
\end{theo}

The reasoning as to why Theorem \ref{theo:cont_barcodes_sphere} follows from the above two theorems is almost identical to the one given above for why Theorem \ref{theo:cont_barcodes_genus} follows from Theorems \ref{theo:at_identity} \& \ref{theo:not_identity}.   Indeed, Theorem \ref{theo:at_identity_sphere} proves continuity of $\B_j$ at $\Id, \mathrm{Rot}$ while Theorem \ref{theo:not_identity_sphere} proves continuity of $\B_j$ elsewhere and furthermore allows us to extend $\B_j$ continuously to $\overline{\UHam}(\S^2, \omega)$.

Proofs of the above theorems are, once again, very similar to the proofs of Theorems \ref{theo:at_identity} \& \ref{theo:not_identity}.   The only difference here is that instead of working with diffeomorphisms (or homeomorphisms) we must work with paths of diffeomorphisms (or homeomorphisms) which are based at the identity.  This is, of course, necessary as we have to deal with elements $\UHam(S^2, \omega)$.   We will not provide proofs for Theorems \ref{theo:at_identity_sphere} \& \ref{theo:not_identity_sphere} as  one could indeed prove them by making simple modifications to the proofs of Theorems \ref{theo:at_identity} \& \ref{theo:not_identity}. 

\subsection{Preliminary results on continuity of barcodes}
  The results of this subsection, which will be needed for the proof of the theorems of the previous section, hold on any symplectic manifold which is monotone, \emph{i.e.}\ not necessarily surfaces.  We begin with the following definition.
  
  \begin{definition}\label{def:epsilon-shift}
  Let $\phi$ be a homeomorphism of $M$ and $\delta$ a positive real number.  We will say that a set $U$ is \emph{$\delta$--shifted} by $\phi$ if $d(x, \phi(x)) > \delta$ for each $x \in U$.
\end{definition}

In the following proposition $V$ denotes a proper closed subset of $M$.  We pick a Hamiltonian $K: \S^1 \times M \rightarrow \R$ such that $D$ is $\delta$--shifted, for some $\delta> 0$, by $\varphi^1_K$.  Recall that $K \# H(t,x) := K(t,x)+H(t,(\varphi^t_K)^{-1}(x))$ is the Hamiltonian whose flow is the composition $\varphi^t_K  \varphi^t_H$.

\begin{prop}\label{prop:epsilon_shift}
Suppose that  $ H : \S^1 \times M \rightarrow \R$ is supported in $V$ and that  $d_{C^0}(Id, \varphi^t_H) < \delta$ for each $t \in [0,1]$.  Then,  for each integer $j$, $\B_j(K \# H) = \B_j(K)$.
\end{prop}
\begin{proof}
  We will only give a sketch of the proof of the above statement as the the proof is similar to the proof of the results contained in \cite{Sey12}.
  
  For each $s \in [0,1]$ consider the Hamiltonian $F_s(t,x) = K(t,x) + sH(st, \\ (\varphi^t_K)^{-1}(x))$ and note that $ \varphi^1_{F_s} = \varphi^1_K \circ \varphi^s_H$.   We leave it to the reader to check that the assumption that $\varphi^1_K$  $\delta$--shifts the support of $H$ and that $d_{C^0}(Id, \varphi^s_H) < \delta$  implies that the Hamiltonian diffeomorphisms $\varphi^1_K$ and $\varphi^1_{F_s} = \varphi^1_K \circ \varphi^s_H$ have the exact same fixed points; see for example the proof of Theorem 4 of \cite{Sey12}. Furthermore, the corresponding Hamiltonians $K$ and $F_s$ have the exact same action spectrum , \emph{i.e.}\ $\Spec(K) = \Spec(F_s)$.  Although  this last claim is not directly stated in \cite{Sey12}, it is contained therein within the proof of Proposition 2.2 and hence we will not give a proof of this fact either.
  
  We will now prove that $\B_j(F_s) = \B_j(F_0)= \B_j(K)$.  Since $\Spec(F_s) = \Spec(K)$, by the spectrality property of barcodes we conclude that  the endpoints of the bars of $\B_j(F_s)$ are all contained in $\Spec(K)$   which is a measure-zero subset of $\R$.  Now, the barcode $\B_j(F_s)$ varies continuously with $s$, by the continuity property of barcodes, and so the end points of $\B_j(F_s)$ vary continuously in the measure-zero set $\Spec(K)$.   Of course, a continuous function taking values in a measure-zero set is necessarily constant.  Hence, $\B_j(F_s) = \B_j(F_0) = \B_j(K)$.   Noting that $F_1 = K \#H$ completes the proof. 
\end{proof}

We will be needing the following proposition as well.  Results similar to the following proposition are ubiquitous within the literature on the theory of spectral invariants; see for example \cite{viterbo, ostrover, Oh05b, usher}.

\begin{prop}\label{prop:displacement}
Suppose that $K, H$ are two Hamiltonians such that $\varphi^1_K$ displaces the support of $H$.  Then, $\B_j (K \#H) = \B_j(K)$, for any integer $j$.   \end{prop}
\begin{proof}
The proof of this proposition is very similar to that of Proposition \ref{prop:epsilon_shift}.    For each $s \in [0,1]$ consider the Hamiltonian $F_s(t,x) = K(t,x) + sH(st, (\varphi^t_K)^{-1}(x))$.

 As in Proposition \ref{prop:epsilon_shift}, using the fact that $\varphi^1_{F_s} = \varphi^1_K \varphi^s_H$ for each $s \in [0,1]$, one can easily show that $\Spec(K) = \Spec(F_s)$ for each $s \in [0,1]$.  Repeating the argument from the last paragraph of the proof of Proposition \ref{prop:epsilon_shift} would lead to the conclusion that $\B_j(F_s) = \B(F_0) = \B(K)$ for each $s \in [0,1]$.  Since  $F_1 = K \# H$, we conclude that $ \B_j (K\# H)= \B_j(K)$ which completes the proof.
\end{proof}

\subsection{Proof of Theorem \ref{theo:at_identity}}\label{sec:proof_at_identity}
Our proof of Theorem \ref{theo:at_identity} will rely on the following fragmentation result. 
\begin{claim}\label{claim:fragmentation}
There exists a covering of $\Sigma$ by two proper open subsets $V_1, V_2$ with the following property:
 For every $r >0$, there exists a neighborhood $\nu$ of the identity in $\Ham(\Sigma, \omega)$ such that if $\psi \in \nu$, then  $\psi$ can be written as a composition $\psi = \psi_1 \psi_2$, where  $\psi_1$ is supported in $V_1$,  $\psi_2$ is supported in $V_2$, and $d_{C^0}(\psi_k, \Id) < r$.  
 
 Furthermore, one can find Hamiltonians $H_1, H_2$ such that $H_k$ is supported in $V_k$,   $\psi_k = \varphi^1_{H_k}$, and $d_{C^0}(\varphi^t_{H_k}, \Id) < r$ for all $t \in [0,1]$.
\end{claim}

We will postpone the proof of the above claim and present a proof of Theorem \ref{theo:at_identity}.  Let $\varepsilon>0$.
For $k=1, 2$, we pick a $C^2$-small Morse function $f_k$  satisfying the following properties: $f_k$ has no critical point which is in the closure of $V_k$,  the critical points of $f_k$ are the only fixed points of  $\varphi^1_{f_k}$, and $  \|f_k \|  <\frac{\varepsilon}{8}$.  Observe that since $\varphi^1_{f_k}$ has no fixed points in the closure of $V_k$, there exists $R_k>0$ such that $\varphi^1_{f_k}$ $R_k$--shifts $V_k$.  We let $R= \min\{R_1, R_2\}$.

We pick $r >0$ to be small in comparison to  $R$. We will pick $\delta> 0$  small enough such that if $d_{C^0}(\Id, \psi ) < \delta$, then $\psi$ belongs to the neighborhood  $\nu$ of the identity given by Claim \ref{claim:fragmentation}.  Let $V_1, V_2, \psi_1, \psi_2$,  and $H_1, H_2$ be as in the conclusion of Claim \ref{claim:fragmentation}.  Clearly, for $k=1,2$ the support of $H_k$ is $R$--shifted by $\varphi^1_{f_k}$ and $d_{C^0}(\varphi^t_{H_k}, \Id) < r < R$.  Thus, applying Proposition \ref{prop:epsilon_shift}, we see that 
\begin{equation}\label{eq:shift1}
\B_j(\varphi^1_{f_k} \psi_k) = \B_j(\varphi^1_{f_k}).
\end{equation}

Furthermore,  the map $\psi_2 \varphi^1_{f_1}$ shifts the support of $H_1$ by more than $R-r$.  Since $r$ is small compared to $R$ and $d_{C^0}(\varphi^t_{H_1}, \Id) < r$, we may again apply Proposition \ref{prop:epsilon_shift} and conclude \begin{equation}\label{eq:shift2}
\B_j(\psi_2 \varphi^1_{f_1} \psi_1) = \B_j(\psi_2\varphi^1_{f_1}).
\end{equation}

Combining the above with the continuity property of barcodes (with respect to the Hofer distance), we obtain the following chain of inequalities
$$\dbot(\B_j(\psi_2 \psi_1), \B_j(\Id)) \leq \dbot(\B_j(\psi_2 \psi_1) ,  \B_j(\psi_1)) + \dbot(\B_j( \psi_1), \B_j(\Id)) $$
$$ \leq \dbot(\B_j(\psi_2 \varphi^1_{f_1}\psi_1), \B_j(\varphi^1_{f_1}\psi_1)) + \dbot(\B_j(\varphi^1_{f_1}\psi_1) ,\B_j(\Id)) + 3 \|f_1\|$$ 
$$ \leq \dbot(\B_j(\psi_2 \varphi^1_{f_1}), \B_j(\varphi^1_{f_1})) + \dbot(\B_j(\varphi^1_{f_1}) ,\B_j(\Id) ) + 3 \|f_1\| $$ 
$$ \leq \dbot(\B_j(\psi_2), \B_j(\Id)) + 6 \|f_1\|  $$ 
$$ \leq \dbot(\B_j(\varphi^1_{f_2} \psi_2), \B_j(\Id)) + \|f_2\| +6  \|f_1\|  $$
$$ = \dbot(\B_j(\varphi^1_{f_2}), \B_j(\Id)) + \|f_2\| +6  \|f_1\|  $$
$$ \leq 2 \|f_2\| +6  \|f_1\|  \leq \varepsilon. $$

In the above chain of inequalities, we have applied Hofer continuity in passing from the first line to the second, from the third to the fourth, from the fourth to the fifth, and from the sixth to the seventh, and we have used Equations \eqref{eq:shift1} \& \eqref{eq:shift2} in passing from the second line to the third and from the fifth to the sixth.

\medskip 

It remains to explain why Claim \ref{claim:fragmentation} is true.  The first half of Claim \ref{claim:fragmentation} follows immediately from the following statement  by setting $V_1 = D_1, V_2 = D_2 \cup \ldots \cup D_{2g+2}.$

\begin{claim} \label{cl:fragmentation_original} Denote by $(M, \omega)$  a closed symplectic surface of genus $g$.  There exists a cover of $M$ by disks $D_1, \ldots,D_{2g+2}$ with the following property: For every $r >0$, there exists a neighborhood $\nu$ of the identity in $\Ham(M, \omega)$, such that if $\psi \in \nu$, then  $\psi$ can be written as a composition $\psi = \psi_1 \psi_2 \ldots \psi_{2g+2}$, where each $\psi_i$ is supported in one of the disks $D_j$, and $d_{C^0}(\psi_i, Id) < r$.
\end{claim}

The above is the statement of  Theorem 3.1 from \cite{Sey12}.  The proof given in \cite{Sey12} was obtained by modifying  a similar result from \cite{EPP}.  The result in \cite{EPP} was in turn inspired by earlier works of Fathi \cite{fathi80} and Le Roux \cite{leroux}.  

It remains to prove the latter statement of Claim \ref{claim:fragmentation}.  We leave it to the reader to check that it follows from Claim \ref{cl:fragmentation_original} and Lemma 3.2 of \cite{Sey12} whose statement we will recall for the reader's convenience:

\emph{  Let $B$ denote the unit ball in $\R^{2n}$, $\omega_0$ the standard symplectic form on $\R^{2n}$, and $\psi \in \Ham_c(B, \omega_0)$.  There exists a Hamiltonian $H$ supported in $B$ such that $\varphi^1_H = \psi$ and $d_{C^0}(Id, \varphi^t_H) < d_{C^0}(Id, \psi) $ for all $t\in[0,1]$.}

This completes the proof of Theorem \ref{theo:at_identity}.

\subsection{Proof of Theorem \ref{theo:not_identity}}\label{sec:proof_not_deintity}
 Let $\eta \in \overline{Ham}(\Sigma, \omega)$ as in the statement, and $\varepsilon>0$.
Let $B$ be a disk in $M$ such that $\eta(B) \cap B = \emptyset$.  Such a disk exists because $\eta \neq Id$.  We may assume that $\varepsilon$ is small in comparison to the area of $B$.  We pick a small $\delta>0$ and $\phi, \psi \in \Ham(\Sigma, \omega)$ such that  $d_{C^0}(\phi, \eta) < \delta$, and $d_{C^0}(\psi, Id) < \delta$.
 
 The following claim proves the theorem in the special case where there exists a disc $D$ which contains the support of $\psi$.

\begin{claim}\label{cl:Step1}
Suppose that 
 $\phi(B) \cap B = \emptyset$.   Let $D \subset \Sigma$ be a disc.  There exists a constant $\delta_D$, depending only on $D$,  such that if $\psi$ is supported in $D$ and $d_{C^0}(\psi, Id) \leq \delta_D$, then for any $j \in \Z$ we have $$\dbot(\B_j(\phi \psi), \B_j (\phi)) <  \varepsilon.$$ 
\end{claim}

Before proving the claim we will show how Theorem \ref{theo:not_identity} follows from it.  According to Claim \ref{cl:fragmentation_original}, we can cover the surface $M$ by disks $D_1$, \ldots, $D_{2g+2}$, where $g$ denotes the genus of the surface, such that if $\psi$ is sufficiently $C^0$--small, then it can be written as $\psi = \psi_1 \psi_2  \ldots \psi_{2g+2}$, where each $\psi_k$ is still $C^0$--close to the identity and is supported in one of the disks $D_j$.  
We may assume that $\delta$ has been picked small enough so that 
$d_{C^0}(\psi_{k}, Id) \leq  \delta_{D_j}$, and so that 
$\phi\psi_1 \psi_2  \ldots \psi_k (B) \cap B = \emptyset$ for each $k$.
Then for each $k$ we can apply Claim \ref{cl:Step1} to  $\phi \psi_1 \ldots \psi_k$ and $\psi_{k+1}$ and conclude that  
$
\dbot(\B_j(\phi \psi_1 \ldots \psi_{k+1}), \B_j(\phi \psi_1 \ldots \psi_k)) <  \varepsilon.
$
Finally, we apply the triangular inequality to get
$$
\dbot(\B_j(\phi \psi), \B_j(\phi)) < (2g+2) \varepsilon,
$$
which implies the theorem.

\bigskip

The proof of Claim \ref{cl:Step1} relies on the following fragmentation-type lemma whose proof will be postponed to the end of this section.
 \begin{lemma}\label{lem:frag_rect}
Consider the unit square $[0,1] \times [0,1]$ in $\R^2$ equipped with the standard symplectic structure.  Let $m$ be a positive integer and $\rho$ a positive real number.

Define a partition of $[0,1] \times [0,1]$ into rectangles $U_1, \ldots, U_m$ where $U_i = [0,1] \times [\frac{i-1}{m}, \frac{i}{m}]$. There exists a constant $\delta_0 >0$, depending on $m$, with the following property: If $\psi$ is a Hamiltonian  diffeomorphism, whose support is compactly contained in the interior of $[0,1]\times[0,1]$ and  satisfies $d_{C^0}(\psi, Id) < \delta_0$, then there exist Hamiltonian diffeomorphisms $\theta, \psi_1, \ldots, \psi_m$ such that the support of $\theta$ is of area less than $\rho$,  each $\psi_i$ is supported in the interior of $U_i$, and  $\psi   =  \psi_1 \ldots \psi_m \theta$.
 \end{lemma}
 
 \begin{proof}[Proof of Claim \ref{cl:Step1}]  
Let $\phi,\psi,D$ be as in the statement.
 We will make the assumption that $\mathrm{Area}(D) = 1$.  Using an area-preserving identification\footnote{Note that, although the identification map will have to be non-smooth (at four points) on the boundary of $D$, we can ensure that it is  smooth in the interior of $D$ and so this will allow us to identify $\psi$ with a smooth area-preserving map of $[0,1] \times [0,1]$.} of the disk $D$ with the unit square $[0,1] \times [0,1]$, we will assume that $D$ is an embedding of the unit square $[0,1] \times [0,1]$ in $M$.  


Let $U_1, \ldots, U_m$ be a partition of $D$ into rectangles as described in the statement of Lemma \ref{lem:frag_rect}. Note that the rectangles $U_i$  have the same area.  We will pick $m$ to be large enough so that the area of each  $U_i$ is smaller than $\varepsilon$ and will pick $\rho$ to be small in comparison to $\varepsilon$.

Now, according to Lemma \ref{lem:frag_rect}, there exists a constant $\delta_D$, which depends on the choice of the map identifying $D$ with $[0,1]\times [0,1]$, such that if $d_{C^0}(Id, \psi) \leq \delta_D$, then we obtain  Hamiltonian diffeomorphisms $\theta, \psi_1, \ldots, \psi_m$ satisfying the properties listed in that lemma. 

 Let $N =  \ceil[\bigg]{ \frac{2\mathrm{Area}(D)}{\mathrm{Area}(B)}}$.  For each $ k \in \{1, \ldots, N \}$, let $W_k = \sqcup \, U_j$, where the (disjoint) union is taken over the set $\{j : j=k \; \mathrm{mod} \; N\}$. Define $\Psi_k$ to be the composition of those $\psi_j$'s such that $j = k \; \mathrm{mod} \; N$.  Clearly, $\Psi_k$ is supported in $W_k$.

 We leave it to the reader to check that the total area of each $W_k$ is less than the area of $B$.  Hence, we can find $f_k \in \Ham(M, \omega)$ such that $f_k(W_k) \subset B$.  Furthermore, since each connected component of $W_k$ has area less than $\varepsilon$, the Hamiltonian diffeomorphism $f_k$ can be picked such that $d_{Hofer}(Id, f_k) < \varepsilon$. Note also that since the support of $\theta$ has total area less than $\rho$, we can find $g \in \Ham(M, \omega)$ such that it maps the support of $\theta$ into $B$ and $d_{Hofer}(Id, g) < \rho < \varepsilon $.  Now, consider the Hamiltonian diffeomorphism $$\psi' =   f_1 \Psi_1 f_1^{-1}  f_2 \Psi_2 f_2^{-1} \ldots f_N \Psi_N f_N^{-1}  g \theta g^{-1}.$$ 
It is easy to see that $\psi'$  is supported in $B$.  Furthermore, because $\psi = \Psi_1 \ldots \Psi_N \theta$, using bi-invariance of the Hofer metric, we obtain the following inequalities 

$$d_{\mathrm{Hofer}}(\psi,\psi') \leq 
\sum_{i=1}^N d_{Hofer}(\Psi_i, f_i \Psi_i f_i^{-1})  \; + \; d_{Hofer}(\theta, g\theta g^{-1})$$
$$\leq  \sum_{i=1}^N   2 \; d_{Hofer}(Id, f_i)   + 2 \; d_{Hofer}(Id, g)  <  2(N+1) \varepsilon.
$$  

Since $\phi(B) \cap B = \emptyset$, Proposition \ref{prop:displacement} tells us that $\B_j(\phi \psi') = \B_j(\phi)$.  Furthermore, the fact that $d_{\mathrm{Hofer}}(\psi,\psi') < 2(N+1) \varepsilon$ implies that $\dbot(\B_j(\phi \psi), \B_j(\phi \psi') ) <2(N+1) \varepsilon$.  We conclude from the above that $$\dbot(\B_j(\phi \psi), \B_j (\phi)) < 2(N+1) \varepsilon.$$   
Lastly,  replacing $\varepsilon$ by $\frac{\varepsilon}{2N+1}$ throughout the above argument, completes the proof of the claim.
\end{proof}

We will finish this section by presenting a proof of Lemma \ref{lem:frag_rect}.
\begin{proof}[Proof of Lemma \ref{lem:frag_rect}]
Our proof of this lemma relies on the following extension lemma a proof of which can be found 
in \cite{EPP}; see Lemma 6.3 therein.  
\begin{lemma} \label{lem:extension_lemma}
Let $V''= [0,R] \times [-c'',c'']$ be a rectangle, equipped with some area-form, and $V \subset V' \subset V''$ be two smaller rectangles of the form $V = [0,R] \times [-c,c], \, V' = [0,R] \times [-c',c'] \, ,  0 < c < c' < c''$.  Let $\psi : V' \rightarrow V''$ be an area-preserving embedding such that
\begin{itemize}
\item $\psi$ is the identity near $\{0\} \times [-c''c'']$ and $\{R\}  \times [-c'',c'']$.
\item The area in $\Pi$ bounded by $[0,R] \times \{y\}$ and its image is zero for some, and hence all, $y \in [-c_2, c_2]$.
\end{itemize}

Then, there exists a Hamiltonian diffeomorphism $\theta$ compactly supported in $V''$ such that
$\theta|_{V} = \psi|_{V}$. 
\end{lemma}
 
Pick $r>0$ to be so small such that $r < \frac{\rho}{3(m-1)}.$  For each $1 \leq i \leq m-1$, consider the rectangles $V_i \subset V_i' \subset V_i''$ defined by $V_i = [0,1] \times [\frac{i}{m} - r, \frac{i}{m} + r]$, $V_i' = [0,1] \times [\frac{i}{m} - 2r, \frac{i}{m} + 2r]$, $V_i'' = [0,1] \times [\frac{i}{m} - 3r, \frac{i}{m} + 3r]$.  (We will assume that $r$ is small enough such that the $V_i'' \cap V_j'' = \emptyset$ if $j \neq i$.)  For each $1\leq i \leq m-1$, consider the restriction $\psi|_{V_i'}$.  Pick $\delta_0$ small enough such that if $d_{C^0}(Id, \psi) < \delta_0$, then $\psi(V_i') \subset V_i''$. In other words, $\psi|_{V_i'}$ is an area-preserving embedding of $V_i'$ into $V_i''$.  We will leave it to the reader to check that the hypotheses of Lemma \ref{lem:extension_lemma} are satisfied and hence, applying the lemma, we obtain Hamiltonian diffeomorphisms $\theta_i, \, 1 \leq i \leq m-1$ such that $\theta_i$ is supported in $V_i''$ and $\theta_i = \psi$ on $V_i$.  

Let $\theta = \theta_1 \ldots \theta_{m-1}$ and note that, because $r < \frac{\rho}{3(m-1)}$, the total area of the support of $\theta$ is less than $\rho$.  Next, observe that $\psi \theta^{-1}$ coincides with the identity on each of the rectangles $V_i$.  This implies that the support of $\psi \theta^{-1}$ is contained in the disjoint union $[0,1] \times [0, \frac{1}{m} -r] \, \cup \, [0,1] \times [\frac{1}{m} + r, \frac{2}{m} -r] \, \cup \dots \, \cup \, [0,1] \times [\frac{i}{m} + r, \frac{i+1}{m} -r] \, \cup \, \ldots \, \cup \, 
[0,1] \times [\frac{m-1}{m} + r, 1].$

Define $\psi_1$ to be the restriction of $\psi\theta^{-1}$ to $[0,1] \times [0, \frac{1}{m} -r]$, $\psi_i$ to be the restriction of $\psi \theta^{-1}$ to $[0,1] \times [\frac{i}{m} + r, \frac{i+1}{m} -r]$, where $2 \leq i \leq m-1$, and $\psi_m$ to be the restriction of $\psi\theta^{-1}$ to $[0,1] \times [\frac{m-1}{m} + r, 1].$ Clearly, each $\psi_i$ is compactly supported in the interior of $U_i$ and $\psi = \psi_1 \ldots \psi_m \theta$.  
\end{proof}

\section{Barcodes as invariants of weak conjugacy classes and the proof of Theorem \ref{theo:lef_index_invariance}} \label{sec:proof_invariance_Lef}
In this section, we will prove Theorem \ref{theo:lef_index_invariance} of the introduction: that is, in the case of diffeomorphisms,  the absolute Lefschetz number is invariant under weak conjugacy.   

\subsection{The weak conjugacy relation}\label{sec:almost_conj}

Let us begin by giving a precise definition of the weak conjugacy relation.  Recall that the graph of an equivalence relation $\sim$ on a set $Z$ is the set of pairs $(z,w) \in Z\times Z$ such that $z \sim w$. Given two equivalence relations $\sim_1, \sim_2$, we  say that $\sim_1$ is smaller than  $\sim_2$ if the graph of $\sim_1$ is a subset of the graph of $\sim_2$.\footnote{Note that if  $\sim_1$ is smaller than $\sim_2$, then $z\sim_1 w$ implies  $z \sim_2 w$, \emph{i.e.\ }$\sim_1$ is stronger than $\sim_2$.}    An equivalence relation $\sim$ on a topological space $Z$ is said to be Hausdorff if the quotient $Z/\sim$ is Hausdorff.

\begin{definition}\label{def:almost_conj}  Let $G$ be a topological group.  The weak conjugacy relation is the smallest equivalence relation on $G$ which is both Hausdorff and larger than the conjugacy relation. That is, its graph is the intersection of graphs of all Hausdorff equivalence relations which are larger than the conjugacy relation.
\end{definition}

The weak conjugacy relation may be characterized by the following universal property: $z$ is weakly conjugate to $w$ if and only if $\sigma(z) = \sigma(w)$ for any continuous function $\sigma : G \to Y$, where $Y$ is a Hausdorff topological space  and $\sigma$ is invariant under conjugation.

\medskip
For the rest of this section we will be primarily concerned with the weak conjugacy relation on the group  $G = \overline{\Ham}(\Sigma, \omega)$.   We will be needing the following lemma.

\begin{lemma}\label{lem:almost_conj_lifts}
Suppose that  $\Sigma = \S^2$ and that $f,g \in \overline{\Ham}(
\S^2, \omega)$ are weakly conjugate.   Then, the lifts of $f,g$ to $\overline{\UHam}(\S^2, \omega)$ are pairwise weakly conjugate, i.e.\ denoting the lifts of $f$ by $\tilde f_1, \tilde f_2$ and the lifts of $g$ by $ \tilde g_1, \tilde g_2$, up to relabeling the lifts, we have that $\tilde f_i$ is weakly conjugate to $ \tilde{g_i}$, for $i =1,2$. 
\end{lemma}
\begin{proof}

Let $\tilde f $ be a lift of $f$.   Recall that the other lift of $f$ is given by $\mathrm{Rot} \circ f$, where $\mathrm{Rot}$ denotes the full rotation of the sphere around the North-South axis.

 Denote the lifts of $g$ by $\tilde g, \mathrm{Rot} \circ \tilde g$.    Observe that to  prove the lemma it is sufficient to show that one of $\tilde g, \mathrm{Rot} \circ \tilde g$ is weakly conjugate to $\tilde f$.  In order to obtain a contradiction, we will suppose that this is not the case.  This assumption has the following consequence. 

\begin{claim}\label{cl:separating_lifts}
There exists a continuous mapping $\tau : \overline{\UHam}(\S^2, \omega) \rightarrow X$ such that $X$ is a Hausdorff topological space,  $\tau$ is invariant under conjugation, and  $\tau(\tilde{f}) \neq \tau(\tilde g)$ and $\tau(\tilde{f}) \neq \tau(\mathrm{Rot} \circ \tilde g)$.
\end{claim}
\begin{proof}
Since $\tilde f$ is not weakly conjugate to any of $\tilde g, \mathrm{Rot} \circ \tilde g$, we can find continuous mappings $\tau_1: \overline{\UHam}(\S^2, \omega) \rightarrow X_1$ and $\tau_2: \overline{\UHam}(\S^2, \omega) \rightarrow X_2$ such that $X_1, X_2$ are Hausdorff, $\tau_1, \tau_2$ are invariant under conjugation, and $\tau_1(\tilde f) \neq \tau_1(\tilde g), \tau_2(\tilde f) \neq \tau_2(\mathrm{Rot} \circ \tilde g)$.

Let $X = X_1 \times X_2$ and define $\tau: \overline{\UHam}(\S^2, \omega) \rightarrow X$ by
$$ \tilde h \mapsto (\tau_1(\tilde h), \tau_2(\tilde h)).$$
The topological space $X$ is clearly Hausdorff and the mapping $\tau$ is continuous and invariant under conjugation.  Finally, one can easily check that $\tau(\tilde f)\neq \tau (\tilde g)$ and $\tau(\tilde f)\neq \tau (\mathrm{Rot} \circ \tilde g)$.
\end{proof}

Using the above claim, we will construct a continuous mapping $\sigma : \Ham(\S^2, \\ \omega) \rightarrow Y$, where $Y$ is a Hausdorff topological space  and $\sigma$ is invariant under conjugation, such that $\sigma(f) \neq \sigma(g)$.  This is, of course, a contradiction and we conclude that the lemma must be true.

Let us describe the construction of $\sigma : \Ham(\S^2, \omega) \rightarrow Y$.  Set $Y$ to be the the quotient of $X \times X$ by the equivalence relation $(z,w) \sim (w,z)$.  We leave it to the reader to check that $Y$ is a Hausdorff topological space.  Define
$\sigma : \overline{\UHam}(\S^2, \omega) \rightarrow Y$
by  $$\sigma(\tilde f) = \left( \tau (\tilde f), \tau(\mathrm{Rot}\circ \tilde f)\right).$$ 
Note that $\sigma$ is constructed such that $\sigma(\tilde f) = \sigma (\mathrm{Rot} \circ \tilde f)$.  Thus, it yields a well-defined mapping  $\sigma : \overline{\Ham}(\S^2, \omega) \rightarrow Y$.  Furthermore, $\sigma$ is invariant under conjugation because $\tau$ is invariant under conjugation.  Finally, to complete the proof, we must check that $\sigma(f) \neq \sigma(g)$.  Now, $\sigma(f) = ( \tau (\tilde f), \tau(\mathrm{Rot}\circ \tilde f))$ and $\sigma(g) = \left( \tau (\tilde g), \tau(\mathrm{Rot}\circ \tilde g)\right)$.  By Claim \ref{cl:separating_lifts}, $\tau(\tilde{f}) \neq \tau(\tilde g)$ and $\tau(\tilde{f}) \neq \tau(\mathrm{Rot} \circ \tilde g)$.  This immediately implies that $( \tau (\tilde f), \tau(\mathrm{Rot}\circ \tilde f) \neq ( \tau (\tilde g), \tau(\mathrm{Rot}\circ \tilde g)$.
\end{proof}

\subsection{Proof of Theorem \ref{theo:lef_index_invariance}}

Theorem \ref{theo:lef_index_invariance} is  an immediate consequence of Theorems \ref{theo:barcode_inv} \& \ref{theo:lef_endpoints}, stated below.  Indeed, Theorem \ref{theo:barcode_inv} tells us that the barcodes  we constructed in Section \ref{sec:floer} are invariants of weak conjugacy classes and Theorem \ref{theo:lef_endpoints} states that the absolute Lefschetz number of a Hamiltonian  diffeomorphism is an invariant of its barcode. 
  
  \medskip
  
  Throughout this section, $\Sigma$ will denote a closed surface which we assume is equipped with a symplectic form $\omega$. 
\begin{theo}\label{theo:barcode_inv}
Suppose that  $f,g \in \overline{\Ham}(\Sigma, \omega)$ are weakly conjugate.  
\begin{enumerate}
\item  If $\Sigma\neq \S^2$, then $\B_j(f) = \B_j(g)$, for every index $j$.  Furthermore, the same is true for the total barcodes, \emph{i.e.}\ $\B(f) = \B(g)$.
\item  Suppose that $\Sigma = \S^2$.   Let $\tilde f_1, \tilde f_2$ and $ \tilde g_1, \tilde g_2$ denote the lifts of $f, g$ to $\overline{\UHam}(\S^2, \omega)$.  Then, up to relabeling the lifts, we have that $\B_j (\tilde f_1) =  \B_j( \tilde g_1)$ and  $\B_j (\tilde f_2) =  \B_j( \tilde g_2)$, for every index $j$.  Furthermore, the same is true for the total barcodes, \emph{i.e.}\ $\B(\tilde f_1) = \B(\tilde g_1)$ and $\B(\tilde f_2) = \B(\tilde g_2)$.
\end{enumerate}
\end{theo}
\begin{proof}
The result in the case where $\Sigma \neq \S^2$  follows immediately from conjugacy invariance of barcodes (Equation \eqref{eq:conj_inv2}) and their continuity (Theorem \ref{theo:cont_barcodes_genus}).
 
 
 Now, suppose that $\Sigma = \S^2$.  Let $\tilde f_1, \tilde f_2$ and $ \tilde g_1, \tilde g_2$ denote the lifts of $f, g$ to $\overline{\UHam}(\S^2, \omega)$.  By Lemma \ref{lem:almost_conj_lifts},  up to relabeling the lifts, we have that $\tilde f_i$ is weakly conjugate to $ \tilde{g_i}$, for $i =1,2$.  Once again, the result follows from conjugacy invariance of barcodes and their continuity (Theorem \eqref{theo:cont_barcodes_sphere}).
\end{proof}

   For the next theorem, it might be helpful for the reader to recall the definition of $\# \mathrm{Endpoints}(\B) \cap [0,1)$  from  Section \ref{sec:barcodes_for_ham_diffeos}.   We suppose, for the statement of this theorem and the rest of this section, that $\int_{\Sigma} \omega =1$ when $\Sigma = \S^2$.
  
  \begin{theo}\label{theo:lef_endpoints}
Consider  $f \in Ham(\Sigma, \omega)$ with finitely many fixed points. 

\begin{enumerate}
\item  Suppose that $\Sigma\neq \S^2$.  Then, $$\sum_{x\in Fix_c(f)} |\ind(f,x)| = \#\mathrm{Endpoints}(\B(f)).$$
\item   Suppose that $\Sigma = \S^2$.   Let $\tilde{f}$ denote any of the two lifts of $f$ to $\UHam(\S^2, \omega)$.  Then, 
$$\sum_{x\in Fix_c(f)} |\ind(f,x)| =  \#\mathrm{Endpoints}(\B(\tilde f)) \cap [0,1).$$
\end{enumerate}
\end{theo}

\begin{remark}
  It follows from our proof of the this theorem that the set of endpoints  of bars of $\B(f)$ coincides with the set of actions of contractible periodic orbits for which the corresponding fixed points have non-zero Lefschetz index.  Fixed points with zero Lefschetz index make no contribution to the barcodes of a Hamiltonian diffeomorphism.  Another way to see this is to combine the continuity of barcodes with the fact that zero index fixed points can be removed by a small perturbation.
\end{remark}

\subsection{Local Floer homology and Theorem \ref{theo:lef_endpoints}} \label{sec:local_floer}
Theorem \ref{theo:lef_endpoints} is an immediate consequence of Propositions \ref{prop:rank_endpoints} \& \ref{prop:lefindex=rank}, which are stated below. Both of these propositions rely on the notion of local Floer homology groups $HF(\varphi,x)$ associated to an isolated fixed point $x$ of a Hamiltonian diffeomorphism $\varphi$. The definition and certain properties of local Floer homology will be recalled in Section \ref{sec:local_floer}.   It is important to notice that many statements in this section hold in any dimension. The only statement exclusively valid in dimension $2$ is Proposition \ref{prop:lefindex=rank}. 

\medskip
Denote by $r(\varphi,x)$ the rank of $HF(\varphi,x)$.  The next two propositions relate local Floer homology to barcodes and the Lefschetz index.\label{def-r(f,x)}

\begin{prop}\label{prop:rank_endpoints}
Consider  $f \in Ham(\Sigma, \omega)$ with finitely many fixed points. 
\begin{enumerate}
\item  Suppose that $\Sigma\neq \S^2$.  Then, $$\sum_{x\in Fix_c(f)} r(f,x)= \#\mathrm{Endpoints}(\B(f)).$$
\item   Suppose that $\Sigma = \S^2$.   Let $\tilde{f}$ denote any of the two lifts of $f$ to $\UHam(\S^2, \omega)$.  Then, 
$$\sum_{x\in Fix_c(f)} r(f,x) =  \#\mathrm{Endpoints}(\B(\tilde f)) \cap [0,1).$$
\end{enumerate}
\end{prop}

Although the above result is stated for surfaces only, an appropriately restated version holds in higher dimensions.  See Remark \ref{rem:rank_endpoints_general}.

\begin{prop}\label{prop:lefindex=rank}
Suppose that $x$ is an isolated fixed point of $f \in \Ham(\Sigma, \omega)$.  Then, $|L(f,x)| = r(f,x)$.
\end{prop}

\subsubsection{A review of local Floer homology}\label{localFloer}
 We will provide a brief review of the definition and some properties of local Floer homology.  For further details we refer the reader to  \cite{ginzburg, GinGu09, GinGu10}; see also the original works of Floer which gave rise to this notion \cite{floer89b, floer89}.

\medskip

We begin by reviewing local Morse homology.  Let $h: M \rightarrow \R$ be a smooth function with an isolated critical point $x$.  Take $U$ to be an open neighborhood of $x$ containing no other critical points of $h$ and let $\tilde h$ be a smooth function which, on the open set $U$, is both  Morse and  $C^2$--close to $h$.  The local Morse complex $CM_*^{loc}(\tilde h,x)$ is  generated, over the field $\F$, by those critical points of $\tilde h$ which are contained in $U$.  It can be checked that every Morse trajectory of $\tilde h$ connecting two such critical points is contained entirely in $U$.  The same is also true for broken Morse trajectories.  It follows that the usual Morse boundary map induces a boundary map $\partial$ on $CM_*^{loc}(\tilde h,x)$ and that $\partial^2 = 0$.  The homology of this complex, denoted by $HM_*^{loc}(h,x)$, is called the local Morse homology of $h$ at $x$.  It does not depend on the choice of $\tilde{h}$.

Having introduced local Morse homology, we will move on to local Floer homology.  Let $\bar{z}= [z, u]$ denote a capped $1$--periodic orbit of a Hamiltonian  $H: \S^1 \times M \rightarrow \R$. Let $x$ be the fixed point of $\varphi^1_{H}$ corresponding to $z$, assume $x$ is isolated
and take some neighborhood $U$ of $x$ containing no other fixed point.

 Assume that $z$ is isolated and take $U$ to be an open neighborhood of $z$ in $\S^1 \times M $ which meets no other 1--periodic orbit of $H$ and let $\tilde{H}$ be a Hamiltonian which coincides with $H$ outside $U$, is $C^2$--close to $H$ on $U$.  The orbit $z$ breaks into several orbits $z_i$ of $\tilde{H}$ which are all $C^1$--close to $z$; we pick $\tilde{H}$ such that all these orbits are non-degenerate.  Using the capping $u$ of $z$ we can produce cappings $u_i$ of $z_i$ such that $\A_H([z,u])$ is close to $\A_{\tilde H}([z_i,u_i])$.

By definition, the local Floer complex $CF_*^{loc}( H, \bar z)$ is  generated, over the field $\F$, by the capped 1--periodic orbits $[z_i,u_i]$ of $\tilde H$.  As explained in Section 3.2 of \cite{ginzburg}, every Floer trajectory connecting two such 1--periodic orbits is contained entirely in $U$.  The same is also true for broken Floer trajectories. It follows that the usual Floer boundary map induces a boundary map $\partial$ on $CF_*^{loc}(H, \bar z)$ and that $\partial^2 = 0$.  The homology of this complex, denoted by $HF_*^{loc}(H, \bar z)$, is called the local Floer homology of $H$ at $z$.  It does not depend on the choice of the $C^2$  nearby Hamiltonian $\tilde{H}$. 

\begin{remark} \label{rem:local_floer_non_deg}
If $z$ is a non-degenerate 1--periodic orbit of $H$, then   $HF_*^{loc}(H,\bar z)$ has rank one in degree $\mu_{CZ}(\bar z)$ and is zero in all other degrees. 
\end{remark}

\begin{remark} \label{rem:local_floer_morse}
Local Morse and Floer homologies are related in the following sense:
Let $z$ be an isolated critical point of a $C^2$--small autonomous Hamiltonian $H$.  We will also denote by $z$ the corresponding 1--periodic orbit with its trivial capping.  Then, $HF_*^{loc}(H,z) = HM_*^{loc}(H,z)$.
\end{remark}

Local Floer homology groups serve as building blocks for filtered Floer homology in the following sense: Suppose that all capped 1--periodic orbits $\bar z_i = [z_i, u_i]$ of $H$ with action $c$ are isolated.  Then, for sufficiently small $\varepsilon$, we have
\begin{equation}\label{eq:filtered=local}
HF_*^{(c- \varepsilon, c + \varepsilon)}(H) = \bigoplus_{\A_H(\bar z_i) = c} HF_*^{loc}(H, \bar z_i).
\end{equation}

It turns out that, up to a shift in degree, $HF_*^{loc}(H, \bar z)$ depends only on the Hamiltonian diffeomorphism $\varphi^1_H$ and the fixed point corresponding to $z$.  Here is a more precise statement: Let $H'$ be a Hamiltonian such that $\varphi^1_{H'} = \varphi^1_H$.  Let  $z'$ be a 1--periodic orbit of the flow of $H'$ such that $z'(0) = z(0)$, and consider the capped orbits $\bar z = [z, u]$ and $\bar z' = [z', u']$, for any choice of cappings $u, u'$. Then, there exists an integer $k$\footnote{If $(M, \omega)$ is symplectically aspherical, like surfaces of positive genus, then $k=0$.} such that    

\begin{equation}\label{eq:local_floer_inv}
HF_*^{loc}(H', \bar z') = HF_{*+k}^{loc}(H,z).
\end{equation} 

Hence, given $\varphi \in \Ham(M, \omega)$ with an isolated fixed point $x$ we can define $HF^{loc}(\varphi,x)$ the (ungraded) local Floer homology of $\varphi$ at $x$,  and denote by $r(\varphi,x)$ its rank.  It only depends on a germ of $\varphi$ near the point $x$.

\begin{remark} \label{Remk56}
Notice that one can similarly define local Lagrangian Floer theory: if $L_1,L_2$ are exact Lagrangians submanifolds in $(M,\omega)$, and $p$ an isolated  point of $L_1\cap L_2$, we may consider $HF^{loc}_*(L_1,L_2,p)$. In the next section, we will need the following two facts concerning local Lagrangian Floer homology.

If $\varphi$ is a Hamiltonian diffeomorphism of $M$,  and we set $\Gamma_{\varphi}$ to be the graph of $\varphi$ in $\overline M\times  M$ (\emph{i.e.}\ $(M\times M, -\omega \oplus \omega)$), then $HF^{loc}_*(\Gamma_\varphi, \Delta_M;(p,p))=HF^{loc}_*(\varphi, p)$. This follows from the identification of the two Floer homologies (\emph{i.e.}\ non-local version) as in \cite{viterbo-fcfh2} and the fact that after a perturbation of $H$ away from $p$, formula (\ref{eq:filtered=local}) reduces to a single term identifying  $HF^{loc}_*$  and $HF_*^{(c- \varepsilon, c + \varepsilon)}$.

Consider the case of a cotangent bundle $T^*N$ equipped with its canonical symplectic structure.  Let $p$ be an isolated point of $L \cap O_N$ where $L$ is an exact Lagrangian in $T^*N$ and $O_N$ is the zero section.  Suppose that near the point $p$, the Lagrangian $L$ can be written as the graph of the differential of a $C^2$--small function $F: N \rightarrow \R$.  Then, similarly to Remark \ref{rem:local_floer_morse}, we have $HF_*^{loc}(L, O_N, p) = HM_*^{loc}(F, p)$.  
\end{remark} 
 
\subsubsection*{Local Floer homology and local generating functions}
Consider $\R^{2n}$ equipped with its standard symplectic structure $\omega_0$.
Let $\varphi$ be a germ of a symplectic diffeomorphism at $0 \in {\mathbb R}^{2n}$ ($\varphi(0)=0)$).  We write $\varphi(x,y)=(X,Y)$ for $x,y, X, Y \in {\mathbb R}^n$. 

Consider $(\R^{2n} \times \R^{2n}, - \omega_0 \oplus \omega_0)$ and let $\Delta \subset \R^{2n} \times \R^{2n}$ denote the diagonal.   
Consider the symplectomorphism $\Psi : (\R^{2n} \times \R^{2n}, - \omega_0 \oplus \omega_0) \rightarrow (T^* \Delta, \omega_{can})$ given by $$(x,y, X, Y) \mapsto (x, Y, y-Y, X-x),$$ which maps the diagonal in $\R^{2n} \times \R^{2n}$ to the zero section in $T^*\Delta$. The symplectomorphism $\Psi$ identifies the graph of $\varphi$ with a Lagrangian, which we will denote by $\Gamma_{\varphi}$, in $T^*\Delta$.

Suppose that  $\Gamma_{\varphi}$ is a graph over $\Delta$ near the point $(0,0,0,0) \in T^*\Delta$.  Then, we can find a function  $F(x,Y)$ which is defined near $(0,0)$ and such that $(X,Y)=\varphi(x,y)$ if and only if

\begin{equation}\label{eq:GF}
 y-Y= \frac{\partial F} {\partial x}(x,Y), X-x= \frac{\partial F}{\partial Y}(x,Y).
\end{equation}

The function $F(x, Y)$ is called a local generating function for $\varphi$ near the fixed point $p$.
In the following lemma $\varphi$ denotes a Hamiltonian diffeomorphism of any closed monotone symplectic manifold with an isolated fixed point $p$.  Recall that we say $p$ is totally degenerate if all of the eigenvalues of the derivative of $\varphi$ at $p$ are one.
\begin{lemma}\label{lem:locla_floer=local_gf}
Suppose that $p$ is a totally degenerate fixed point of $\varphi$.  Then, $\varphi$ admits a local generating function of the form  $F(x,Y)$ near $p$.   Furthermore, $r(\varphi, p)$ coincides with the rank of $HM^{loc}_*(F,p)$, the local  Morse homology of $F$ at $p$. 
\end{lemma}
\begin{proof}
We will only provide a sketch of the proof as this is a rather standard result; see \cite{hingston, ginzburg, GinGu09}.  The argument we give is borrowed from the first paragraph of the proof of Lemma 3.10 of  \cite{GinGu15}.

  Since $p$ is a totally degenerate fixed point, up to conjugation by a symplectomorphism, we may assume that $\varphi$ is $C^1$ close to the identity near $p$ which in turn implies that $\Gamma_{\varphi}$ is $C^1$--close to $\Delta$ near the point $(0,0,0,0)$.  This clearly implies that we can find a generating function of the form $F(x,Y)$.  Moreover, $F$ will be $C^2$--small.  The equality of $r(\varphi, p)$ and the rank of $HM^{loc}_*(F,p)$ then follows from Remark \ref{Remk56}.
\end{proof}

\subsubsection{Proof of Proposition \ref{prop:rank_endpoints}}\label{sec:rank_endpoints}
    Let $H$ be a Hamiltonian such that $\varphi^1_H = f$.  In the case of $\Sigma = \S^2$, we pick $H$ such that $\{\varphi^t_H\}_{0\leq t \leq 1}$ is a representative of $\tilde f$.  Clearly, it is sufficient to prove the proposition for the barcode $\B(H)$ instead of $\B(f)$.

   Suppose that $c \in \Spec(H)$ and denote by $\# \mathrm{Endpoints}(\B(H)) \cap \{c\}$ the total number of bars of $\B(H)$ which have $c$ as an endpoint.  By Proposition  \ref{prop:endpoints_in_interval}, we have 
   $$\# \mathrm{Endpoints}(\B(H)) \cap \{c\} =  \mathrm{rank} (HF_*^{(c-\varepsilon, c+ \varepsilon)}(H)),$$
for sufficiently small $\varepsilon$.  Using Equation \eqref{eq:filtered=local}, we obtain 
$$\# \mathrm{Endpoints}(\B(H)) \cap \{c\} =  \sum_{ \mathcal{A}_H(\bar z)=c} \mathrm{rank}(HF_*^{loc}(H,\bar z)).$$ 

First, suppose that $\Sigma \neq \S^2$.  Then,  summing both sides of the above equation over all $c \in \Spec(H)$ yields the result.  Next, suppose that  $\Sigma = \S^2$.  In this case, summing both sides of the above equation over all $c \in \Spec(H) \cap [0,1)$ yields 

 $$\# \mathrm{Endpoints}(\B(H)) \cap [0,1) =  \sum_{ \mathcal{A}_H(\bar z) \in [0,1) } \mathrm{rank}(HF_*^{loc}(H,\bar z)).$$
 
 We must now show that $$\sum \limits_{ \mathcal{A}_H(\bar z) \in [0,1) } \mathrm{rank}(HF_*^{loc}(H,\bar z)) = \sum \limits_{x \in Fix_c(f)} r(f,x).$$
We leave it to the reader to check that this is an immediate consequence of  the following claim.
\begin{claim}
For every periodic orbit $z$ of $H$, there exists a unique capping $v$ such that $$\A_H([z,v]) \in [0,1).$$
\end{claim}
\begin{proof}
  Let $A$ denote the generator of $\pi_2(\S^2)$ such that $\omega(A) =1$. (Recall our assumption that $\int_{\S^2} \omega = 1$.)  Now, fix a capping $u$ of $z$. Any other capping of $z$ is of the form $u \# kA$ for some $k\in \Z$.
Recall that $\A_H([z, u\# kA]) = \A_H([z, u]) - k\omega(A) = \A_H([z, u]) - k$.  Clearly, there exists a unique $k \in \Z$ such that  $\A_H([z, u]) - k \in [0,1)$.  The capping $v$ is given by $u \#kA$.
\end{proof}
This completes the proof of Proposition \ref{prop:rank_endpoints}.

\begin{remark}\label{rem:rank_endpoints_general} The proof  presented above may easily be generalized to prove the following statement in higher dimensions:
 \begin{enumerate}
\item  Suppose that $(M, \omega)$ is aspherical. Then, $$\sum_{x\in Fix_c(f)} r(f,x)= \#\mathrm{Endpoints}(\B(f)).$$
\item   Suppose that $(M, \omega)$ is monotone.   Let $\tilde{f}$ denote any lift of $f$ to $\UHam(M, \omega)$ and let $\tau$ be the positive generator of $\omega(\pi_2(M))$.  Then, $$\sum_{x\in Fix_c(f)} r(f,x) =  \#\mathrm{Endpoints}(\B(\tilde f)) \cap [0, \tau) .$$
\end{enumerate}
\end{remark}
   
\subsubsection{Proof of Proposition \ref{prop:lefindex=rank}} \label{sec:lefindex=rank}

Let $p$ be an isolated fixed point of $f$. We want to prove that   $r(f,p) = \vert \ind (f,p) \vert $. 

When $p$ is non-degenerate,  $\ind (f,p)= \pm 1$ and so the result follows from Remark \ref{rem:local_floer_non_deg}.  Thus, we will assume, for the rest of the proof, that $p$ is a degenerate fixed point.  Since we are working in dimension $2$, this implies that $p$ is totally degenerate.

To simplify notations, we will consider a small chart centered at $p$ and work with a germ of $f$ at $p$. Thus, we will think of $f$ as a Hamiltonian diffeomorphism defined near $p=(0,0) \in \R^2$ with an isolated fixed point at $p$. By Lemma  \ref{lem:locla_floer=local_gf}, $f$ admits a local generating function of the form $F(x,Y)$ near $p$; see Equation \eqref{eq:GF}. 
Moreover,  according to the same lemma, $r(f,p)$ coincides with the rank of  the local Morse homology group $HM^{loc}_*(F, p)$.
Proposition \ref{prop:lefindex=rank} follows immediately from the following two lemmas.

\begin{lemma}\label{lem:indices}
$\ind(f,p)  = \ind(\chi_F, p)$, where $\chi_F$ denotes the Hamiltonian vector field of $F$ and $\ind(\chi_F, p)$ denotes its Lefschetz index at $p$.
\end{lemma}
\begin{proof}
 Let $T =\{(x,y): x^2 + y^2 = r^2\}$ for a very small value of $r$.  Observe that $L(f,p)$ is the degree of the map  $\Phi : T \rightarrow \S^1$ given by $(x,y) \mapsto (\frac{X-x}{\|X-x\|},\frac{Y-y}{\|Y-y\|})$.
 
 The Lefschetz index $\ind(\chi_F, p)$ is, by definition, the degree of the map $\Psi : T' \rightarrow \S^1$ given by the formula $(x,Y) \mapsto  \frac{\chi_F(x,Y)}{\|\chi_F(x,Y)\|},$ where $T'$ can be taken to be the boundary of any disk whose interior contains $p$ and no other zeros of $\chi_F$.  We will take $T'$ to be the image of the circle $T$ from the previous paragraph under the mapping $(x,y) \mapsto (x,Y)$.  It follows from the definition of $F$ (see Equation \eqref{eq:GF}), that  $\Psi : T' \rightarrow \S^1$ is given by the formula $(x,Y) \mapsto ( \frac{X-x}{\|X-x\|},\frac{Y-y}{\|Y-y\|})$.

Therefore, to prove the lemma we must show that the map from $T$ to $T'$ given by $(x,y) \mapsto (x,Y)$ has degree one.  It is easy to see that the degree of this map is given by the sign of $\frac{\partial Y}{\partial y}$.  We claim that $\frac{\partial Y}{\partial y}$ is positive: indeed, since $p$ is totally degenerate, up to conjugation by a symplectomorphism, we may assume that $f$ is $C^1$ close to the identity near $p$.  Of course, this would imply that $\frac{\partial Y}{\partial y}$ is close to $1$ near $p$. 
\end{proof}

\begin{lemma}\label{lem:rank=lef_gf}
The rank of $HM_*^{loc}(F, p)$ coincides with $| \ind(\chi_F, p)|.$
\end{lemma}
\begin{proof}
Using the fact that $\chi_F$ is an area-preserving vector field, one can show that it can be described as follows near the point $p$:  Either, every orbit near $p$ is periodic, or one can find a neighborhood of $p$ which may be divided into a finite number, say $h$, of hyperbolic sectors of $\chi_F$; see Figure \ref{Fig:figure1}.  Hyperbolic sectors are defined in Section \ref{sec:local-homeo-area-pres}.

\begin{figure}[h!]
\centering
\includegraphics[width=5cm]{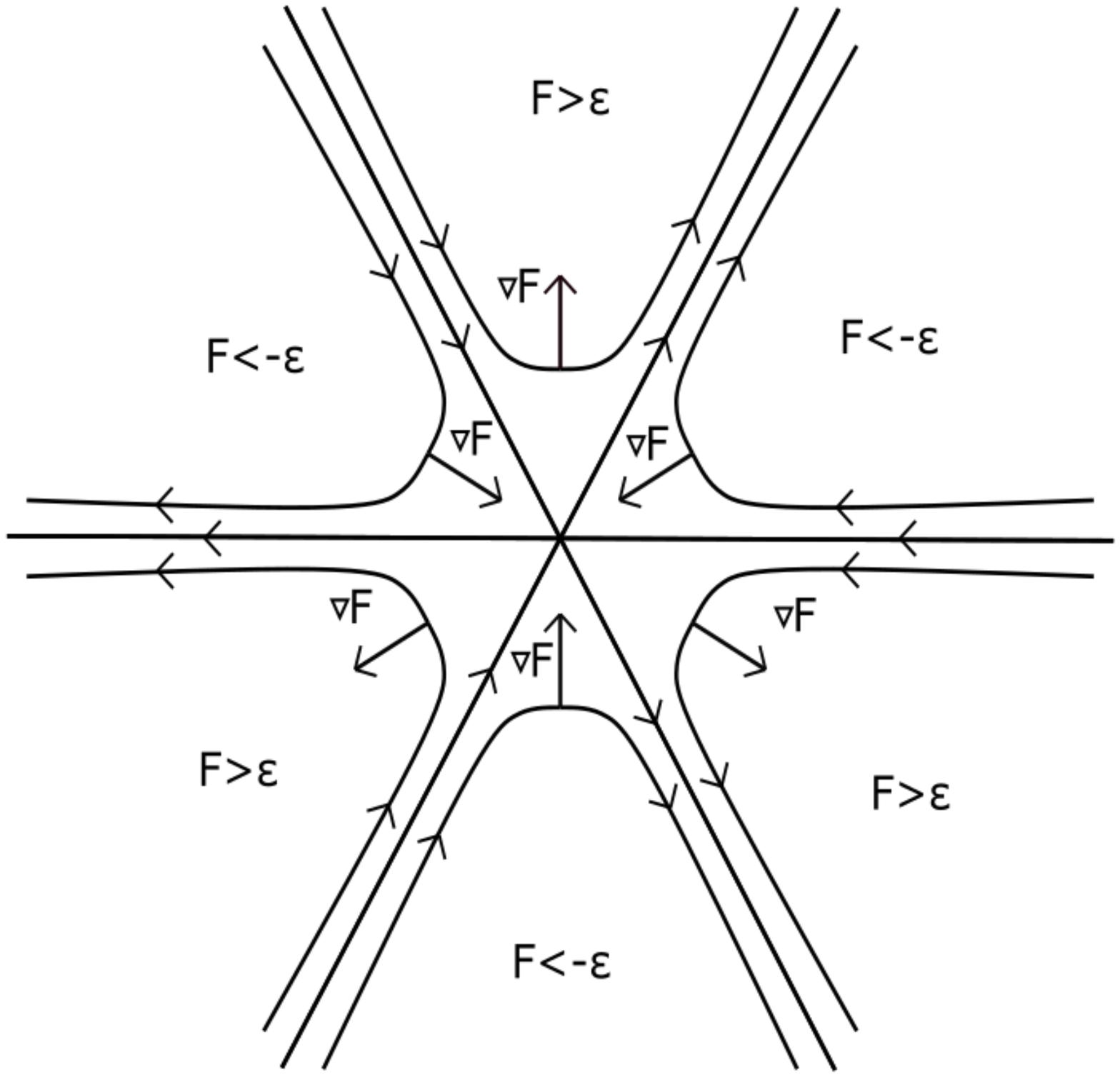}
\caption{\label{Fig:figure1}Phase portrait of $\chi_F$ near a singular point with $6$ hyperbolic sectors.  The Lefschetz index of the singularity is -2.}
\end{figure}

In the first case, where all the orbits of $\chi_F$ in a neighborhood of $p$ are periodic, $p$ is either a local maximum/minimum of $F$.  In this case, it is easy to see the $L(\chi_F, p)$ and the rank of $HM_*^{loc}(F,p)$ are both $1$.

Now suppose that $\chi_F$ has $h$ hyperbolic sectors.   An easy computation would show that $\ind(\chi_F, p) = 1 - \frac{h}{2}$. Suppose that $F(p) = 0$.   
The local Morse homology $HM_*^{loc}(F,p)$ coincides with the singular homology of the pair $(F ^{\varepsilon}, F^{-\varepsilon}) $.  Hence, it is sufficient to show that the singular homology $H_*(F ^{\varepsilon}, F^{-\varepsilon})$ has rank $\frac{h}{2} - 1$.  Now, $H_*(F ^{\varepsilon}, F^{-\varepsilon}) = H_*(F ^{\varepsilon} \cap D, F^{-\varepsilon} \cap D)$ where $D$ is a small disk centered at $p$.  Recall that the trajectories of $\chi_F$ are levels of $F$ and $F$ increases in the direction of $\nabla F$ (see Figure \ref{Fig:figure1}). Thus, it is clear that $(F^\varepsilon \cap D, F^{-\varepsilon}\cap D)$ has the homotopy type of the pair  $(D, \partial D^-)$, where $\partial D^-$ is the set of points on the unit circle where 
$F$ is strictly negative. Note that trajectories of $\chi_F$ are contained in a level of $F$, and on a trajectory in a hyperbolic sector, $F$ has positive value if $\nabla F$ points away from the origin (so that $F$ increases as we move away from $0$) and negative when $\nabla F$ points towards $0$. Thus, $\partial D^-$ is a union of intervals on the circle, one for each hyperbolic sector such that $\nabla F$ points towards the origin. We have $h/2$ such sectors, so $\partial D^-$ is a union of $h/2$ disjoint intervals and $H_k(D,\partial D^-)$ has rank $h/2-1$ for  $k=1$ and $0$ otherwise.  This completes the proof.
\end{proof}

\subsection{Absolute Lefschetz number for smoothable homeomorphisms}

In this section we will prove Theorem~\ref{theo:lef_index_invariance_homeos} of the introduction:  Suppose that $f,g$ are Hamiltonian homeomorphisms which are weakly conjugate and  have a finite number of fixed points.  If $f,g$ are smoothable, then they have the same absolute Lefschetz numbers.


The proof of Theorem \ref{theo:lef_index_invariance_homeos} follows the same general outline as the proof of Theorem \ref{theo:lef_index_invariance}.  Indeed, as in the case of  Theorem \ref{theo:lef_index_invariance}, we use  Theorem \ref{theo:barcode_inv} to conclude that the barcodes  we have constructed are invariants of weak conjugacy classes.  We then use the theorem below to conclude that the absolute Lefschetz number of a smoothable Hamiltonian  homeomorphism is an invariant of its barcode.

 \begin{theo}\label{theo:lef_endpoints_homeos}
Suppose that  $f \in \overline{Ham}(\Sigma, \omega)$ has finitely many fixed points and that $f$ is smoothable.   

\begin{enumerate}
\item  Suppose that $\Sigma\neq \S^2$.  Then, $$\sum_{x\in Fix_c(f)} |\ind(f,x)| = \#\mathrm{Endpoints}(\B(f)).$$
\item   Suppose that $\Sigma = \S^2$.   Let $\tilde{f}$ denote any of the two lifts of $f$ to $\overline{\UHam}(\S^2, \omega)$.  Then, 
$$\sum_{x\in Fix_c(f)} |\ind(f,x)| =  \#\mathrm{Endpoints}(\B(\tilde f)) \cap [0,1).$$

\end{enumerate}
\end{theo}

Although the statement of the above theorem is   similar to that of Theorem \ref{theo:lef_endpoints}, the proof is different as local Floer theory is not well-defined for homeomorphisms.  One may interpret this theorem as a definition for local Floer homology of homeomorphisms.

\begin{proof}[Proof of Theorem \ref{theo:lef_endpoints_homeos}]
 Since $f$ is smoothable there exists a sequence $f_n \in \Ham(M, \omega)$ such that $f_n$ converges uniformly to $f$ and $Fix_c(f_n) = Fix_c(f)$ for every $n$.  Because the sets $Fix_c(f_n), Fix_c(f)$ all coincide we will simplify denote them by $Fix_c$.

In the case of $\Sigma = \S^2$, we pick lifts $\tilde f_n$ of $f_n$ to $\UHam(\S^2, \omega)$ such that $\tilde f_n$ converges to $\tilde f$ with respect to the $C^0$ topology on $\overline{\UHam}(\S^2, \omega)$. By Theorem \ref{theo:cont_barcodes_genus},  $\B_j(f_n) \to \B_j(f)$ and $\B(f_n) \to \B(f)$ when $\Sigma \neq \S^2$.  By Theorem \ref{theo:cont_barcodes_sphere}, $\B_j(\tilde f_n) \to \B_j(\tilde f)$ when $\Sigma = \S^2$.

As the reader might anticipate, the strategy for the proof is to show that  the number of endpoints for each barcode of $f_n$ coincides with the number of endpoints for the corresponding barcode of $f$; for example, we will show that $\#\mathrm{Endpoints}(\B_j( f_n)) = \#\mathrm{Endpoints}(\B_j(( f))$.   What makes this non-trivial is that a priori $\B_j( f_n)$ might possess small bars which disappear as $n \to \infty$.   This difficulty is tackled by the following lemma. 
 
\begin{lemma}\label{lem:barcodes_stabilize}
The sequences $f_n, \tilde f_n$ may be picked such that
\begin{enumerate}
\item When $\Sigma \neq \S^2$, we have $\B_j(f_n) = \B_j(f)$ and $\B(f_n) = \B(f)$,

\item When $\Sigma = \S^2$, we have $\B_j(\tilde f_n) = \B_j(\tilde f)$ and $\B(\tilde f_n) = \B(\tilde f)$.
\end{enumerate}  
\end{lemma}

Before proving this lemma, let us show that it implies Theorem \ref{theo:lef_endpoints_homeos}.  Indeed, by continuity of the Lefschetz index, for each $x \in Fix_c$, we have $L(f,x) = L(f_n,x)$ for $n$ sufficiently large.  Hence,  for $n$ large enough, $$\sum_{x\in Fix_c} |\ind(f,x)| = \sum_{x\in Fix_c} |\ind(f_n,x)| .$$  On the other hand, Theorem \ref{theo:lef_endpoints} tells us that if $\Sigma \neq \S^2$, then $$\sum_{x\in Fix_c} |\ind(f_n,x)|  = \# \mathrm{Endpoints}(\B(f_n)),$$ 
and if $\Sigma = \S^2$, then
$$\sum_{x\in Fix_c} |\ind(f_n,x)| =  \#\mathrm{Endpoints}(\B(\tilde f_n))\cap [0,1).$$
The theorem then follows from  Lemma \ref{lem:barcodes_stabilize}.

\medskip

It remains to prove the lemma.   Before getting into the proof, recall that the barcodes $\B_j(f_n), \B_j(f)$ belong to $\widehat{\mathcal{B}}$ which is the space of barcodes considered up to shift.  It will be more convenient to work with barcodes as opposed to barcodes considered up to shift.  To that end, we pick Hamiltonians $H_n$ such that $\varphi^1_{H_n} = f_n$ and work with $\B_j(H_n), \B(H_n)$.  In the case of the sphere, we pick the $H_n$'s such that $\{\varphi^t_{H_n}\}_{0\leq t \leq 1}$ represents $\tilde f_n$ in $\UHam(\S^2, \omega)$. 
We start with any sequence $H_{n}$ and we will show that for each fixed $j$ we can modify the $H_{n}$'s
so that the sequence of barcodes $\B_j(H_n)$ eventually stabilizes, \emph{i.e.}\ we have $\B_j(H_n) = \B_j(H_{n+1})$ for sufficiently large $n$.  Likewise we will also get  $\B(H_n) = \B(H_{n+1})$ for sufficiently large $n$.  It is clear that Lemma \ref{lem:barcodes_stabilize} follows from the above.  

\medskip

We will begin by proving that, for each fixed $j$, the sequence $\B_j(H_n)$ eventually stabilizes. Denote by $E_j(H_n)$ the set of values in the spectrum of $H_n$ which appear as endpoints of bars in $\B_j(H_n)$.  Note that this is the set of endpoints of bars in $\B_j(H_n)$ counted without their multiplicities.\footnote{Recall that, on the other hand, $\mathrm{Endpoints}(\B_j(H_n))$ denotes the set of endpoints of the bars of $\B_j(H_n)$ taken with their multiplicities.}

%

We claim that  $E_j(H_n)$ is finite, for all $ n$. Indeed, the cardinality of this set is bounded by the total number of fixed points: This is obvious, in the case where $\Sigma \neq \S^2$.  We leave it to the reader to check that it continues to hold when $\Sigma = \S^2$.

Write $E_j(H_n)= \{ \mathcal{A}^n_1, \mathcal{A}^n_2, \ldots, \mathcal{A}^n_{k}\}$, where  we suppose that $\mathcal{A}^n_1 < \mathcal{A}^n_2 < \ldots <  \mathcal{A}^n_{k}$. Note that we're assuming here that the sets $E_j(H_n)$ have the same cardinality for different values of $n$.  This is justified, up to passing to a subsequence of $\B_j(H_n)$, because the cardinality of $E_j(H_n)$ is bounded by the total number of fixed points. 

\begin{claim}\label{cl:modify_action}
For every $\delta>0$ there exists $\varepsilon$ such that  for every $a_1, a_2, \ldots, a_{k} \in [-\varepsilon, \varepsilon]$ we can find a hamiltonian $K$ with $C^2$ norm less than $\delta$, such that for every $n$,
\begin{itemize}
\item   $K$ is constant on a neighborhood of $Fix_c$ and $\varphi^1_{K} f_{n}$ has the same set of contractible fixed points as $f_{n}$,
\item   $E_j(\tilde H_n) =  \{ \mathcal{A}^n_1 + a_1, \mathcal{A}^n_2 + a_2, \ldots, \mathcal{A}^n_k + a_k\}$, where
$\tilde H_{n} = K \# H_{n}$ is a hamiltonian with $\varphi^1_{\tilde H_{n}} = \varphi^1_{K} f_{n}$.
\end{itemize}
\end{claim}

Next, shift each  $H_n$ by a constant, if necessary, to ensure that  that the minimum value in $E_j(H_n)$ is zero.

\begin{claim}\label{cl:spectrum_bdd}
The set $\cup_n E_j(H_n)$ is bounded.
\end{claim}
\begin{proof}[Proof of Claim]
  Applying the previous claim, we may assume that, for each $n$, the difference between the largest value in $E_j(H_n)$ and the second largest value is at least $\varepsilon$.  Therefore, each barcode $\B_j(H_n)$ has a bar $I_n=[c_n, d_n]$ of length at least $\varepsilon$.  If $\cup_n E_j(H_n)$ was not bounded, we would conclude  that $d_n \to \infty$.  It is easy to see that these conditions force the sequence of barcodes $\B_j(H_n)$ not to have a limit in $\widehat{\mathcal{B}}$, the space of barcodes considered up to shift. (Of course, it does not have a limit in the space of barcodes either).  This contradicts the fact $\B_j(H_n) \to \B_j(f)$, when $\Sigma \neq \S^2 $, and $\B_j(H_n) \to \B_j(\tilde f)$, when $\Sigma = \S^2 $. 
\end{proof}

We continue the proof of Lemma \ref{lem:barcodes_stabilize}.  Because the sequence of barcodes  $\B_j(H_n)$ is convergent, to prove the lemma it is sufficient to pick the Hamiltonians $H_n$  such  that the sets $E_j(H_n) $ stabilize, \emph{i.e.}\ $E_j(H_n) =  E_j(H_{n+1})$ for $n$ large. 

Let us show that we can modify the $H_n$'s so that the sets $E_j(H_n)$  stabilize.      By Claim \ref{cl:spectrum_bdd}, after passing to a subsequence, we may assume that $\mathcal{A}^n_i$ converges to a value which we will denote by $\mathcal{A}_i$.    Since $\mathcal{A}^n_i - \mathcal{A}_i$ converges to $0$, by Claim \ref{cl:modify_action},  we can perturb $H_n$ to guarantee that  $\mathcal{A}^n_i = \mathcal{A}_i $, for $n$ sufficiently large. Thus, we have $E_j(H_n) = \{\mathcal{A}_1, \ldots, \mathcal{A}_k \}$.  This completes the proof of the fact that the sequence $\B_j(H_n)$ stabilizes for $n$ large enough.

\medskip

It remains to show that the sequence of total barcodes $\B(H_n)$ also stabilizes\footnote{Note that when $f$ is non smooth, it is possible that $B_{j}(f)$ is non-empty for infinitely many $j$'s, even in the case $\Sigma \neq \S^2$.}.  If $\Sigma \neq \S^2$, then we can repeat all of the above with $\B_j(H_n)$ replaced with $\B(H_n)$ and $E_j(H_n)$ replaced with $E(H_n)$, where $E(H_n)$ denotes the set of values in the spectrum of $H_n$ which appear as endpoints of bars in $\B(H_n)$. 

The above reasoning does not apply to the case of $\S^2$ because, on $\S^2$, the sets $E(H_n)$ are neither finite nor bounded. However, $\B(H_n)$ is the disjoint union of the sets $\B_j(H_n)$ and, by Equation \eqref{eq:periodicity_sphere}, we have $\B_j(H_n) = \B_{j-4}(H_n) -1$.  Therefore, the fact that for each $j$ the sequence $\B_j(H_n)$ stabilizes implies that $\B(H_n)$ stabilizes as well.  This completes the proof of Lemma \ref{lem:barcodes_stabilize}.

\end{proof}

\section{Proof of Theorem \ref{theo:almost_conj2}:   Homeomorphisms which are not weakly conjugate to  diffeomorphisms}\label{sec:proof_almost_conj2}
In this section we provide a proof for Theorem \ref{theo:almost_conj2}.  We first treat the case of surfaces other than the sphere leaving the case of the sphere to the end.  

As mentioned in the introduction it is sufficient to give an example of  a Hamiltonian homeomorphism $f$ such that the set of endpoints of $\B(f)$ is unbounded.   Pick an area-preserving chart $V$  with coordinates $(x, y)$.   We may identify $V$ with a Euclidean disk in $\R^{2}$ whose radius we will denote by $R>0$.  We  denote the origin in these coordinates by $O$ and define $r(x,y)$ to be the usual Euclidean distance between $(x,y)$ and $O$.  Let $G: \Sigma \setminus \{O\}  \rightarrow \R$ be a function whose support is compactly contained in $V$ and which is of the form:  
$G(x,y) = h(\frac{r^2}{2})$, where  $h: (0,\frac{R^2}{2}) \to \R$ is smooth,  $h(s) = \frac{1}{s}$ when $s \leq r_0$, where $r_0$ is sufficiently small, and $h(s) = 0$ for $s$ near $R$. Define $f: \Sigma \rightarrow \Sigma$ by $f(O) = O$ and  $\forall p \in \Sigma \setminus\{O\},\ f(p)=\varphi^1_G(p).$ 

Now, $f$ is a Hamiltonian homeomorphism of $\Sigma$,  because it is  the uniform limit of $\varphi^1_{G_i}$ where $G_i$ is a smooth Hamiltonian such that $G_i(x,y) = h_i(\frac{r^2}{2})$, where  $h_i: [0, \infty) \to \R$ is smooth and  $h_i(\frac{r^2}{2}) =  h(\frac{r^2}{2})$  for $r \geq \rho_i$  and  $\rho_i \to 0$ .   

 It is convenient here to work with a fixed representative of $\B(f)$ which is only defined up to shift.  To pick a representative we  normalize $\B(f)$ such that the ends of bars corresponding to the fixed points outside its support have action zero.  This can be achieved as follows:  Let $\B(G_i)$ be the barcode of  the Hamiltonian $G_i$ from the last paragraph; note that $\B(G_i)$ is well-defined as a barcode as opposed to a barcode up to shift. It can be checked that the sequence $\B(G_i)$ has a limit in the bottleneck distance and so we take this limit  to represent $\B(f)$.

Remark that for $r\leq r_0$, the $1$--periodic orbits of $G$ appear at values $r_k$ such that $h'(\frac{r_k^2}{2}) =  - 2 \pi k$, where $k$ is a positive integer.  A simple computation  shows that the action of a periodic orbit corresponding to $r_k$ is given by 

 $$c_k := h\left(\tfrac{r_k^2}{2}\right) - \tfrac{r_k^2}{2} \,h'\!\left(\tfrac{r_k^2}{2}\right).$$
 
One can easily check that $c_k = 2 \sqrt{2\pi k}$.  Now, the following claim tell us that the values $c_k$ appear as endpoints of some bars in $\B(f)$.  This, of course, implies that the set of endpoints of $\B(f)$ is not bounded as $c_k \to \infty$.  

\begin{claim}\label{cl:endpoint_at_Ak}
  The number of bars in $\B(f)$ with an endpoint at $c_k$ is given by the rank of $HF^{(c_k -  \delta, c_k + \delta)}(G)$ for sufficiently small $ \delta$.   Furthermore,  $HF^{(a_k -  \delta, a_k + \delta)} \\ (G)$ has rank $2$.  
\end{claim}
\begin{proof}
At first glance, the first statement appears to be the content of Proposition \ref{prop:endpoints_in_interval}.  But this proposition does not apply verbatim as $G$ is not smooth.  However, $G$ is a smooth Hamiltonian on a neighborhood of the set  $G^{-1}(c_k - \delta, c_k + \delta)$ and it is not difficult to see that the proposition does apply in this setting.

Let us now prove that $HF^{(a_k -  \delta, a_k + \delta)} (G)$ has rank $2$.  Computation of this type of Floer homology groups is a classical example which goes back to \cite{CFHW}; see also \cite{  oancea, Sey15}.  Therefore, we will only sketch an outline of the computation.
  
  Recall that we are considering periodic orbits of $G$ corresponding to $r = r_k$.  These orbits form circles which we will denote by $S_k$.  Let $U_k$ denote a small open neighborhood of $S_k$.  Performing a $C^2$--small perturbation of  $G$ near $S_k$ one obtains a  Hamiltonian $\tilde G$ which has exactly two non-degenerate $1$--periodic orbits in $U_k$.  The Conley--Zehnder indices of these orbits  are $2k-1$ and $2k$.    Hence, the Floer chain complex $CF_*^{(c_k - \delta, c_k + \delta)}(\tilde G)$ has rank two and is supported in degrees $2k-1$ and $2k$.  It is shown in Proposition 2.2 of \cite{CFHW} that the boundary map of this complex is zero.  The result follows immediately.
\end{proof}

It remains to explain why $f$ is not weakly conjugate to any Hamiltonian diffeomorphism in the case where $\Sigma = \S^2$.   By Lemma \ref{lem:almost_conj_lifts}, if $f$ were weakly conjugate to  $h\in \Ham(\S^2, \omega)$, then the lifts of $f$ to $\overline{\UHam}(\S^2, \omega)$ would be pairwise weakly conjugate to the lifts of $h$.   Hence, it is sufficient  to show that $f$ has a lift $\tilde f \in  \overline{\UHam}(\S^2, \omega)$ which is not weakly conjugate to any $\tilde h \in \UHam(\S^2, \omega)$.  

Note that the non-smooth function $G$, introduced above, has a well-defined Hamiltonian flow which we will denote by $\phi^t_G$.  We let $\tilde f$ to be the lift of $f$ given by the path $\{\phi^t_G\}_{t \in [0,1]}$.  For each index $j$, we pick a representative of $\B_j(\tilde f)$ by taking the barcode  $\lim_{i \to \infty} \B_j(G_i)$, where the $G_i$'s are the smooth Hamiltonians introduced above.

In the case of $\S^2$, the action spectrum of no smooth Hamiltonian is bounded because of the effect of cappings.  However, one can easily show that for $\tilde h \in \UHam(\S^2,\omega)$ the set of endpoints of $\B_j(\tilde h)$  is bounded for any fixed $j$.  We will show below that the set of endpoints of (at least) one of the two barcodes $\B_0(\tilde f), \B_1(\tilde f)$ is unbounded. This would then finish the proof in the case of $\S^2$.

Recall from Section \ref{sec:barcodes_for_hamiltonians}  that the lower ends of the bars in $\B_j$ are actions of orbits of index $j$, while the upper ends of the bars are actions of orbits of index $j+1$.  Combining this fact with Claim \ref{cl:endpoint_at_Ak}, we conclude that the value $c_k$ is the endpoint of a bar in at least one of $\B_{2k-1}(\tilde f)$, $\B_{2k}(\tilde f)$. Denote $E = \{k :  c_k \text{\;\ is an endpoint of } B_{2k}\}$ and $O = \{k :  c_k \text{\;\ is an endpoint of } B_{2k-1}\}$ and observe that at least one of these sets is infinite.  We will suppose that $E$ is infinite leaving the case where $O$ is infinite to the reader.

For each $k \in E$, we pick a capped 1--periodic  orbit $[z_k, u_k]$ of $G$ corresponding to $c_k$.  Here, $u_k$ is the unique (up to homotopy) capping of $z_k$ which is contained in the set $V$.  Let $A$ denote  the generator of $\pi_2(S^2)$ with $\omega (A) > 0$.  Consider the capped orbit $[z_k , u_k \# k A]$: As a consequence of Equation \eqref{eq:cZ_capping}, this orbit appears as an endpoint of a bar in $\B_0(\tilde f)$.  Furthermore, the action of this orbit is given by $c_k' := c_k - k \omega(A) = c_k -k$.  Recall that $c_k = 2 \sqrt{2\pi k}$ and thus  thus $c_k' \to - \infty$ as $k\to \infty$.  Therefore, the set of endpoints of $B_0(\tilde f)$ is not bounded.  This completes the proof.

\section{Hamiltonian homeomorphisms and smoothability}\label{sec:Ham-homeo-smoothable}
Recall from the introduction of the paper that we say a Hamiltonian homeomorphism $f$, with a finite number of fixed points, is \emph{smoothable} if there exists a Hamiltonian diffeomorphism $g$ which is arbitrarily $C^0$-close to $f$ and such that $Fix_{c}(g)=Fix_{c}(f)$.   We conjecture that every such $f$ is smoothable.  In this section we will prove that a very large class of homeomorphisms consists of smoothable ones. We will be relying on techniques from local dynamics of surface homeomorphisms.

Let $M$ be a surface,  $f:M \to M$ be an area-preserving homeomorphism, and $p$ an isolated fixed point of $f$. 
 We say that  $f$ is \emph{smoothable at $p$} if for every neighborhood $U$ of $p$ such that $Fix(f) \cap U = \{p\}$, there exists an area preserving homeomorphism $f': M \to M$ which coincides with $f$ outside $U$, which is smooth near $p$, and such that $Fix(g) \cap U = \{p\}$.  According to Proposition \ref{prop.local-global-smoothing}, if $f$ is smoothable at every fixed point then it is smoothable.

The \emph{local rotation set of $f$ around $p$} is a closed interval included in $[-\infty,+\infty]$ defined up to an integer translation, which is a conjugacy invariant (see~\cite{leroux13}, and section~\ref{subsubsec:rotation} below for more details). Assume furthermore that $f$ preserves the area. If $p$ is an isolated fixed point then the local rotation set contains no integer in its interior, and so it is included in $[0,1]$ (up to integer translation). We will say that $p$ is a \emph{maximally degenerate} fixed point if the local rotation set is equal to $[0,1] \text{ mod } \mathbb{Z}$. Maximally degenerate fixed points are accumulated by periodic orbits of every period, and have Lefschetz index $1$ (\cite{leroux13}, Section 3.2). Thus, the following theorem, combined with Proposition \ref{prop.local-global-smoothing}, immediately implies Theorem~\ref{theo:smoothable-easy} of the introduction.

\begin{theo}\label{theo:smoothable}
 Let $f: M \to M$ be an area preserving homeomorphism, and let $p$ be an isolated fixed point of $f$ which is not maximally degenerate.
Then, $f$ is smoothable at $p$.
\end{theo}

We split the proof into two cases.  Recall that, in the case of area preserving maps,  the Lefschetz index of an isolated fixed point is always less or equal to one  \cite{nikishin, simon, pelikanslaminka, lecalvez99}; see the discussion at the beginning of Section \ref{sec:local-homeo-area-pres} for a proof of this fact. 
We will use different strategies in the nonpositive index case and in the index one case. The two cases will be treated in Sections~\ref{subsec:negative-index} and~\ref{subsec:index-one}, respectively.

\subsection{Local \emph{vs} global smoothability}

The following result shows that a homeomorphism which is  smoothable at each fixed point is globally smoothable. 
\begin{prop}\label{prop.local-global-smoothing}
Let $f$ be an area preserving homeomorphism of a compact surface $M$ with a finite number of contractible fixed points.  Suppose that $f$ is smoothable at every contractible fixed point. Then, there exists an area-preserving diffeomorphism $g$, arbitrarily $C^0$ close to $f$, such that $\mathrm{Fix}_{c}(g) = \mathrm{Fix}_{c}(f)$. If furthermore $f \in \overline{Ham}(M, \omega)$, then $g$ can be chosen in $\Ham(M,\omega )$.
\end{prop}
The proof is a variation of the proof of the fact that, on surfaces, area-preserving diffeomorphisms are dense in the group of area-preserving homeomorphisms.  Hence, we will only provide a sketch of the proof and leave much of the details to the reader.

\begin{proof}
By definition of smoothability at a point, there exists a homeomorphism $f'$ such that $f'$ is $C^0$ close to $f$, $\mathrm{Fix}_{c}(f') = \mathrm{Fix}_{c}(f)$, and $f'$ is smooth near every contractible fixed point. 
Let $M'$ be a subsurface of $M$, whose complement in $M$ is a neighborhood of $Fix_{c}(f)$ on which $f'$ is smooth.  The first statement of the proposition is a consequence of the following claim.
\begin{claim}\label{cl:density_diffeos}
There exists an area-preserving diffeomorphism $g : M \rightarrow M$ which is arbitrarily $C^0$ close to $f'$ and coincides with $f'$ outside $M'$.
\end{claim} Note that $g$ has no new contractible fixed points since it remains $C^0$ close to $f'$ on $M'$ and $f'$ has no contractible fixed points in $M'$.

 We will finish the proof of the proposition before giving a proof of the above claim.  Assume that $f \in \overline{Ham}(M,\omega)$. It then follows that the flux of $g$ must be be small, and so we can perform a $C^0$-small modification of $g$, far from the fixed point set, turning it into a Hamiltonian diffeomorphism:  Indeed, consider some simple closed curve $\gamma$. We will explain how to modify $g$ so that its flux through $\gamma$ becomes zero and will leave the rest to the reader.  Let $\gamma'$ be a  simple closed curve which intersects $\gamma$ exactly once and which is far from the fixed point set. Denote by $\delta$  the flux of $g$ through $\gamma$. It is easy to construct a symplectic diffeomorphism $\varphi_{\delta}$, supported on a small tubular neighborhood of $\gamma'$, whose flux through $\gamma$ is equal to $-\delta$. Furthermore, since $\delta$ is small then $\varphi_{\delta}$ can be chosen to be $C^0$ close to the identity.  Replacing $g$ with $g\phi_{\delta}$ we obtain a map with zero flux through $\gamma$.  
\end{proof}
\begin{proof}[Proof of Claim \ref{cl:density_diffeos}]
There exists a diffeomorphism $g'$  which is  arbitrarily $C^0$ close to $f'$ and coincides with $f'$ outside $M'$. 
 This follows from applying the Handle Smoothing Theorem in \cite{hatcher2013kirby} successively to the vertices, edges, and faces of a sufficiently fine triangulation of $M$.  
 
 Let $\Omega = g'^*\omega$.  We will show that one can find a diffeomorphism $\Psi: M \rightarrow M$, arbitrarily $C^0$ close to the identity, such that $\Psi = \Id$ outside $M'$ and $\Psi^* \Omega = \omega$.  Of course, we will then set $g = g' \Psi$.

Let $\cal T$ be a triangulation of $M'$.  Using Moser's method, independently on a neighborhood of each vertex of $\cal T$, 
we can find a diffeomorphism $\Psi_{1}$ such that $\Psi_{1}^* \Omega = \omega$ near the vertices of $\cal T$. Likewise, 
we can find a diffeomorphism $\Psi_{2}$ such that $ \Psi_{2}^* \Psi_{1}^* \Omega = \omega$ near the 1-skeleton of $\cal T$. 
Furthermore, $\Psi_{1}$ and $\Psi_{2}$ can be picked to coincide with the identity outside $M'$ and to be as $C^0$-small as one wishes (independently of $\cal T$); we will leave it to the reader to check this latter statement.

\begin{claim}\label{cl:adjust_area}
Let $\Omega' = \Psi_{2}^* \Psi_{1}^* \Omega$.  There exists a function $\eta : M  \to \R$ such that: \begin{enumerate}
\item $\eta$ is uniformly close to the constant function $1$, 
\item $\eta = 1$ near the 1--skeleton of $\cal T$ and outside $M'$,
\item $\eta \Omega' (T) = \omega(T)$ for every triangle $T$ of the triangulation $\cal T$.
\end{enumerate}
\end{claim}
\begin{proof}
Observe that for every triangle $T$ of the triangulation, the ratio $\frac{\omega (T)}{\Omega'(T)} $  is close to $1$ because $\Omega'$ is the pull-back of $\omega$ by $g' \Psi_1 \Psi_2$ and $g' \Psi_1 \Psi_2$ is $C^0$ close to the area-preserving homeomorphism $f'$.  It follows that, for every triangle $T$, we can pick a function $\eta_T: M \rightarrow \R$ such that $\eta_T$ is uniformly close to the constant $1$, $\eta_T -1$ is supported in the interior of $T$, and $\int_{T}\eta_T \Omega'  = \int_{T}\omega$.  

Define the function $\eta$ by setting $\eta(x) = \eta_T(x)$ if $x$ is in the triangle $T$.  We leave it to the reader to check that  $\eta $ satisfies all of the stated properties.
\end{proof}
Observe that $\int_M \eta \Omega' = \int_M \Omega'$.  Therefore, by applying Moser's method, we find $\Psi_3$ such that $\Psi_3^* \Omega' = \eta \Omega'$.  Because $\eta$ is uniformly close to $1$,  the diffeomorphism $\Psi_3$ can be picked to be $C^0$ close to the identity; see Proposition 5 of \cite{EPP}.  Furthermore, it coincides with the identity outside $M'$.  

Next, applying the Moser method again, we find a diffeomorphism $\Psi_4$ such that $\Psi_4^* \; \eta \Omega' = \omega$.  Note that $\eta \Omega' = \omega$ near the 1--skeleton of the triangulation and outside $M'$.  Thus, $\Psi_4$ may be picked such that it is supported in the union of the interiors of the triangles of $\cal T$.    Thus, $d_{C^0}(\Psi_4, Id)$ is bounded by the maximum of the diameters of the triangles of $\cal T$ and so by picking a sufficiently fine triangulation, $\Psi_4$ can be chosen as $C^0$--small as we wish.  Note that since $\eta \Omega' = \omega$ outside $M'$ we can ensure that $\Psi_4 = \Id$ outside $M'$.  Finally, we $\Psi = \Psi_1 \Psi_2 \Psi_3 \Psi_4$.  

\end{proof}

\subsection{Transverse foliations for local area-preserving homeomorphisms}
\label{section:transverse-foliations}
We begin this section by introducing some notions, as well as notations, from the theory of transverse foliations.  For further details, we refer the reader to~\cite{lecalvez08}.
Let $\Sigma$ be a surface, and fix $p_{0}$ to be a point of $\Sigma$. We will call $(f,U)$ a \emph{local homeomorphism} if  $f:U \to f(U)$ is a homeomorphism between some open subsets $U,f(U)$ of $\Sigma$ that fixes $p_{0}$. We will always assume tacitly that $U$ and $f(U)$ are interiors of some closed topological disks in $\Sigma$. 
We will often assume that $p_{0}$ is an isolated fixed point, in which case, by diminishing $U$ if necessary, we get a local homeomorphism for which $p_{0}$ is the only fixed point.


A \emph{local isotopy} $(I,V)$ for $(f,U)$ is a continuous family $I=(f_{t})_{t \in [0,1]}$ of local homeomorphisms $(f_{t}, U_{t})$ such that  $f_{t}$ fixes  $p_{0}$ for each $t$,  $V \subset U_{t} \subset U$, $f_{0}$ is the identity, and $f_{1}=f$ on $U_{1}$. We will say that $(I,V)$ is \emph{compactly supported in $U$} if $U_{t} = U$ for every $t$ and $f_{t}$ is the identity near the boundary of $U$.
If $(J,V')$ is another local isotopy with $J = (g_{t})_{ t \in [0,1]}$, with $g_{1}(V') \subset V$, then we may define the local isotopy $(IJ,V')$ with the concatenation $IJ=(h_{t})_{t \in [0,1]}$ defined by $h_{t} = g_{2t}$ when $t \in [0,1/2]$ and $h_{t} = f_{2t-1} g_{1}$ when $t \in [1/2,1]$.

A curve is a continuous map $\gamma:[0,1] \to \Sigma$. The curve is closed if $\gamma(0)=\gamma(1)$. Given a local isotopy $(I,V)$ and a point $p$ in $V$, the trajectory of $p$ is the curve $I.p: t \mapsto f_{t}(p)$, where $I=(f_{t})_{t \in [0,1]}$. 

\subsubsection{Transverse foliations}\label{sec:trans-fol}
A \emph{local foliation} $(\cF,W)$ is a smooth ($C^\infty$) oriented foliation $\cF$ defined on $W \setminus \{p_{0}\}$, where $W$ is some disk neighborhood of $p_{0}$; we will say that  $p_{0}$ is the \emph{singularity} of $\cF$; note that no regularity is required at $p_{0}$. A curve included in $W\setminus \{p_{0}\}$ is \emph{positively transverse to the foliation} if it crosses every leaf it meets from left to right: for every $t_{0} \in [0,1]$ there is a chart $\Psi$ for the foliation, defined on some neighbourhood of $\gamma(t_{0})$ and with values in the plane, that sends the foliation to the foliation by vertical lines oriented from bottom to top, and such that the first coordinate of $\Psi \circ \gamma$ is an increasing function on some neighbourhood of $t_{0}$. Following Le Calvez, we say that a local isotopy $(I,V)$ and a local foliation $(\cF,W)$ are \emph{dynamically transverse} if there exists a neighborhood $V' \subset V \cap W$ of $p_{0}$ such that for every point $p\in V' \setminus \{p_{0}\}$,  the trajectory $I.p$ is homotopic in $W \setminus \{p_{0}\}$ to a curve $\gamma$ which is positively transverse to $\cF$ (see~\cite{lecalvez08}, Section 3, where this property is called ``localement dynamiquement transverse''; here we drop the word ``locally'' since everything is local). 
We should add that occasionally we will use topological (non-smooth) foliations; for precise definitions we refer to \cite{lecalvez05}. Most of the time, in our context, there is not much difference between using topological or smooth foliations, but there will be one point for which the smoothness will be crucial (see the normal form lemma~\ref{lemma-normal-form-foliation} below).

Recall the definition of the Poincar\'e-Lefschetz index $L(f,p_{0})$ from the introduction.  We denote by  $L(\cF,p_{0})$ the Lefschetz index of the foliation $\cF$ at $p_0$:  this is simply the Lefschetz index of the a vector field which is tangent to $\cF$ and vanishes at $p_0$.  The index $L(\cF,p_{0})$ is often referred to as the Poincar\'e-Hopf index. 

\begin{theo}\label{theo.locally-transverse-foliation}
Let $(f,U)$ be a local homeomorphism with an isolated fixed point $p_{0}$.
Let $(I,V)$ be a local isotopy for $(f,U)$. Then, there exists a local foliation $(\cF,W)$ which is dynamically transverse to $(I,V)$.
Furthermore, 
\begin{itemize}
\item if $L(f,p_{0})=1$ then $L(\cF,p_{0})=1$,
\item if $L(f,p_{0})\neq1$ then we can choose the isotopy $(I,V)$ such that $L(\cF,p_{0})=L(f,p_{0})$.
\end{itemize}
\end{theo}
\begin{proof}
We will first  prove the statement concerning the indices $L(f,p_{0})$  and  $L(\cF,p_{0})$, assuming the existence of the dynamically transverse local foliation. The proof requires passing through the notion of the index of an isotopy $L(I, p_{0}).$  This index is defined in \cite{leroux04}.  Proposition 3.2 of  \cite{lecalvez08}  proves that the index of the transverse foliation is the same that the index of the isotopy, \emph{i.e.}\ $L(\cF,p_{0})= L(I, p_{0}). $   Furthermore, according to Proposition 4.7 of \cite{leroux13} if $L(f,p_{0})=1$, then  $L(I, p_{0})=1 $ and if $L(f,p_{0})\neq1$, then we can choose the isotopy $I$ such that $L(I,p_{0})=L(f,p_{0})$.  

It remains to prove the existence of the location foliation  $(\cF,W)$ which is dynamically transverse to $(I,V)$.    Without loss of generality, we may suppose that  $U \subset \R^2$.  By Appendix A of \cite{leroux13}, we may also assume that the homeomorphism $f$ is defined on the entire plane $\R^2$ and that $p_0$ is the only fixed point of $f$ in $\R^2$. We pick an isotopy $I' =(f_t)_{t \in [0,1]}$ of $\R^2$ such that $f_0 = Id$, $f_1 =f$, and $f_t(p_0)= p_0$ for all $t \in [0,1]$.  Furthermore, the isotopy $I'$ may be picked such that for every point $p\in V$ the trajectories $I'.p$ and $I.p$ are homotopic relative to endpoints in $V \setminus \{p_0\}$:  Indeed, this can be achieved by replacing $I'$ with $I' R^q$, where $q$ is an integer and  $R$ is a full rotation of the plane around the point $p_0$. 

According to Le Calvez \cite{lecalvez05}, there exists a \emph{topological}, \emph{i.e.}\ not necessarily smooth, foliation $\cF$ of $\R^2 \setminus \{p_0\}$ which is dynamically transverse to the isotopy $I'$.  A priori, this foliation is only globally transverse to the isotopy  $I'$: for every point $p \in \R^2 \setminus \{p_0\}$  the trajectory $I'.p$ is a homotopic to a curve $\gamma$ which is transverse to the foliation  $\cF$.   Now, pick $W$ to be any sufficiently small neighborhood of $p_0$ and consider the local foliation $(\cF, W)$.  By Proposition 3.4 of \cite{lecalvez03}, global transversality of $I'$ and $\cF$ implies that $(I',V)$ and $(\cF, W)$ are locally transverse.   Since we picked $I'$ such that for every point $p\in V$ the trajectories $I'.p$ and $I.p$ are homotopic relative to endpoints in $V \setminus \{p_0\}$, we see that $(I,V)$ and $(\cF, W)$ are dynamically transverse as well.    

We are not completely done yet because, as mentioned in the previous paragraph, the article  \cite{lecalvez05} of Le Calvez provides us with a foliation $\cF$ of $\R^2 \setminus \{p_0\}$ which is a priori non-smooth.  We will now outline an  argument which will allow us to perturb $\cF$ to a smooth foliation.  According to Proposition 3.3 of \cite{leroux13}, on any surface,  the set of foliations which are transverse to a given isotopy forms an open subset of the set of all foliations, where the set of all foliations is equipped with the Whitney topology.   For the precise definition of the Whitney topology please see Section 3 of \cite{leroux13}.  Now, the natural argument would be to prove  that the set of smooth foliations forms a dense subset of the set of all foliations and then it would follow that we may pick  our foliation to be smooth.  However, we have not been able to find a proof of the density of smooth foliations in the literature and so we will prove that in the very specific settings of our article the foliation $\cF$ may be perturbed to a smooth foliation.

In this article we will only rely on Theorem \ref{theo.locally-transverse-foliation} when $(f,U)$ is area-preserving.  As a consequence, and as explained  at the beginning of Section \ref{sec:local-homeo-area-pres}, the transverse foliation $(\cF, W)$ must be gradient-like.   This implies, according to Appendix B of \cite{leroux13},  that $(\cF, W)$  is locally homeomorphic to a smooth foliation:  Up to possibly shrinking the set $W$, we can a find smooth local foliation, say $(\cF', W')$, and a homeomorphism $\phi: W'\setminus \{p_0\}  \rightarrow W\setminus \{p_0\}$ which maps $\cF'$ to $\cF$.   As we will explain in the next paragraph, diffeomorphisms form a dense subset of homeomorphism, with respect to the Whitney topology, and so we may find a diffeomorphism $\psi: W'\setminus \{p_0\}  \rightarrow W\setminus \{p_0\}$ which is arbitrarily close to $\phi$ in the Whitney topology.  As a consequence,   the smooth foliation $\psi(\cF')$ will be Whitney close to $\cF$ and so (up to shrinking $W$) we may replace $(\cF, W)$ with $\psi(\cF', W')$.

  It remains to explain why diffeomorphisms form a dense subset of homeomorphisms with respect to the Whitney topology.  We will only provide a brief sketch of the argument as it is very similar to the usual argument for proving that, in the case of surfaces, diffeomorphisms are dense in homeomorphisms with respect to the uniform topology.
 
 Let $\Sigma$ be any surface.   Let us recall the definition of a Whitney neighborhood of a homeomorphism $\phi$:  Consider an open cover $(U_i)$ of $\Sigma$ which is locally finite, \emph{i.e.}\ no point is contained in infinitely many of the $U_i$'s.  To each $U_i$ we associate $\varepsilon_i > 0$.  A basic open neighborhood of $\phi$ is given by the set of homeomorphisms  $\theta: \Sigma \to \Sigma$ such that for each $i$ we have $d(\phi(p), \theta(p)) < \varepsilon_i$ for all $p \in U_i$.  Now, we will show that every basic open neighborhood of $\phi$ contains a diffeomorphism.  Given a basic open neighborhood $\phi$ as above, we pick a smooth triangulation of $W$ satisfying the following two criteria:   First, each triangle $T$ is contained in at least one of the sets $U_i$ and second, if $T$ is contained in $U_i$ then the diameter of $\phi(T)$ is very small compared to $\varepsilon_i$. One can then apply the Handle Smoothing Theorem of \cite{hatcher2013kirby}, to obtain a diffeomorphism $\psi$ by smoothing out $\phi$, successively, near the vertices, edges, and finally faces of the triangulation.  This can be done such that for each point $p$ the distance between $\phi(p)$ and $\psi(p)$ is controlled by the diameter of the triangle(s) of the triangulation which contain the point $p$. 
\end{proof} 

\subsubsection{$(\cF, W)$ is gradient-like when $(f, U)$ is area-preserving}\label{sec:local-homeo-area-pres}  We equip the surface $\Sigma$ with a symplectic form $\omega$ and denote by $\mu_0$ the measure induced on $\Sigma$.  For the rest of the article, unless otherwise stated, it will be our standing assumption that $(f, U)$ is area-preserving, meaning that the push-forward measure $f_* \mu_0$ coincides with $\mu_0$.  As we will explain below, an important consequence of this assumption is that the foliation given by Theorem \ref{theo.locally-transverse-foliation} will be \emph{gradient-like}.

 Let $(I,V)$ a local isotopy for $(f,U)$, and $(\cF,W)$ a local (smooth or topological) foliation dynamically transverse to $(I,V)$.
Let $\mathbf{\ell}$ be a closed leaf of $\cF$ (that is, a leaf homeomorphic to a topological circle) and $D$ be the topological open disk bounded by $\mathbf{\ell}$. A curve topologically transverse to $\cF$ meets $\mathbf{\ell}$ at most once, thus by transversality we have either $f(\overline{D}) \subset D$ or $\overline{D} \subset f(D)$. This contradicts the fact that $f$ preserves the area. Thus, $\cF$ has no closed leaf. By Poincar\'e-Bendixson theory, for every oriented leaf $\mathbf{\ell}$, the $\omega$-limit set of $\gamma$ is either empty or equal to the singularity $\{p_{0}\}$ (that is, $\gamma$ either hits the boundary of $U$ or the singularity). The $\alpha$-limit set shares the same properties. Furthermore, transversality implies that the $\alpha$ and $\omega$-limit sets cannot be both equal to $\{p_{0}\}$ because this would (as in the case of a closed leaf)  lead to the existence of a disk $D$ such that$f(\overline{D}) \subset D$ or $\overline{D} \subset f(D)$. Hence we have three kinds of oriented leaves, namely those which go from the boundary (of $W$) to the boundary, from the boundary to the singularity, or from the singularity to the boundary. A local foliation with only these kinds of leaf is called \emph{gradient-like} (see~\cite{lecalvez05}). 

\subsection*{Classification} Local gradient-like foliations may be classified up to homeomorphisms: $(\cF, W)$ consist of \emph{hyperbolic sectors} and \emph{parabolic sectors} (see the appendix of~\cite{leroux13} for more details).     
{
A \emph{sector} $S$ is a subset of $W$ which is homeomorphic to the closed unit disc and whose boundary contains $p_{0}$.
If $\alpha,\beta$ are two curves of which $p_{0}$ is an end-point, we will say that \emph{$S$ is between $\alpha$ and $\beta$}, and we write $S=S(\alpha,\beta)$,
if there exists a curve $\gamma$ such that the concatenation of $\alpha,\gamma,\beta$ in that order is a simple closed curve that parametrizes the boundary of $S$, in such a way that $S$ is on the left-hand side of this curve.
Consider a sector between two pieces of leaves of the foliation $\cF$.
}
  The sector is called \emph{parabolic} if the restriction of the foliation is homeomorphic to the foliation of an angular sector by radial lines. In a parabolic sector either every leaf has its $\omega$-limit set equal to $\{0\}$, or every leaf has its $\alpha$-limit set equal to $\{0\}$; we call the parabolic sector \emph{positive} in the first case and \emph{negative} in the latter case.
A sector is called \emph{hyperbolic} if it is homeomorphic to the foliation of $\{x \geq 0, y \geq 0\}$ in the plane by the hyperbolae $xy= \text{constant}$. The topological classification of local foliations states that any gradient-like local foliation is homeomorphic to a foliation obtained by gluing together a finite number of hyperbolic and parabolic sectors. Furthermore the number of hyperbolic sectors is given by $N=2(1-L(\cF,p_{0}))$.  Thus, taking into account the orientation, there are exactly two topological types of gradient-like local foliation of index one (called sinks and sources). For indices different than one, there  are $2^N$ foliations of a given index $L(\cF,p_{0})$: this is because between any two adjacent hyperbolic sectors one can choose to add a parabolic sector or not.   Observe that  the index of a gradient-like foliation is at most one.  This, combined with Theorem \ref{theo.locally-transverse-foliation}, implies that the index of a fixed point of an area-preserving homeomorphism is at most one.  

Note that the existence of  a gradient-like foliation transverse to a local isotopy has interesting dynamical consequences on the local homeomorphism $(f,U)$. For instance, (1) each hyperbolic sector is either locally attractive ($f(S \cap W') \subset S$ for some smaller neighbourhood $W'$ of the fixed point) or repulsive; (2) if the foliation is a sink, then the local rotation set is included in $[0, \infty]$. Both remarks will be crucial in what follows (see the proof of Lemma~\ref{lem:smoothable-negative2} and the beginning of section~\ref{subsec:index-one}).

\bigskip

We will be needing the following lemmas.  The first of these is not difficult and so we leave the proof to the reader.
\begin{lemma}
\label{lemma-locally-transverse}
Let $(\cF,W)$ be a local foliation and assume that  there exists a closed curve  $\gamma$  included in $W \setminus \{p_{0}\}$ which is positively transverse to $\cF$. Then,  $\gamma$ is not contractible in $W \setminus \{p_0\}$.  Furthermore, if $\cF$ is gradient-like, then it must be  a sink or a source.
\end{lemma}

The next lemma tells us that we can find well-adapted area-preserving charts on a neighborhood of any parabolic sector.   We should point out that the smoothness of the foliation is crucial here.
\begin{lemma}[Normal Form for Parabolic Sectors]
\label{lemma-normal-form-foliation}
Let $(\cF, W)$ be a local foliation with a parabolic sector $S$, and $O$ an open subset of $W$ containing $S \setminus \{p_{0}\}$.
 Then, there exists  some open set $O'$ with $S \setminus \{p_{0}\} \subset O' \subset O$, and an area-preserving diffeomorphism $\Phi: O' \rightarrow \Phi(O') \subset \R^2 \setminus \{(0,0)\}$  with the following properties:
\begin{enumerate}
\item  $\Phi(O')\cup \{(0,0)\}$ contains $\{0\} \times [-\varepsilon,\varepsilon]$ for some positive $\varepsilon$,

\item $\Phi$ sends the edges of $S$ to Euclidean rays respectively included in 
$$\{(x,y) \in \R^2: y=0, x>0\} \text{ and } \{(x,y): y=x, x>0\},$$

\item $\Phi_* \mu_0 = \Leb$, where $\mu_0$ is the measure induced by the symplectic form and $\Leb$ is the Lebesgue measure on $\R^2$,

\item on a neighborhood of the origin in the half plane $\{(x, y) \in \R^2: x \geq 0 \}$, the foliation $\Phi(\cF)$ is transverse to the foliation of the half  plane by vertical lines.
\end{enumerate}
\end{lemma}  

\begin{proof}
Our proof consists of two steps: In the first step, we will construct $\Phi$ on the parabolic sector $S$ such that it satisfies all the requirements of the lemma on $\Phi(S)$.  In the second step,  we will extend $\Phi$ beyond the edges of $S$.

\medskip

\noindent \textbf{Step 1: Constructing $\Phi$ on $S$.}  Let $T_1$ be the triangle $\{(x,y) \in \R^2:0 \leq x \leq 1, 0 \leq y \leq x\}$.  We begin with the following claim.
\begin{claim}\label{claim.radial-foliation}
There exists a homeomorphism $\Phi':S \rightarrow T_1$ such that $\Phi'(\cF)$ is the radial foliation of $T_1$ whose leaves are the lines $y=kx$, for $k \in [0,1]$.   Furthermore, $\Phi':S \setminus \{p_0\} \rightarrow T_1 \setminus \{(0,0)\}$ is a diffeomorphism.
\end{claim}
\begin{proof}
 Denote the edges of the parabolic sector by $\partial^-S$ and $\partial^+S$.  Let $\gamma:[0,1] \rightarrow W$ be a smooth curve which is transverse to the parabolic sector $S$ such that $\gamma(0) \in \partial^-S$ and $\gamma(1) \in \partial^+S$.  Such $\gamma$ exists for the following reason:  The fact that $S$ is topologically equivalent to a standard sector with its radial foliation implies that one may pick such $\gamma$ which is continuous.  One can then smooth out $\gamma$ by an isotopy which preserves the leaves of the foliation  (recall that the foliation is smooth on $W \setminus \{p_{0}\}$). 

Consider the segment of the leaf of the foliation which starts at $\gamma(s)$ and ends at $p_0$; denote its length (which could be $+\infty$) by  $\ell(s)$.  Let us show that $S \setminus \{p_0\}$ may be parametrized by the set $U =  \{(t,s) \in \R^2 : 0 \leq t < \ell(s), 0 \leq s \leq 1\}$:  pick $X$ to be the unit-length vector field defined on $S \setminus \{p_0\}$ which is tangent to the leaves of $\cF$ such that $X$ points into the sector $S$ at every point on $\gamma$.    Denote the flow of $X$ by $\varphi^t_X$ and define $\Psi_1 : U \rightarrow S \setminus \{p_0\}$ by $\Psi_1 (t,s) = \varphi^t_X(\gamma(s))$.  This is a diffeomorphism from $U$ to $S \setminus \{p_0\}$ which sends the foliation of $U$ by horizontal lines to the restriction of $\cF$ to the sector $S$.

We leave it to the reader to check that one can find a  diffeomorphism  $\Psi_2:  T_1 \setminus \{(0,0)\} \rightarrow U$ which sends the the radial foliation of $T_1$ to the horizontal foliation of $U$.  Define  $\Phi' = (\Psi_1 \Psi_2) ^{-1}: S \setminus \{p_0\} \rightarrow T_1 \setminus \{(0,0)\}$ and set $\Phi'(p_0) = (0, 0)$.  It is easy to check this is our desired map.
\end{proof}

 Observe that, because  $\Phi'$ is smooth away from $p_0$, we can apply the change of variables formula to conclude that  the measure $\mu = \Phi'_{*} \mu_{0}$ is of the form $\eta \, \mathrm{Leb}$, where $\eta$ is a smooth function on $T_1 \setminus \{(0,0)\}$.  We will  adjust the coordinates \emph{radially} such that the measure $\eta \, \Leb$ is sent to a new measure  $\rho\mathrm{Leb}$ with  the property  that the vertical projections of $\rho\mathrm{Leb}$ and $\Leb$ onto the $x$--axis coincide.  To achieve this, for each $X \in (0,1]$ consider the triangles $T_{X}=\{ 0 \leq x \leq X, 0 \leq y \leq x\}$ and  define 
$
H(X) = (2 \mu(T_{X}))^\frac{1}{2}.
$
This is a homeomorphism between $(0,1]$ and some interval $(0, a]$. Now, note that $x$ and the angular coordinate $\theta$ together define a smooth system of coordinates on $x>0$.
Consider the map 
 $$
 \Phi''(x,\theta)= (H(x),\theta).
 $$
This is a homeomorphism between the triangles $T_{1}$ and $T_{a}$ which preserves the radial foliation.  Moreover, it is smooth away from the origin and so it sends the measure $\eta \Leb$ to a measure of the form $\rho \Leb$ where $\rho$ is a function on $T_{a}$ which is defined everywhere except possibly at the origin.  Furthermore,
$$
 \Phi''(T_{X}) = T_{H(X)} \text{ and }  \Leb(T_{H(X)}) = \mu(T_{X})
$$
for every $X \in (0,1]$, and thus
the projections of $\rho \Leb$ and $\Leb$ onto the first coordinate are the measure $xdx$. In other words,
$
\int_{0}^x \rho(x,y) dy = x
$
 for every $x \in (0, a]$.
Now, define $\Phi'''$ by the formula 
$$
\Phi'''(x,y) = \left(x,\int_{0}^y \rho(x,t)dt\right).
$$
This map leaves the triangle $T_{a}$ invariant and sends the measure
$\rho\mathrm{Leb}$ to the Lesbegue measure. Furthermore, it preserves the vertical foliation and thus it maps the radial foliation to a foliation which is transverse to the vertical one.   

The homeomorphism $\Phi=\Phi'''\Phi''\Phi' : S \mapsto T_{a}$ is an area preserving chart which sends $\cF$ to a foliation transverse to the vertical foliation.  This completes the first step of our construction.

\medskip

\noindent \textbf{Step 2: Extension of $\Phi$ beyond the edges of $S$.} 

Let $\partial^+ S, \partial^- S$ denote the two leaves of the foliation which are at the boundary of the parabolic sector $S$.  The map $\Phi$ constructed above sends $\partial^+ S$ and $\partial^-S$, diffeomorphically, to line segments which have an  endpoint at the origin.  By the following lemma $\Phi$ may be extended to a neighborhood $O'$ of $S$ in $W \setminus \{p_0\}$ such that all the requirements of Lemma~\ref{lemma-normal-form-foliation} are satisfied.  

\begin{lemma}\label{lem:normal_form_near_leaf}
 Suppose that $(\cF, W)$ has a leaf $F$ whose $\omega$-limit, or $\alpha$-limit, set is $p_0$. Let $R$ be a line segment in $\R^2$ of the form  $R= \{(x, kx): x \in (0,1)\}$, for some $k \in \R$. Let  $\Phi:F \to R$ be a smooth diffeomorphism. There exists  an open set $U$ containing $F$ and included in $W\setminus \{p_0\}$, and an extension of $\Phi$ to an area-preserving  diffeomorphism $\Phi: U \to \Phi(U) \subset \R^2$  such that 
 \begin{enumerate}
\item  $\Phi(U)\cup \{(0,0)\}$ contains the segment $\{0\} \times [-\varepsilon,\varepsilon]$ for some positive $\varepsilon$,
 \item On the half-plane  $\{(x,y) \in \R^2: x \geq 0\}$, the leaves of $\Phi(\cF)$ are transverse to the foliation given by vertical lines.
 \end{enumerate} 
\end{lemma}

\begin{proof}
{The proof is very similar to the proof of Step 1, we will just sketch it and leave the details to the reader.
We preliminary note that the linear map $(x,y) \mapsto (x, y - k x)$ preserves the Lebesgue measure and the vertical foliation and sends the ray $R$  into the $x$ axis; using this map we can assume that $k=0$.   
Using some curve which passes through the leaf $F$ and which is transverse to the foliation, we first extend $\Phi$ into a diffeomorphism $\Phi'$ that  sends $\cF$ to the foliation by horizontal lines, and whose image contains a segment $\{0\} \times [-\varepsilon,\varepsilon]$ for some positive $\varepsilon$; the construction is similar to the proof of Claim~\ref{claim.radial-foliation}.
Then we post-compose $\Phi'$ by a diffeomorphism $\Phi''$ which is the identity on $\Phi'(F)$, and sends the measure $\Phi'_{*} \mu_{0}$ to the Lebesgue measure; this construction is similar to the end of Step 1.}

\end{proof}
This completes the  second step of our construction.
\end{proof}

\subsection{Smoothability of fixed points of nonpositive index }\label{subsec:negative-index}
In this section, we will prove that a fixed point with nonpositive index is smoothable.
Before going into the proof, let us reduce the problem to a purely local one. Let $(f,U)$ be a local homeomorphism for which $p_{0}$ is the only fixed point. We will say $(f,U)$ is \emph{smoothable} if there exists a local area-preserving homeomorphism $(g,U)$ for which $p_{0}$ is the only fixed point, which is smooth near $p_{0}$ and which coincides with $f$ near the boundary of $U$. Paradoxically enough, we note that this notion is invariant under topological conjugacy. Indeed, assume $(f,U)$ is smoothable and let $(g,U)$ be as above.  Suppose that $\Psi: U \to U'$ is a homeomorphism, where $U'$ is an open subset of some surface, and consider $(\Psi f \Psi^{-1},U')$.  By the handle smoothing lemma there is a homeomorphism $\hat \Psi$ that coincides with $\Psi$ near the boundary of $U$ and that is smooth near $p_{0}$ (\cite{hatcher2013kirby}). Then,  $(\hat \Psi g \hat \Psi^{-1},U)$ is smooth near $p_{0}$ and coincides with $\Psi f \Psi^{-1}$ near the boundary of $U'$, showing that $(\Psi f \Psi^{-1},U')$ is also smoothable.
We conclude  that Theorem~\ref{theo:smoothable} is equivalent to the following purely local statement: \emph{every area preserving local homeomorphism $(f,U)$ with a single fixed point $p_{0}$ which is not maximally degenerate is smoothable.}
We will now prove that statement in the case when the index $L(f,p_{0})$ is nonpositive.    We will first give an outline of the proof relying on several technical lemmas which will be proved later.

\begin{proof}[Proof of Theorem~\ref{theo:smoothable} in the nonpositive index case]
Let $(f,U)$ be an area preserving local homeomorphism with a single fixed point $p_{0}$, with $L(f,p_{0}) \leq 0$.
According to theorem~\ref{theo.locally-transverse-foliation} there is a local isotopy $(I,V)$ for $(f,U)$ and a local foliation $(\cF,W)$ dynamically transverse to $(I,V)$, such that  $L(\cF, p_{0}) = L(f,p_{0})$. By applying Proposition~\ref{pro:smoothable-negative1} below we only need to consider the case when the foliation $\cF$ has no parabolic sectors. 

\begin{prop}\label{pro:smoothable-negative1}
Let $(f,U)$ be an area preserving local homeomorphism with a single fixed point $p_{0}$ of index $L(f,p_{0}) \leq 0$. 
There exists an area preserving local homeomorphism $(f',U)$ with a single fixed point $p_{0}$, that coincides with $f$ near the boundary of $U$,
and an isotopy $(I',V)$ for $(f',U)$ admitting a dynamically transverse local foliation that has the same index $L(f,p_{0})$ and has no parabolic sectors.
\end{prop}

To fix ideas, from now on we assume $L(f,p_{0})=-1$ (see the comments at the end of the proof for the other cases). 
We define the \emph{model dynamics} to be the linear map $m_{-1} : (x,y) \mapsto (2x, y/2)$ on the plane $\mathbb{R}^2$.
Let $D$ be a closed topological disk included in $U$ and containing the fixed point $p_{0}$ in its interior. Assume $f(D)$ is included in $U$. We say that $D$ is \emph{in canonical position for $f$} if there is a continuous injective mapping $\Phi: U \to \mathbb{R}^2$ taking $D$ to the unit disk and such that
$\Phi f \Phi^{-1} = m_{-1}$ on the boundary of the unit disk (see~\cite{pelikanslaminka,bonino01}). Note that this is a purely topological definition (no measure involved). The next lemma is purely topological, while the following one will take care of the area.

\begin{lemma}\label{lem:smoothable-negative2}
There exists a disk $D$ in canonical position for $f$.
\end{lemma}

We consider the standard area on the plane. It is preserved by $m_{-1}$.
\begin{lemma}\label{lem:smoothable-negative3}
Let $D$ be a disk in canonical position for $f$. Then 
there exists a continuous injective mapping $\Psi: U \to \mathbb{R}^2$ such that 
\begin{itemize}
\item $f = \Psi^{-1} m_{-1} \Psi $ holds on $\partial D$,
\item $\Psi$ sends the area on $U$ to the standard area,
\item $\Psi$ is smooth near $p_{0}$ and sends $p_{0}$ to $0$.
\end{itemize}
\end{lemma}
 Let $D$ be the disk given by the first lemma, and  $\Psi$ be given by the second lemma. Define $g$ as the homeomorphism that coincides with $f$ on $U \setminus D$ and with $\Psi^{-1} m_{-1} \Psi$ on $D$. Then $g$ has only one fixed point and is smooth near $p_{0}$, as wanted. This concludes the proof in the case when $L(f,p_{0}) = -1$.

The proof when $L(f,p_{0})$ is any non-positive number $p$ is entirely similar. 
Instead of $m_{-1}$, one uses a model map $m_{p}$ of index $p$ which is smooth and is made of $2(1-p)$ hyperbolic sectors (obtained, for instance, by integrating an autonomous hamiltonian function with a degenerate critical point).  Likewise, the proofs of Lemma~\ref{lem:smoothable-negative2} and~\ref{lem:smoothable-negative3} are entirely similar for every nonpositive value of the index, and so, for expository reasons, they will be described only in the $-1$ case. 
\end{proof}

\begin{figure}[h!]
\centering
\def\svgwidth{0.8\textwidth}
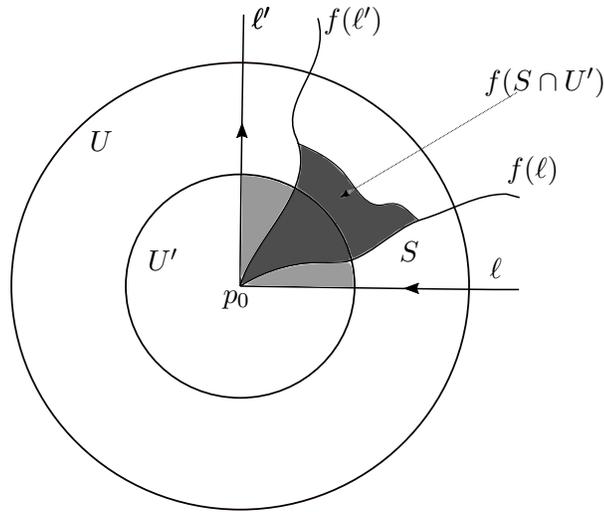
\caption{$f(\ell), f(\ell')$ are on the right-hand sides of  $\ell, \ell'$, respectively.  As a consequence $f(S\cap U') \subset S$.}
\label{fig:canonical1}
\end{figure}

\begin{figure}[h!]
\centering
\def\svgwidth{1.1\textwidth}
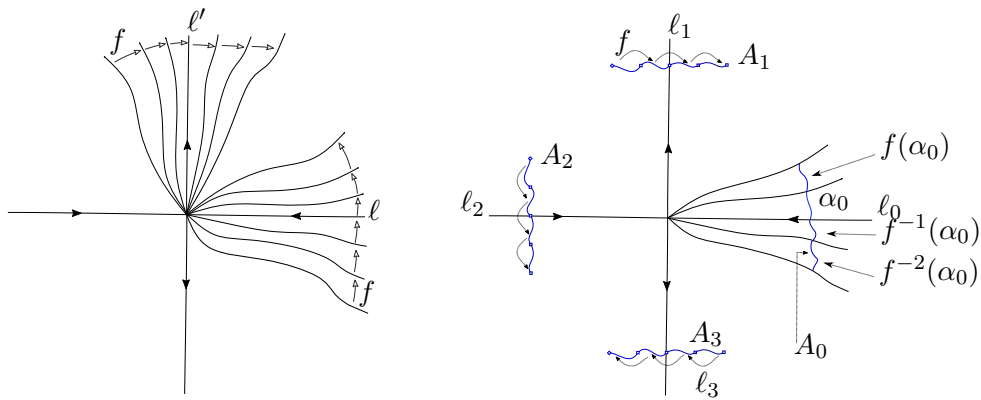
\caption{The figure on the left shows the cyclic order of the images of $\ell, \ell'$ under the iterates of $f$. The figure on the right depicts the construction of the curves $A_0, A_1, A_2, A_3$.}
\label{fig:canonical2}
\end{figure}

\begin{figure}[h!]
\centering
\def\svgwidth{0.8\textwidth}
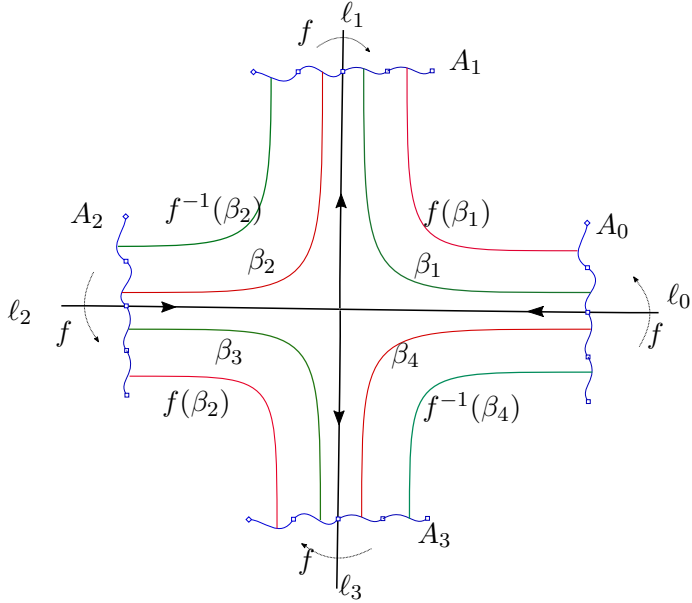
\caption{$\beta_1, \ldots, \beta_4$ are leaves of the transverse foliation.  The simple closed curve $C$ is the union of $\beta_1$, a segment along $A_1$, $f^{-1}(\beta_2)$, a segment along $A_2$, $\beta_3$, a segment along $A_3$, $f^{-1}(\beta_4)$, and a segment along $A_4$.  In the above picture, $C$ consists  of the green curves and segments along the blue curves.  The image $f(C)$ consists of the red curves and segments along the blue curves. }
\label{fig:canonical4}
\end{figure}

\begin{figure}[h!]
\centering
\def\svgwidth{1.0\textwidth}
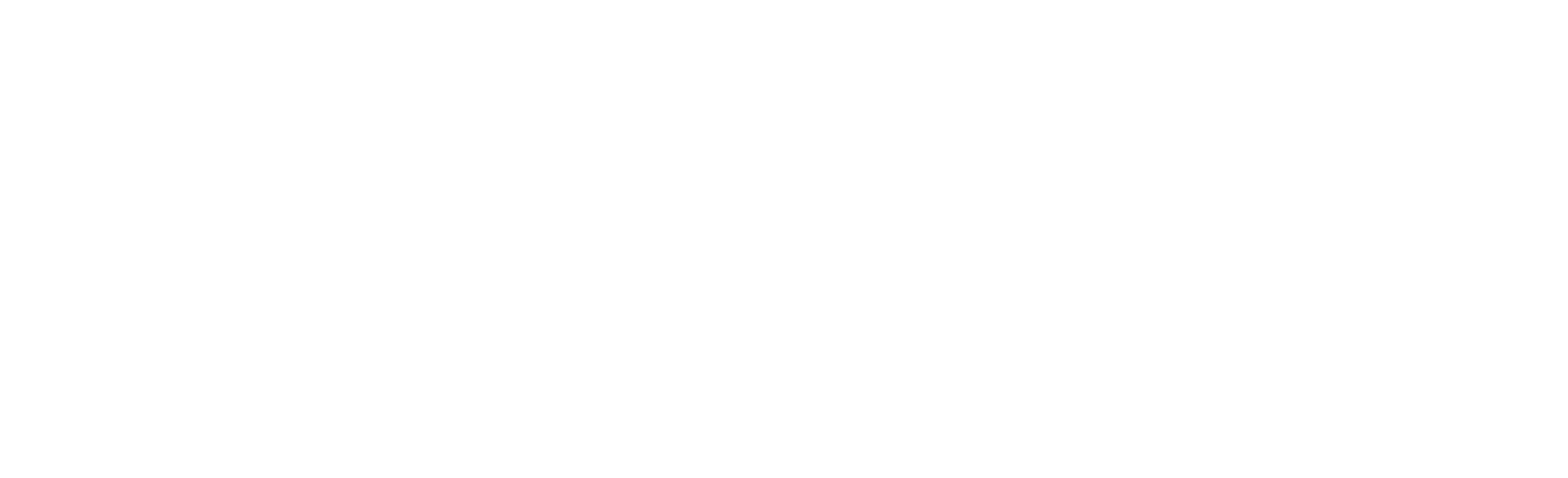
\caption{In the above image $P_i', Q_i'$ denote the preimages of $P_i, Q_i$ under $f, m_{-1}$, respectively. 
The curve $C'$ is in canonical position for $f$.  Note that the cyclic order of the points $P_0', P_0, P_1 , P_1', P_2', P_2, P_3, P_3'$ on $C'$ matches that of the points $Q_0', Q_0, Q_1 , Q_1', Q_2', Q_2, Q_3, Q_3'$.}
\label{fig:canonical5}
\end{figure}

\begin{proof}[Proof of Lemma~\ref{lem:smoothable-negative2}]
Remember that the isotopy $(I,V)$ is dynamically transverse to the foliation $(\cF,W)$, which by Proposition \ref{pro:smoothable-negative1} has no parabolic sectors, and whose Lefschetz index is $-1$: there exists a neighbourhood $U' \subset  V\cap W$ such that 
the trajectory of every point in $U' \setminus \{p_{0}\}$ is homotopic in $W{ \setminus \{p_0\}}$ to a curve which is positively transverse to $\cF$.
Let us call a leaf of $\cF$ whose $\alpha$ or $\omega$-limit set is $\{p_{0}\}$ a \emph{separatrix}. Two consecutive separatrices in the cyclic order around $p_{0}$ are of opposite types, one is a stable separatrix oriented towards $p_{0}$ and the other one an unstable separatrix oriented from $p_{0}$.
 Let $S$ be a sector in $U$ bounded by two consecutive separatrices $\ell, \ell'$, and which is on the right-hand side of both $\ell$ and $\ell'$.
By dynamical transversality we have $f(S \cap U') \subset S$; see Figure~\ref{fig:canonical1}. This implies that for every $n>0$, there is some neighborhood of $p_{0}$ in which the $n$ first iterates of $\ell$ and $\ell'$ are pairwise disjoint, and their cyclic order around $p_{0}$ is either given by $\ell, f(\ell), \dots, f^{n}(\ell), f^n(\ell'), \dots, f(\ell'), \ell'$ or by the reverse order; see Figure ~\ref{fig:canonical2}.

Denote by $\ell_0,\ell_1, \ell_2, \ell_3$ the four separatrices of $\cF$, {choosing the numbering so that the sector between $\ell_{0}$ and $\ell_{1}$ is attracting, as on the right-hand part of Figure~\ref{fig:canonical2}}. Let $S_i$ denote a small open sector between $\ell_i$ and its image. Choose for each $i$ a simple curve $\alpha_i$ close enough to the fixed point, joining one point on $\ell_i$ to its image, and whose interior is included in $S_i$. By local transversality the sector $S_i$ is disjoint from $f(S_i), f^2(S_i), f^3(S_i)$, so that $A_i = f^{-2}(\alpha_i) \cup f^{-1}(\alpha_i) \cup \alpha_i \cup f(\alpha_i)$ is a simple curve. Note that again by local transversality the $A_i$'s are pairwise disjoint; see Figure~\ref{fig:canonical2}.

Choose for each $i=1, \dots, 4$ a segment $\beta_i$ of a leaf of the foliation included in the hyperbolic sector between $\ell_{i-1}$ and $\ell_{i}$ (where the indices are taken modulo 4), very close to the boundary $\ell_{i-1} \cup \ell_{i}$ of the sector, with 
\begin{itemize}
\item one endpoint on $\alpha_{i-1}$  and the other  on $\alpha_i$ for $i=1,3$
\item one endpoint on $f^{-1}(\alpha_{i-1})$  and the other on $f^{-1}(\alpha_i)$ for $i=2,4$
\end{itemize}
  and whose interior is disjoint from $A_{i}$ and $A_{i-1}$; see Figure ~\ref{fig:canonical4}.
Let $C$ be the (only) simple closed curve included in the union of the curves
$$
\beta_1, A_1, f^{-1}(\beta_2), A_2, \beta_3, A_3, f^{-1}(\beta_4), A_4.
$$
The reader may check that $C \cap f(C)$ is the union of four segments included respectively in $A_1, A_2, A_3, A_4$; see Figure \ref{fig:canonical4}.  
We modify slightly  $C$ near these segments to get a Jordan curve $C'$ such that $C' \cap f(C')$ consists in four points $P_0, P_1, P_2, P_3$ respectively close to  $A_0, A_1, A_2, A_3$; see Figure \ref{fig:canonical5}.

We claim that the disk bounded by $C'$ is in canonical position for $f$. To see this we define the homeomorphism $\Phi$ as follows. 
Let $D$ be the unit disk, $Q_{0}, Q_1, Q_2, Q_3$ be the points in $D \cap m_{-1}(D)$ as in Figure~\ref{fig:canonical5}. We orient  $\partial D$ and $C'$ in the counter-clockwise direction.
Note that the cyclic order of the points $P_i, f^{-1}(P_{i})$'s along $C'$ coincides with the cyclic order of the points $Q_i, m_{-1}^{-1}(Q_{i})$  on $\partial D$, so that there exists an orientation preserving homeomorphism from $C'$ to $\partial D$ that sends $P_i$ to $Q_i$ and $f^{-1}(P_i)$ to  $m_{-1}^{-1}(Q_i)$ for every $i$. Let $\Phi$ coincide with this homeomorphism on $C'$. Then we extend $\Phi$ on $f(C')$ by the formula 
$$
\Phi = m_{-1} \Phi_{\mid C'} f^{-1}.
$$
The complement of $C'\cup f(C')$ in $U$  has six connected components.  For each connected component, say $\Delta$, it can easily be checked that $\Phi(\partial \Delta)$ is the boundary of some connected component $\Delta'$ of $\partial D \cup m_{-1}(\partial D)$ in the plane. Furthermore, the map $\Delta \mapsto \Delta'$ is a bijection.
Thus, we can use Schoenflies's theorem independently on each $\Delta$ to extend $\Phi$ to a homeomorphism between $U$ and $\mathbb{R}^2$ as required by the definition of canonical position.
\end{proof}

\begin{proof}[Proof of Lemma~\ref{lem:smoothable-negative3}]
If $D$ is a disk in canonical position for $f$, then $U \setminus (\partial D \cup f(\partial D))$ has five bounded connected components (we call unbounded the connected component which contains a neighborhood of the boundary of $U$). We denote their respective areas by $A,A',B,B',C$ so that 

\begin{enumerate}
\item $C$ is the area of $D \cap f(D)$, 
\item $D \setminus f(D)$ has two connected components with areas $A,A'$, 
\item $f(D) \setminus D$ has two connected components with areas $B,B'$.
\end{enumerate}
 Note that since $f$ preserves the area we have $A+A'=B+B'$.  See Figure \ref{fig:canonical6}.
 
\begin{figure}[h!]
\centering
\def\svgwidth{.8\textwidth}
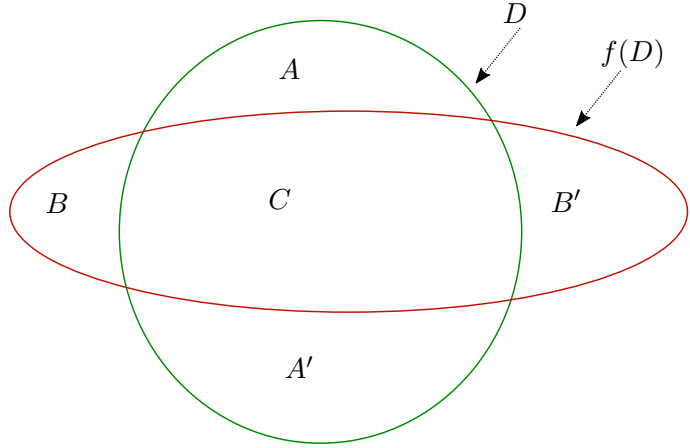
\caption{The areas of the five bounded connected components of $D \cup f(D)$ are denoted by $A,A', B, B'$.  Note that $A+A' = B + B'$.}
\label{fig:canonical6}
\end{figure}

\begin{figure}[h!]
\centering
\def\svgwidth{.8\textwidth}
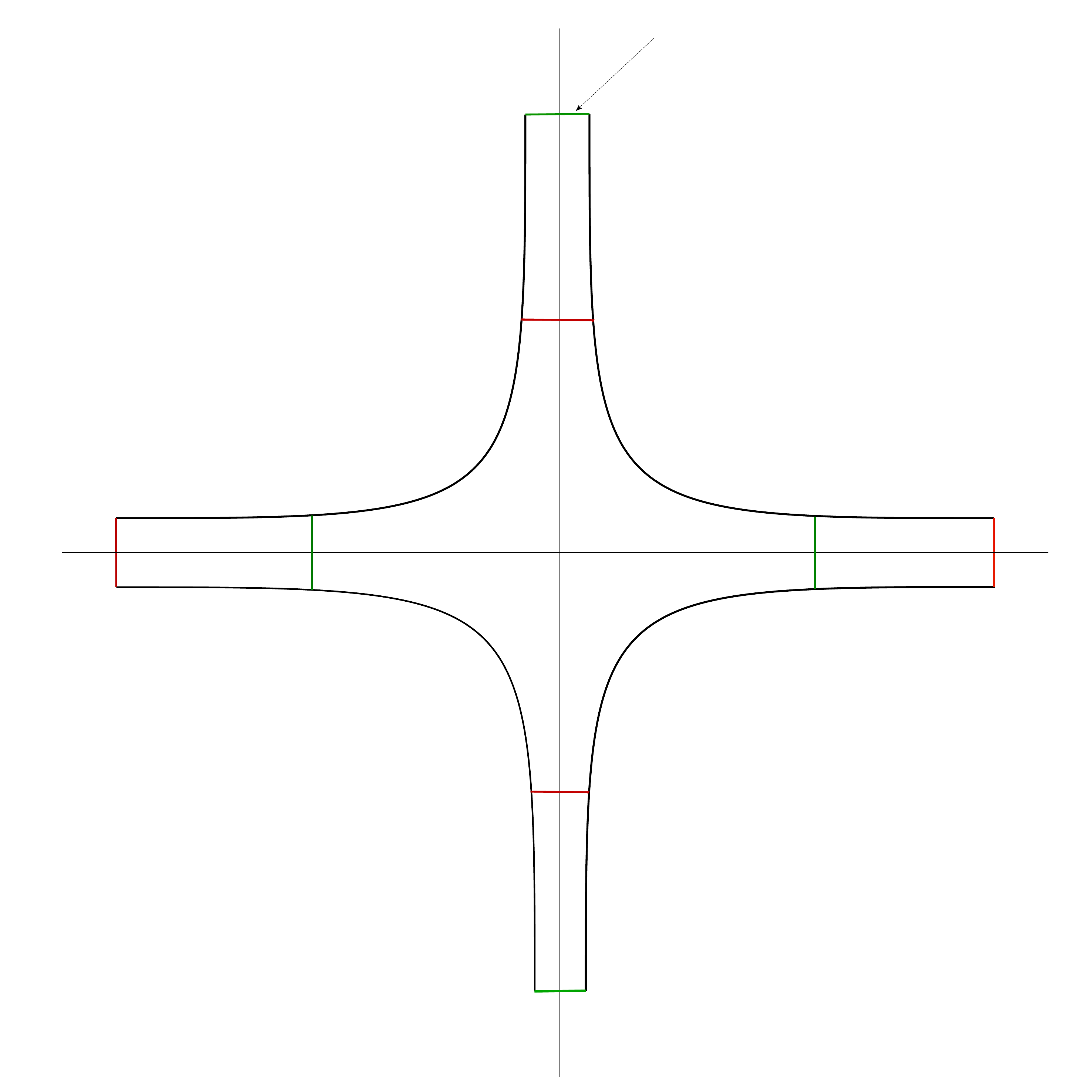
\caption{Obtaining small values of $\rho :=\frac{A}{C}$ in the case where $A = A' = B =B'$: $D_0$ is the disk bounded by the curves $x= \pm \frac{1}{2 \eps}, y = \pm \frac{1}{\eps}$ (in green) and segments along $xy \pm 1$. Note that $xy = \pm 1$ is invariant under $m_{-1}$ and so the image of $D_0$ under $m_{-1}$ is bounded by $x= \pm \frac{1}{\eps}, y = \pm \frac{1}{2\eps}$ (in red) and segments along $xy = \pm 1$.  By picking $\eps$ to be small we can attain arbitrarily large values of $C$. The disk $D_0$ is not in canonical position but one can make a small perturbation to place it in canonical position.}
\label{fig:canonical7}
\end{figure}
 
 We will make use of the following claim whose proof is given below.
\begin{claim}
There exists a disk $D_0$ in canonical position for the model map $m_{-1}$ such that the above properties (1), (2), (3) hold with the same values of $A,A',B,B', C$ when  $D$ and $f$ are replaced by $D_0$ and $m_{-1}$.
\end{claim}
Postponing the proof of the above claim, we continue with the proof of Lemma~\ref{lem:smoothable-negative3}. Using the definition of canonical position both for $f$, $D$ and for $m_{-1}$, $D_0$,
we get a local homeomorphism $(\Phi,U)$ 
taking $D$ to $D_0$ such that 
$\Phi f \Phi^{-1} = m_{-1}$ on the boundary of $D_0$.
We first define $\Psi$ on $\partial D \cup f(\partial D)$ by setting $\Psi=\Phi$ there.
Then, {we choose a small smooth disk $\delta$ around $p_{0}$, a smooth disk $\delta_{0}$ around $0$, included in $D_{0}$, having the same area as $\delta$,
 and define $\Psi$ on $\delta$ to be any area-preserving diffeomorphism between $\delta$ and $\delta_{0}$ that sends $p_{0}$ to $0$}.
Let $M$ be one of the five bounded connected components of the complement of $(\partial D \cup f(\partial D)) \cup \delta$. Let $M'$ be the corresponding  connected component of the complement of $\partial D_0 \cup m_{-1}(\partial D_0) \cup \delta_{0}$.
Then $M'$ has the same area as $M$, so we can apply the Oxtoby-Ulam Theorem (\cite{OxUl41}, Corollary~1 in part~II; see also Theorem~3.1 in \cite{fathi80}) and extend  $\Psi$ to an area preserving homeomorphism from $M$ to $M'$. 


\begin{figure}[h!]
\centering
\def\svgwidth{.6\textwidth}
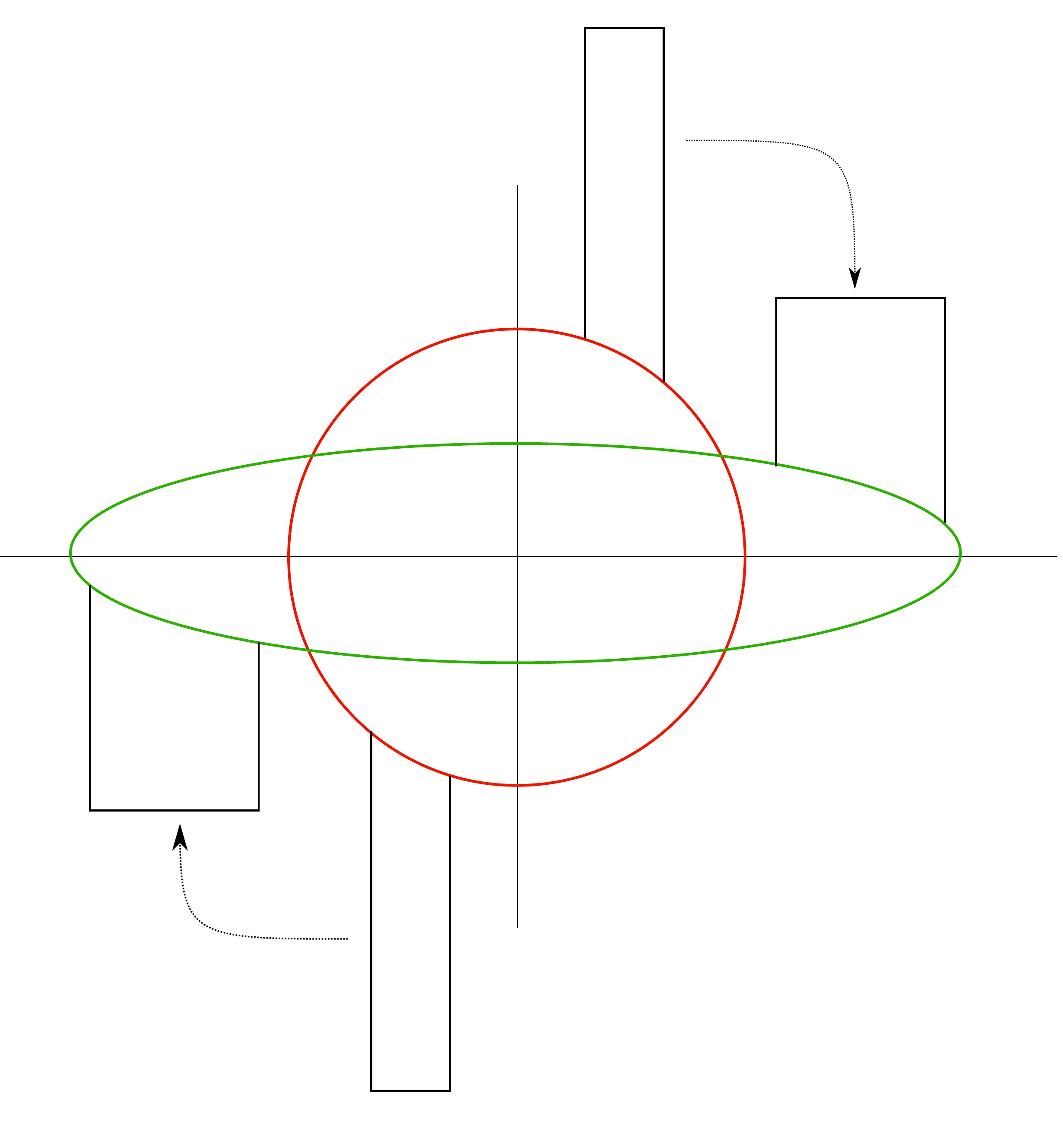
\caption{Obtaining values of $\rho > \rho_0$ in the case where $A=A' = B = B'$: The disk bounding the regions $A, A', C$ is in canonical position and $\rho_0 =\frac{A}{C}$.   We attach two strips of equal area labeled $X, X'$ as depicted above.  The strips $Y, Y'$ are the images of the previous strips under $m_{-1}$. The disk bounding the regions $X, A, C, A', X'$ is also in canonical position.  For this disk we have $\rho = \frac{A+X}{C}$.  By picking $X$ suitably we can attain any value $\rho > \rho_0$.}
\label{fig:canonical8}
\end{figure}

It remains to prove the claim.
We first treat the case when $A=A'$ and $B=B'$. In this case we have $A=B$.
Let $\rho = A/C$. The map $m_{-1}$ commutes with the homothety $z \to \lambda z$, thus the image under $z \to \lambda z$ of a disk in canonical position for $m_{-1}$  is still a disk on canonical position. Consequently we just have to check that we may obtain every value of $\rho$.

{Figures~\ref{fig:canonical7} explains how to construct a disk in canonical position with an arbitrarily small value of $\rho$.  On the other hand, consider given a disk in canonical position with some value of $\rho = \rho_0$. Figure \ref{fig:canonical8} explains how to modify the given disk to obtain a disk in conical position for any value of $\rho > \rho_0$. }


Now we turn to the general case. Note that the rotation by one quarter of a turn conjugates $m_{-1}$ and its inverse. 
Up to replacing the disk $D$ by its image under this rotation, we may assume $A < B < B' < A'$.  We begin by constructing a disk, say $D_0$, in canonical position for which $A, C$ have the required values and $B=B'=A'=A$. Such disk exists by what was explained above.  Now, we modify the disk $D_0$ by attaching two strips of areas $B-A$ and $B'-A$ respectively to the regions which  are supposed to eventually have areas $B$ and $B'$; this is depicted in Figure \ref{fig:canonical9}. To adjust the area of the region which is supposed to have area $A'$, we attach to it the preimages, under $m_{-1}$, of the previous two strips.  It is easy to see that this region will have area $A + (B-A) + (B' - A) = A'$.

\begin{figure}[h!]
\centering
\def\svgwidth{.6 \textwidth}
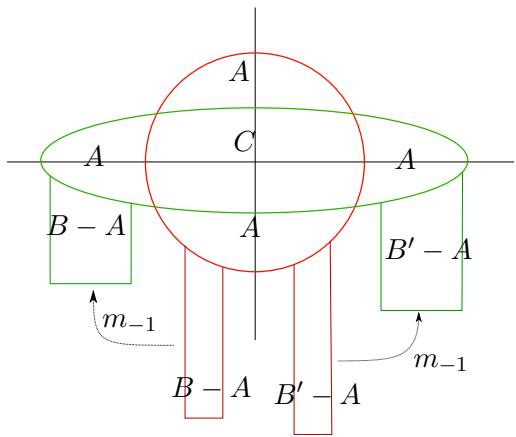
\caption{We construct the regions $B, B'$ by attaching strips of area $B-A$, $B'-A$ (in green) as above.  The region $A'$ is obtained by attaching the preimages, under $m_{-1}$, of the previous strips (in red).}
\label{fig:canonical9}
\end{figure}

\end{proof}
   
 \medskip

\begin{proof}[Proof of Proposition~\ref{pro:smoothable-negative1}]~ Let $(f,U)$  be as in the statement, we apply Theorem~\ref{theo.locally-transverse-foliation} to get a local isotopy $(I,V)$ for $(f,U)$ and a local foliation $(\cF,W)$ which is dynamically transverse to $I$ and has the right index. Since $L(\cF,p_{0}) = L(f,p_{0}) \leq 0$, the foliation $\cF$ consists of at least two hyperbolic sectors, and maybe some parabolic sectors. If there is no parabolic sector then we are done.  Assume that $\cF$ has a parabolic sector.
{Proposition~\ref{pro:smoothable-negative1} is an immediate consequence of the following claim by a descending induction on the number of parabolic sectors around $p_{0}$.
\begin{claim}\label{claim:removing-sector}
In this situation, there exists an area preserving local homeomorphism $(f',U)$ with a single fixed point $p_{0}$, that coincides with $f$ near the boundary of $U$, and a local isotopy $(I',V)$ for $f'$ with a dynamically transverse foliation $(\cF', W)$ that has one less parabolic sector than $(\cF,W)$.
\end{claim}
It remains to prove the claim.
}
We consider a parabolic sector $S$ which is \emph{maximal} in the following sense: for every parabolic sector $S'$ that contains $S$ there exists a neighborhood $U'$ of $p_{0}$ such that $S' \cap U' = S \cap U'$. Proposition~\ref{prop:parabolic-pushing} will apply and provide an isotopy $(I',V)$ that ``sweeps out the parabolic sector $S$'' in the following sense.
\begin{definition}
Let $(I,V)$ 
be a local isotopy, dynamically transverse to a local foliation $(\cF,W)$, and $S$ be a parabolic sector of $\cF$. We denote by $\partial^- S, \partial^+ S$ the two leaves of $\cF$ which bound $S$ near $p_{0}$, so that $S$ is locally on the right-hand side of the oriented leaf $\partial^- S$ and on the left-hand side of $\partial^+ S$.
We say that \emph{$I$ sweeps out the parabolic sector $S$} if the trajectory of every point $x$ in $S \setminus\{p_{0}\}$ close enough to the fixed point $p_{0}$ has positive intersection with the leaf $\partial^+S$; see Figure \ref{fig:sweep}.
\end{definition}

\begin{figure}[h!]
\centering
\def\svgwidth{.6 \textwidth}
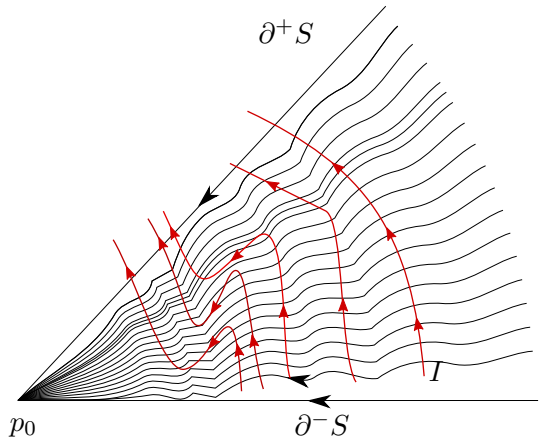
\caption{The isotopy $I$, whose trajectories are in red, sweeps out the parabolic sector $S$.}
\label{fig:sweep}
\end{figure}
Indeed, point 3 of the proposition ensures that $I'=IJ$ sweeps out the sector $S$. In our situation $\cF$ has non positive index, thus the last part of the proposition applies and says that $p_{0}$ is the only fixed point of $f'=fg$.

\begin{prop}[Parabolic Pushing]
\label{prop:parabolic-pushing}
Let $(I,V)$ be a local isotopy for an area preserving local homeomorphism $(f,U)$ dynamically transverse to a local foliation $(\cF,W)$. Assume $p_{0}$ is the only fixed point of $f$ in $U$. Let $S\subset U $ be a parabolic sector of $\cF$ and $O \subset U$ be some open set  containing $S\setminus \{p_{0}\}$.

Then, there exists an isotopy $J$ for some area preserving  homeomorphism $g,$ with the following properties:
\begin{enumerate}
\item $J$  is supported in $\overline{O}$,
\item the local isotopy $(IJ,U)$ is dynamically transverse to $(\cF,W)$,
\item there exists some neighborhood $U'$ of $p_{0}$ such that $g(U' \cap \partial^-S) \subset \partial^+S$.

\end{enumerate}
Furthermore, under any of the following two hypotheses, we may require that $p_{0}$ be the only fixed point of  $fg$ in $U$:
\begin{enumerate}
\item [(i).]  $\cF$ has nonpositive index, 
\item [(ii).] for every $p$ in $O$ such that $f(p)$ is in $O$, the trajectory $I.p$ is homotopic in $U\setminus \{p_{0}\}$ to a curve included in $O$.
\end{enumerate}
\end{prop}   

Next, we apply Lemma~\ref{lemma-modification-foliation} below to modify the foliation into a new transverse foliation $(\cF',W)$ with the parabolic sector $S$ replaced by a single leaf.  To state the lemma we need the following notions.  Let $\ell$ be a leaf of $\cF$ whose $\alpha$ or $\omega$-limit set is $\{p_{0}\}$, and let $p$ be a point on $\ell$.  We will refer to the connected component of $\ell\setminus\{p\}$ whose closure contains $p_{0}$ as a \emph{half-leaf of $\cF$}.  
\begin{lemma}[Modification of the foliation]
\label{lemma-modification-foliation}
Let $(I,V)$ be a local isotopy, dynamically transverse to a gradient-like local  topological foliation $(\cF,W)$. Let $S$ be a maximal parabolic sector of $\cF$, and  $O$ be some open set containing  $S\setminus \{p_{0}\}$, such that every half-leaf of $\cF$ which is included in $O$ is also included in $S$.

Assume $I$ sweeps out the sector $S$. Then there exists a local  topological  foliation $(\cF',W)$ which is dynamically transverse to $(I,V)$, equal to $\cF$ outside $\overline{O}$, and that has a single half leaf included in $O$.
\end{lemma}
The new foliation has one less parabolic sector than the original one. This completes the proof of Claim~\ref{claim:removing-sector}, and thus also the proof of Proposition~\ref{pro:smoothable-negative1}.
\end{proof}


\begin{proof}[Proof of the parabolic pushing proposition]
 Up to conjugation by an orientation reversing map, we may assume that the leaves of $\cF$ through any point of $S$ have their $\omega$-limit set equal to $\{0\}$, that is, $S$ is a positive parabolic sector.
By the ``normal form'' lemma (Lemma \ref{lemma-normal-form-foliation}), we may assume the following: $U,V,W \subset \R^2$, $p_{0}= 0$, $S$ locally coincides with the sector $0< x, 0 \leq y \leq x$, the leaf $\partial^-S$ is included in $y=0$ and oriented towards $0$,  $O$ contains $(0,\varepsilon) \times (-\varepsilon,\varepsilon)$ and on this set the topological foliation $\cF$ is transverse to the vertical foliation. Up to shrinking $U$ and $O$ we may assume that $O = U \cap \{x>0\}= (0,\varepsilon) \times (-\varepsilon,\varepsilon)$.


\begin{figure}[h!]
\centering
\def\svgwidth{.8 \textwidth}
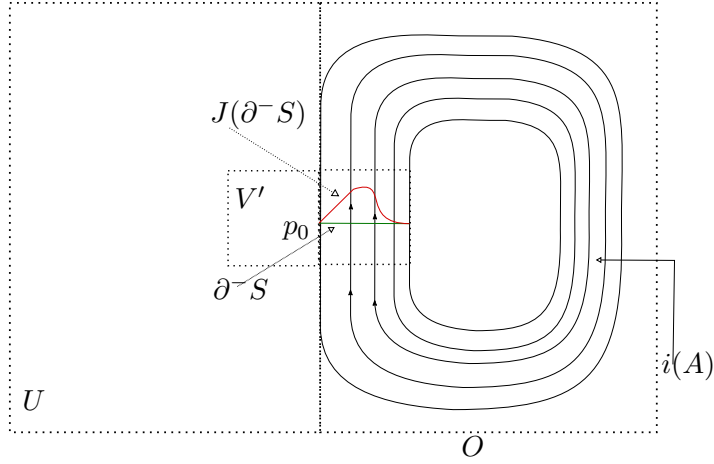
\caption{The isotopy $J$ is supported in the annulus $i(A)$ and sends $\partial^{-}S$ to $\partial^{+}S$ near $p_0$.}
\label{fig:push}
\end{figure}

Let $V'= (-v,v) \times (-v,v) \subset U$  be an open neighborhood of $0$ on which $I$ is transverse to $\cF$ (in the sense of the definition of local transversality).
Fix some small $a>0$ and consider the annulus $A = [0,a] \times \mathbb{S}^1$ equipped with coordinates $(r,\theta)$ and the standard area form $dr d\theta$, and let $i : A \to \overline{O}$ be a symplectic embedding 
such that for every $(r,\theta)$ near $(0,0)$, $i(r,\theta)=(r,\theta)$, where we use the cartesian coordinate system on the target $\R^2$. In particular the curves $r=\text{constant}$ coincide with the vertical lines. Up to shrinking $V'$, we assume $i(r,\theta)=(r,\theta)$ on $V' \cap \{x>0\}$.

Let $W' = (-w,w) \times (-w,w)$ with $w<v$. Since $f$ has no fixed point in $U$ there is some $\delta>0$ such that $d(f(p),p)>\delta$ for every $p \in U \setminus W'$.
Let $J=(g_{t})$ be an isotopy supported on the annulus $i(A)$ defined by $g_{t}(r,\theta) = (r,\theta+t\Delta(r))$
where $\Delta(r)=r$ near $r=0$ and $\Delta$ is small enough so that $g_{1}(W') \subset V'$ and $d(g_{1}(p),p)< \delta$ for every $p$; see Figure \ref{fig:push}.

We leave it to the reader to check that $J$ meets  the first and the third properties. As for the second property, note that for every $p$ in $W'$, the trajectory of $p$ under $J$ is a vertical segment, and thus it is transverse to the foliation. Note that  if $p \in W'$, then $g(p) \in V'$ and thus the trajectory of $p$ under $IJ$ is homotopic to a curve positively transverse to $\cF$ ; in particular the local isotopy $(IJ,V)$ is dynamically  transverse to $(\cF,W)$.

Now assume some point $p \in U \setminus\{p_{0}\}$ is fixed under $fg$. We will prove that none of the two hypotheses (i), (ii) of the last part of the proposition hold.
Since $g$ is supported in $\overline{O}$ and $f$ has no fixed point in $U \setminus \{p_{0}\}$, $p$ must belong to $O$. By definition of $\delta$, the map $fg$ has no fixed point in $U \setminus W'$, thus $p$ must belong to $W' \cap O$. By the previous paragraph, the trajectory $IJ.p$ is homotopic in $U$ to a closed curve positively transverse to $\cF$.  According to Lemma~\ref{lemma-locally-transverse},  $IJ.p$ is not contractible in $U \setminus \{0\}$  and  $\cF$ must be a sink or a source. In particular, $L(\cF,0) =1$, and hypothesis (i) does not hold. Since $J.p$ is included in $O$, the point $p'=g(p)$ is in $O$ and its image $f(p')=p$ is also in $O$, and the trajectory $I.p'$ is not homotopic in $\R^2 \setminus \{0\}$  to a curve included in $O$ (otherwise $IJ.p$ would be contractible in $O \subset \R^2 \setminus \{0\}$). Thus hypothesis (ii) does not hold.
\end{proof}

\begin{proof}[Proof of Lemma~\ref{lemma-modification-foliation}]
{As in the proof of the previous lemma, we may assume the parabolic sector $S$ is positive in the sense that $\{0\}$ is the $\omega$-limit set of any leaf of $\cF$ which is contained in $S$}. By dynamical transversality and the fact that the sector is swept out by the isotopy, up to decreasing $V,W$ we may assume that 
 (1) the trajectory of every point in $V \setminus \{0\}$ is homotopic in $W \setminus \{0\}$ to a curve which is positively transverse to $\cF$, and
 (2) the trajectory of every point in $S \cap V \setminus \{0\}$ has positive intersection number with $\partial^+S$.
Since $\cF$ is gradient-like and the parabolic sector is maximal, it is adjacent to two hyperbolic sectors.
We consider a (non area preserving) chart in which  $V = [-1,1]^2$, $W=[-2,2]^2$, the parabolic sector $S$ is $W \cap \{x \geq0, y \geq 0\}$, 
 and $W \cap \{x \leq0, y \geq 0\}$ is a hyperbolic sector $H$; see Figure \ref{fig:gluing1}. The open ball with center $P$ and radius $\varepsilon$ will be denoted by $B_{\varepsilon}(P)$.
 
 \begin{figure}[h!]
\centering
\def\svgwidth{.8 \textwidth}
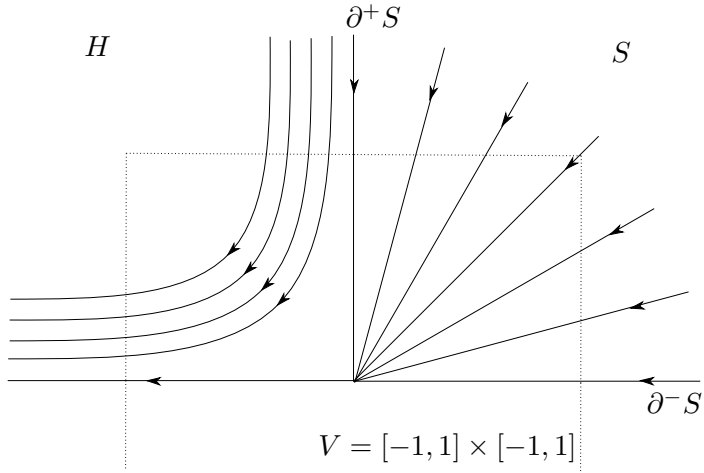
\caption{The maximal parabolic sector $S$ and an adjacent hyperbolic sector $H$.}
\label{fig:gluing1}
\end{figure}

 \begin{figure}[h!]
\centering
\def\svgwidth{.8 \textwidth}
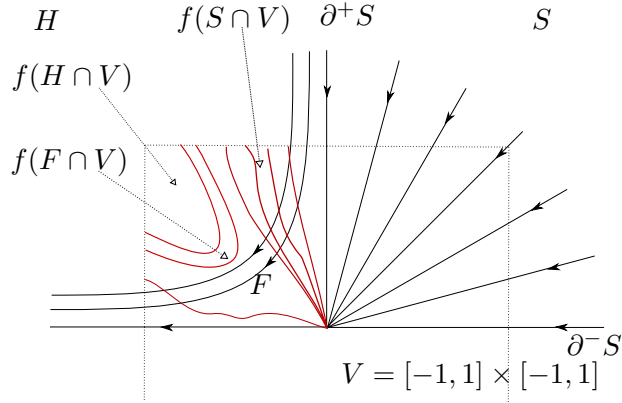
\caption{As a consequence of transversality $f((S \cup H) \cap V)$ is included in $H$ and  $f(F\cap V)$ is on the right-hand side of $F$ in $H$.  }\label{fig:gluing2}
\end{figure}


Denote by $f$ the time one of the isotopy $I$. Note that by transversality, $f((S \cup H) \cap V)$ is included in $H$, and for any leaf $F$ in the interior of $H$, $f(F\cap V)$ is on the right-hand side of $F$ in $H$; see Figure \ref{fig:gluing2}.
 The foliation $\cF'$ will be obtained from $\cF$ by modifying it on $V\cap \{y>0\}$.
Let $y \in (0,1)$, let $P=(0,y)$, and define $F'_{y} = [0,1] \times \{y\}$.

By dynamical transversality there exists $\varepsilon>0$ such that the ball $B_{\varepsilon}(P)$ is included in $V\cap O$ and there exists a leaf $F$ in the interior of $H$ that separates in $W$ the set 
$F'_{y} \cup B_{\varepsilon(y)}(P)$ from its image under $f$; we denote by $\varepsilon(y)$ the maximum of such numbers. 
Let us call $F'$ a good curve if it is obtained the following way (see Figure \ref{fig:gluing3}).
\begin{figure}[h!]
\centering
\def\svgwidth{.8 \textwidth}
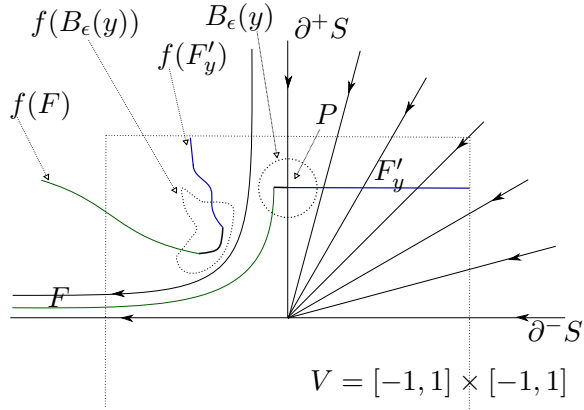
\caption{ There exists a leaf of the foliation separating $F'_y \cup B_{\varepsilon}(y)$ from its image.  The good curve $F'$ consists of $F'_y$ (in blue), $\gamma$ (in black) and a part of the leaf $F$ (in green).  }\label{fig:gluing3}
\end{figure}
We choose a point $P_{1}=(-2,y_{1})$  with $y_{1}>0$ small enough so that the leaf $F$ through $P_{1}$ meets the ball $B_{\varepsilon(y)}(P)$. Let $\gamma$ be a curve in that ball that joins a point $P_{2}$ in $F$ to the point $P$ and whose interior is on the left-hand side of $F$ and on the right-hand side of $\partial^+S$ in $W$. Finally define $F'$ to be the union of $F'_{y}$, $\gamma$ and the piece of $F$ from $P_{2}$ to $P_{1}$, oriented in this order. It is easy to check that the image of $F' \cap V$ is on the right-hand side of $F'$ in $W$.
We leave it to the reader to check that there is a topological  foliation $\cF'$, defined on some neighbourhood of $0$, whose leaves in the upper half plane $\{y >0\}$ are good curves.
Any such foliation  is dynamically transverse to the isotopy $I$. Furthermore it coincides with $\cF$ outside $\overline{O}$ and the $x$--axis is the only half-leaf of $\cF'$ included in $O$.
\end{proof}

\subsection{Smoothability of fixed points of index one}\label{subsec:index-one}

\subsubsection{Local rotation set and transverse foliations} \label{subsubsec:rotation}
Let  $(f,U)$ be an area preserving local homeomorphism with fixed point $p_{0}$. Choose a local isotopy $(I,V)$ for $(f,U)$. 
Let us outline the basic properties of the local rotation set $\rho(I)$, and in particular the relationship with transverse foliations (we refer to~\cite{leroux13} for more details). It is by definition a closed subset of $\R \cup \{\pm \infty\}$. 
Let $(R,V)$ be a local isotopy from the identity to the identity that makes a full turn around $p_{0}$ in the positive direction. The formulae $\rho(I^p)=p\rho(I)$, $\rho(R^pI)=\rho(I)+p$ hold. 

Let $\alpha:[0,1] \to U$ be a simple curve with $\alpha(1)=p_{0}$, and orient $\alpha$ from $\alpha(0)$ to $\alpha(1)$. Let $\gamma:[0,1] \to U$ be a curve which is disjoint from $\alpha(0), \alpha(1)$ and such that $\gamma(t) \not \in \alpha$ except for finitely many values of $t$.  We define the intersection number $\sharp \gamma \wedge \alpha$ as follows. We count $1$ each time $\gamma$ crosses $\alpha$ from the left-hand side of $\alpha$ to the right-hand side, $-1$ when it crosses in the opposite direction, $\pm\frac{1}{2}$ if $\gamma(0) \in 
\alpha$ according to whether $\gamma$ starts on the right-hand side or left-hand side of $\alpha$, and similarly if $\gamma(1) \in \alpha$. We remark that we have $\sharp \gamma' \wedge \alpha = \sharp \gamma \wedge \alpha$ if $\gamma'$ is another curve which is homotopic to $\gamma$ in $U \setminus \{\alpha(0),\alpha(1)\}$ with respect to its endpoints. This remark allows to extend the definition to the case of a curve $\gamma$ that intersects $\alpha$ infinitely many times.

We say that $\alpha$ is a \emph{positive arc} for the isotopy $(I,V)$ 
 if the intersection number $\sharp I.x \wedge \alpha$ is positive for every $x \in \alpha$ close enough to $p_{0}$.
 We  say that $\alpha$ is a \emph{direct positive arc} 
   if $\sharp I.x \wedge \alpha = \frac{1}{2}$ for every $x \in \alpha$ close enough to $p_{0}$. Note that in this last case  $\alpha$ is locally disjoint from its image, that is, there exists some neighborhood $O$ of $p_{0}$ such that  $f(\alpha) \cap \alpha \cap O= \{p_{0}\}$.
The following simple claim will be used repeatedly (this is a special 
and simple case of the more general fact that the local rotation set is included in the \emph{rotation interval}, see again~\cite{leroux13}).
\begin{claim}\label{claim:positive-arc}
If there is a positive arc for the local isotopy $I$, then the local rotation set $\rho(I)$ is included in $[0,+\infty]$. If there is a direct positive arc then $\rho(I)$ is included in $[0,1]$. 
\end{claim}
For the next claim, {please note that the sector $S$ is not between $\alpha$ and $f(\alpha)$ but between $f(\alpha)$ and $\alpha$.}
\begin{claim}\label{claim:sector-direct}
Assume that  $\alpha$ is a direct positive arc for $I$. Let $S=S(f(\alpha),\alpha)$ be a sector between $f(\alpha)$ and $\alpha$. Then   for every point $x \in S \setminus \{p_{0}\}$, close enough to $p_{0}$,  and such that $f(x)$ also belongs to $S$, the trajectory $I.x$ is homotopic in $U \setminus p_{0}$ to a curve included in $S$.
\end{claim}
\begin{proof}
If $x$ is close enough to $p_{0}$ then its trajectory does not meet the end-points of $\alpha$ and $f(\alpha)$. Since $\alpha$ is a direct positive arc, the intersection number $n(x) = \sharp I.x \wedge f(\alpha)$vanishes for $x$ in the interior of $S$ and close to $\alpha$.
Furthermore the function $x \mapsto n(x)$ is constant on the interior of $S$: indeed the interior of $S$ is connected and this function is locally constant, since when $x$ is in the interior of $S$ the endpoints of $I.x$ do not belong to $f(\alpha)$.  Thus it vanishes identically on $S$. The claim follows.
\end{proof}

Assume now that there is a local foliation $(\cF,W)$ dynamically transverse to $(I,V)$. 
A leaf whose $\omega$-limit set is $\{p_{0}\}$ is a positive arc for $I$, thus Claim~\ref{claim:positive-arc} applies if there is such a leaf, and $\rho(I) \subset [0,+\infty]$. Symmetrically, if $\cF$ admits a leaf whose $\alpha$-limit set is $\{p_{0}\}$ then $\rho(I) \subset [-\infty,0]$.
Since $f$ preserves the area, the foliation $\cF$ is gradient-like, and every gradient-like foliation admits a leaf whose $\alpha$ or $\omega$-limit set is $\{p_{0}\}$: thus  we see that $\rho(I)$ is either included in $[-\infty,0]$ or in $[0,+\infty]$. 

If $p_{0}$ is an isolated fixed point of $f$ then by Theorem~\ref{theo.locally-transverse-foliation} for every $p$ the isotopy $R^pI$ admits a dynamically transverse local foliation, thus we see that $\rho(I)$ is included in an interval $[q,q+1]$ for some integer $q$. Up to replacing $I$ by $R^{-q}I$ we may assume $\rho(I) \subset [0,1]$. The same kind of arguments, applied to the powers of $f$, shows that the closed set $\rho(I)$ is actually an interval: every rational number in the interior of the convex hull of $\rho(I)$ is the rotation number of a sequence of periodic orbits accumulating $p_{0}$, and thus actually belongs to $\rho(I)$.

\subsubsection{Strategy}
Let us start the proof of Theorem~\ref{theo:smoothable}, namely that  $f$ is smoothable at an isolated fixed point $p_{0}$ provided that the local rotation set is not equal to $[0,1]$ modulo $1$. 
 If the index $L(f,p_{0})$ is non positive then the result follows from the previous section. Thus we may and we will assume that $L(f,p_{0})=1$. As explained at the beginning of section~\ref{subsec:negative-index}, we need only to consider an area preserving local homeomorphism $(f,U)$ with a single fixed point $p_{0}$, and we want to find an area preserving local homeomorphism $(g,U)$ that fixes only $p_{0}$, coincides with $f$ near the boundary of $U$, and in addition is smooth near $p_{0}$. By the above considerations we may choose a local isotopy $(I,V)$ from the identity to $f$ whose rotation set is included in $[0,1]$. Theorem~\ref{theo.locally-transverse-foliation} provides a local foliation which is dynamically transverse to $(I,V)$, and whose index is one. Since it is gradient-like, it is a sink or a source. The source case may happens only when $\rho(I)=\{0\}$, and we can avoid it by changing $f$ into $f^{-1}$, $I$ into $I^{-1}$, and $\cF$ into the same foliation with reverse orientation.

By hypothesis, the fixed point is not totally degenerate, \emph{i. e.} the local rotation set is a proper sub-interval of $[0,1]$.
We first note that up to changing $f$ into $f^{-1}$ and changing $I$ into $RI^{-1}$, we may assume that the local rotation set in included in $[0,1)$. 
Let us describe the strategy of the proof. The starting point is a situation in which it is easy to smooth $f$ at $p_{0}$ : in this situation we assume that the fixed point belongs to the interior of a curve $\alpha$ such that  $f$ permutes the two connected components of $\alpha \setminus\{p_{0}\}$ (Lemma~\ref{lem:periodic-arc2} below). This situation is certainly very special, in particular it implies that the local rotation set equals $\{\frac{1}{2}\}$. 
In the general case the rotation set is positive, and the local isotopy is dynamically transverse to the sink foliation $\cF$. The general idea of the proof is to push $f$ transversally to the foliation $\cF$, in the positive direction, in order to reduce the size of the rotation set until it becomes equal to $\{\frac{1}{2}\}$ and we can hope to have this special situation. The transverse foliation guarantees that the pushing creates no new fixed point with rotation number zero (points are already turning in the positive direction around $p_{0}$ and we push in the same direction). One difficulty will be to avoid the creation of fixed points with positive rotation number. This control will be made possible by using a transverse foliation with the following extra dynamical property: there is a leaf of $\cF$ whose first iterates are also leaves of $\cF$; this is the content of the ``iterated leaf lemma''~\ref{lem:iterated-leaf} below.

More concretely, in Proposition~\ref{prop.smoothability-0-un-demi} we consider the case when the local rotation set is not too large, namely it is included in $[0,\frac{1}{2})$, and we show how to apply the pushing strategy to get to the special situation that we have described at the beginning (Lemma~\ref{lem:periodic-arc1}).
Finally we will consider the general case, when the rotation set is included in $[0,1)$. 
    Roughly speaking we will apply the pushing strategy to reduce the rotation set to a subset of $[\frac{2}{3},1]$. Then by considering $f^{-1}$ and a well chosen isotopy this amounts to a rotation set included in $[0,\frac{1}{3}]$, and we conclude by applying Proposition~\ref{prop.smoothability-0-un-demi}.

\bigskip
We end this subsubsection by proving Proposition~\ref{prop.smoothability-0-un-demi} (the $[0,\frac{1}{2})$ case). The general case is treated in the next subsubsection, and the last subsubsection contains the proof of the iterated leaf lemma.

\begin{lemma}\label{lem:periodic-arc2}
Let  $(f,U)$ be an area preserving local homeomorphism with a single fixed point $p_{0}$.
Assume there is a simple curve $\gamma$ in $U$ such that $p_{0}$ is in the interior of $\gamma$, $f$ exchanges both connected components of $\gamma \setminus\{p_{0}\}$, and $f^2$ is the identity on $\gamma$.
Then $(f,U)$ is smoothable.
\end{lemma}

\begin{proof}[Proof of Lemma~\ref{lem:periodic-arc2}]
By conjugation we may assume $U$ is in the plane, $\gamma$ is a horizontal segment, the fixed point $0$ is the midpoint of $\gamma$, and $f$ coincides on $\Gamma$ with the euclidean rotation $R$ of angle $\pi$.
Let $D$ be a small disk centered at $0$, and denote $D^+$ and $D^-$ the upper and lower half-disks.
By continuity, if $D$ is small enough, then $f(D^+)$ is included in the half-plane $\{y \leq 0\}$, and in particular the interior  of $D^+$ is disjoint from its image~; likewise for $D^-$.
Choose a euclidean disk $D'$ centered at $0$, such that $fR(D')$ is included in the interior of $D$. Applying independently the Oxtoby-Ulam theorem on $D^+$ and $D^-$, we may find a homeomorphism $\Phi$ of $D$, which preserves the area, which is the identity on the circle $\partial D$ and on the diameter $\gamma$, and which coincides with $fR$ on $D'$. We extend $\Phi$ by the identity outside $D$. Let $f'=\Phi^{-1} f$. Then $f'$ preserves the area,  coincides with $f$ outside $D$, and with $R$ on $D'$. 
Furthermore, since the interiors of $D^+$ and $D^-$ are disjoint from their images, $f'$ has no fixed point but $0$. Of course, $f'$ is smooth near $0$.
\end{proof}

\begin{prop}\label{prop.smoothability-0-un-demi}
Let  $(f,U)$ be an area preserving local homeomorphism with a single fixed point $p_{0}$. Let $(I,V)$ be a local isotopy for $(f,U)$, and assume that the local rotation set of $I$ is included in $[0, \frac{1}{2})$. Then $(f,U)$ is smoothable.
\end{prop}

The proof of the proposition relies on the following fundamental lemma which will be proved in the next subsubsection.
\begin{lemma}[iterated leaf lemma]
\label{lem:iterated-leaf}
Let  $(f,U)$ be an area preserving local homeomorphism with a single fixed point $p_{0}$, such that $L(f,p_{0})=1$.
Let $(I,V)$ be a local isotopy for $(f,U)$, and  assume that the local rotation set of $I$ is included in $[0, \frac{1}{q})$ for some positive integer $q$.
 
Then there exists a local foliation $(\cF,W)$ with singularity $p_{0}$ which is a sink or a source, which is dynamically transverse to the isotopy $(I,V)$, and an open arc $\alpha$ of which $p_{0}$ is an end-point, such that $\alpha, f(\alpha), \dots, f^{q}(\alpha)$ are included in some leaves of $\cF$.
\end{lemma}

{Note that under the hypothesis of this lemma, if furthermore the local rotation set is not $\{0\}$, then $\cF$ is a sink,  $\alpha$ is a direct positive arc for $I$, and $\alpha \setminus \{0\}, \dots, f^q(\alpha) \setminus \{0\}$ are pairwise disjoint and in this cyclic order around $p_{0}$.
On top of the existence of such an arc $\alpha$, the lemma provides the foliation $\cF$ dynamically transverse to $(I,V)$ ; this adds some information \emph{only in the sector between $f^q(\alpha)$ and $\alpha$} (in the other sectors, any topologically radial foliation is transverse to the isotopy), but this little piece of information will be crucial in what follows.}

\begin{proof}[Proof of Proposition~\ref{prop.smoothability-0-un-demi}]
Let $f,I$ be as in the statement of the proposition. We may assume that $L(f,p)=1$.
 Apply Lemma~\ref{lem:iterated-leaf} with $q=2$ to get a foliation $\cF$ dynamically transverse to $I$ with a half leaf $\alpha$ such that $f(\alpha)$ and $f^2(\alpha)$ are also half-leaves of $\cF$. As said before the source case may happen only when the rotation set is $\{0\}$, thus we may always assume that $\cF$ is a sink up to changing $f$ into $f^{-1}$, $I$ into $I^{-1}$ and reversing the orientation of $\cF$.
The proposition is now an immediate consequence of lemma~\ref{lem:periodic-arc2} and lemma~\ref{lem:periodic-arc1} below.
\end{proof}

\begin{lemma}\label{lem:periodic-arc1}
In the situation of the previous proof, there is a local homeomorphism $(f',U)$ that satisfies the hypotheses of lemma~\ref{lem:periodic-arc2} and which
equals $f$ near the boundary of $U$.
\end{lemma}

\begin{figure}[h]
\centering
\def\svgwidth{.65 \textwidth}
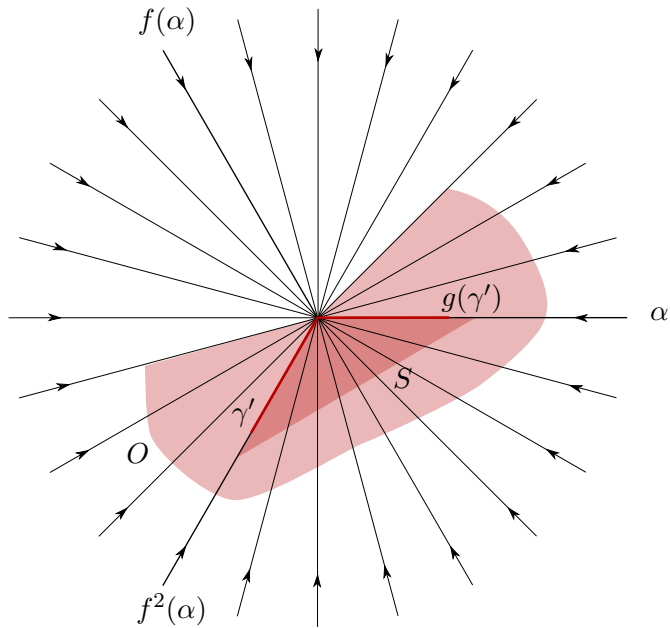
\caption{The foliation $\cF$ and the parabolic sector $S$}
\label{fig:index1-onehalf}
\end{figure}

\begin{proof}[Proof of Lemma~\ref{lem:periodic-arc1}]
Since $\cF$ is a sink the arc $\alpha$ is positive, and since $\alpha$ and $f(\alpha)$ are leaves of $
\cF$ and $\rho(I) \subset [0,\frac{1}{2})$, $\alpha$ is a direct positive arc.
Let $\beta$ be a simple curve with $\beta(1)= p_{0}$ and otherwise included in the interior of a small sector between $\alpha$ and $f(\alpha)$. Then it is easy to see that $\beta$ is also a direct positive arc.

 Let $S$ be a small parabolic sector of $\cF$ between $f^2(\alpha)$ and $\alpha$, and $O$ be a small  sector between $f(\beta)$ and $\beta$  (see the figure). Claim~\ref{claim:sector-direct} applies:  
 for every point $x$ close enough to $p_{0}$ in $O$ and such that $f(x)$ also belongs to $O$, the trajectory $I.x$ is homotopic in $U \setminus p_{0}$ to a curve included in $O$. In other words, the hypothesis (ii) at the end of the parabolic pushing proposition~\ref{prop:parabolic-pushing} holds.
We apply proposition~\ref{prop:parabolic-pushing} to the parabolic sector $S$. We get a local isotopy $(J,U)$ for some area preserving homeomorphism $(g,U)$ supported in $O$, dynamically transverse to $\cF$, with some small sub-arc $\gamma'$ of $f^2(\alpha)$ such that $g(\gamma') \subset \alpha$. Furthermore  $p_{0}$ is the only fixed point of $fg$ in $U$. 

We have $(fg)^2(\gamma')  \subset fgf(\alpha) \subset f^2(\alpha)$ since $g$ is the identity on $f(\alpha)$. So both $\gamma'$ and $(fg)^2(\gamma')$ are included in $f^2(\alpha)$, which is disjoint from its image except at $p_{0}$.
Let $Z$ be an open set containing $\gamma' \cup (fg)^2(\gamma') \setminus \{p_{0}\}$, and sufficiently small so that $fg(Z) \cap Z = \emptyset$.
Using the Oxtoby-Ulam theorem, one can construct an area preserving homeomorphism
 $h$, supported  on $Z$, and such that $h = (fg)^2$ on $\gamma'$. Let $f'=h^{-1}fg$. Since $h$ is the identity on $fg(\gamma')$ we get that  $f'^2$ is the identity on $\gamma'$, and thus also on $\gamma= \gamma' \cup f'(\gamma')$. Furthermore $f'$ coincides with $f$ near the boundary of $U$ as required by the lemma, and  the proof is complete.
 \end{proof}

\subsubsection{Proof of the general case}

We now proceed to the proof of Theorem~\ref{theo:smoothable} in the case when the local rotation set is included in $[0,1)$.
We consider an area preserving local homeomorphism   $(f,U)$ with a single fixed point $p_{0}$, a local isotopy $(I,V)$ for $(f,U)$, and we  
assume that $\rho(I)$ is included in $[0, 1)$. By Proposition~\ref{prop.smoothability-0-un-demi} we may also assume that $\rho(I)$ is not included in $[0,\frac{1}{2})$, and in particular it is not $\{0\}$ and we have $L(f,p_{0})=1$.
We want to show that $(f,U)$ is smoothable. The strategy will be to make a small modification of $f$ and $I$ to get $f''$ and an isotopy $I''$ whose local rotation set is included in $[-\frac{1}{3},0]$. Then the inverse of $I''$ is a local isotopy whose time one map is $f''^{-1}$, and whose rotation set is included in $[0, \frac{1}{3}]$. Then, as explained in the strategy, Proposition~\ref{prop.smoothability-0-un-demi} entails that $(f''^{-1},U)$ is smoothable, thus $(f'',0)$ is also smoothable, and so is $(f,U)$. It remains to describe the construction of $f''$, which will take three steps.

\begin{lemma}[Step 1]
There exists a foliation $\cF$ transverse to $I$ with two leaves $\alpha,\alpha'$ of $\cF$ whose images are also leaves of $\cF$.
\end{lemma}

\begin{figure}[h]
\centering
\def\svgwidth{\textwidth}
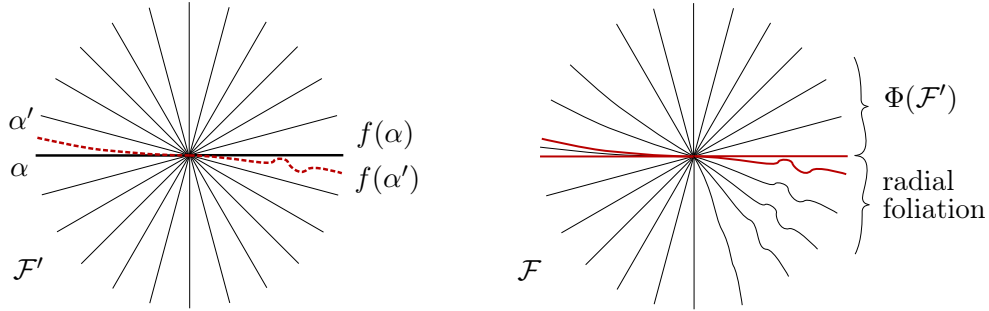
\caption{Proof of step 1}
\label{fig:index1-2leaves}
\end{figure}

\begin{proof}
We apply the iterated leaf lemma~\ref{lem:iterated-leaf} with $q=1$ and get a first foliation $\cF'$ with a leave $\alpha$ whose image is also a leaf of $\cF'$. We have assumed that $\rho(I)$ is included in $[0,1)$ and is not $\{0\}$, thus $\cF$ is a sink, and $\alpha$ is a direct positive arc for $I$.
We will modify $\cF'$ into a foliation $\cF$ having the desired property; this modification will be a small perturbation in the Whitney topology, and thus $\cF$ will still be dynamically transverse to $I$ by the stability property (Proposition 3.3 of \cite{leroux13}).
Let $\alpha'$ be a simple curve with $\alpha'(1)=p_{0}$ and the remaining of $\alpha'$ included in the open sector between $f(\alpha)$ and $\alpha$, and which is  Whitney-close to $\alpha$. Then $f(\alpha')$ is disjoint from $\alpha$ and  included in the open sector between $\alpha$ and $f(\alpha)$.
We choose a Whitney small local homeomorphism $\Phi$ that fixes $f(\alpha)$ and sends  $\alpha$ to  $\alpha'$.
By Whitney stability, the foliation $\Phi(\cF')$ is still locally transverse to $I$, and $f(\alpha)$ and $\alpha'$ are leaves of this foliation. Let $\cF$ be the foliation that coincides with with $\Phi(\cF')$ on  $S(f(\alpha), \alpha')$, and which is a topologically radial foliation on $S(\alpha',f(\alpha))$ that includes $\alpha$ and $f(\alpha')$ as leaves: 
to construct $\cF$, we can use coordinates in which all four curves are rays and take the foliation to be radial on that sector. Note that, near $p_{0}$, $f$ sends the parabolic sector $S(\alpha',\alpha)$ to the disjoint parabolic sector $S(f(\alpha'), f(\alpha))$. We leave it to the reader to check, using this property, that the foliation $\cF$ is again locally transverse to $I$.

\end{proof}

\begin{figure}[h]
\centering
\def\svgwidth{0.9\textwidth}
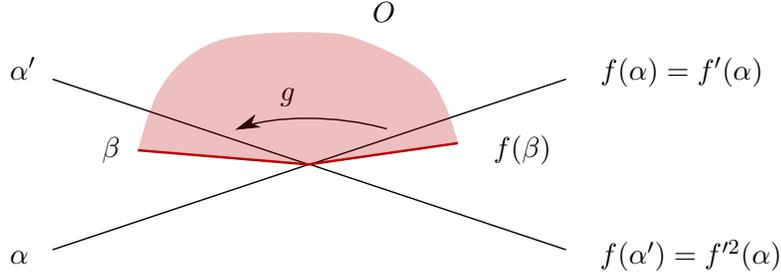
\caption{Proof of step 2}
\label{fig:index1-3iterates}
\end{figure}

{In the next lemmas we will make a slight language abuse by saying that the curves $\alpha,\alpha'$ are disjoint if they meet only at $p_{0}$.}
\begin{lemma}[Step 2]
Consider $f,U,\alpha,\alpha'$ as in step 1. Then there exists a  local homeomorphism $(f',U)$ whose only fixed point is $p_{0}$, which coincides with $f$ near the boundary of $U$, and such that the three curves  $\alpha$, $f'^2 (\alpha)$,  $f'(\alpha)$ are pairwise disjoint and in this cyclic order around $p_{0}$.
\end{lemma}
\begin{proof}
Consider the foliation $\cF$ provided by step 1. Choose a curve $\beta$ such that $\beta(1)=p_{0}$ between $\alpha'$ and $\alpha$ in the cyclic order around $p_{0}$ (see Figure~\ref{fig:index1-3iterates}). Note that $f(\beta)$  is between $f(\alpha')$ and $f(\alpha)$. 
Let $S = S(f(\alpha),\alpha')$ be a small parabolic sector of $\cF$ included in $U$, and $O = S(f(\beta),\beta)$ be a small sector containing $S \setminus \{p_{0}\}$. We apply the ``parabolic pushing'' Proposition~\ref{prop:parabolic-pushing} to the parabolic sector $S$ to get an isotopy $J$ compactly supported in $U$ that pushes $f(\alpha)$ to $\alpha'$ near $0$ and is locally transverse to $\cF$. Let $g$ be the time one of $J$. Claim~\ref{claim:sector-direct} applies to $O$, and
according to the end of the proposition the map $f'=fg$ has no fixed point in $U$ but $p_{0}$. Finally note that since $\alpha$ does not belong to the support of $J$, $f'(\alpha) = f(\alpha)$ and thus $f'^2(\alpha) = f(\alpha')$ and we get  $\alpha, f'^2 (\alpha), f'(\alpha)$ as wanted.
\end{proof}

\begin{lemma}[Step 3]
Let $f',U,\alpha$ be as in the conclusion of step 2.
Then there exists a  local homeomorphism $(f'',U)$ whose only fixed point is $p_{0}$, which coincides with $f$ near the boundary of $U$, and such that the curves 
$$
\alpha, f''^3(\alpha), f''^2(\alpha),  f''(\alpha)
$$
are pairwise disjoint and in that cyclic order around $p_{0}$. In particular, there exists a local isotopy $I''$ from the identity to $f''$ whose local rotation set is included in $[-\frac{1}{3},0]$.
\end{lemma}

\begin{figure}[h]
\centering
\def\svgwidth{\textwidth}
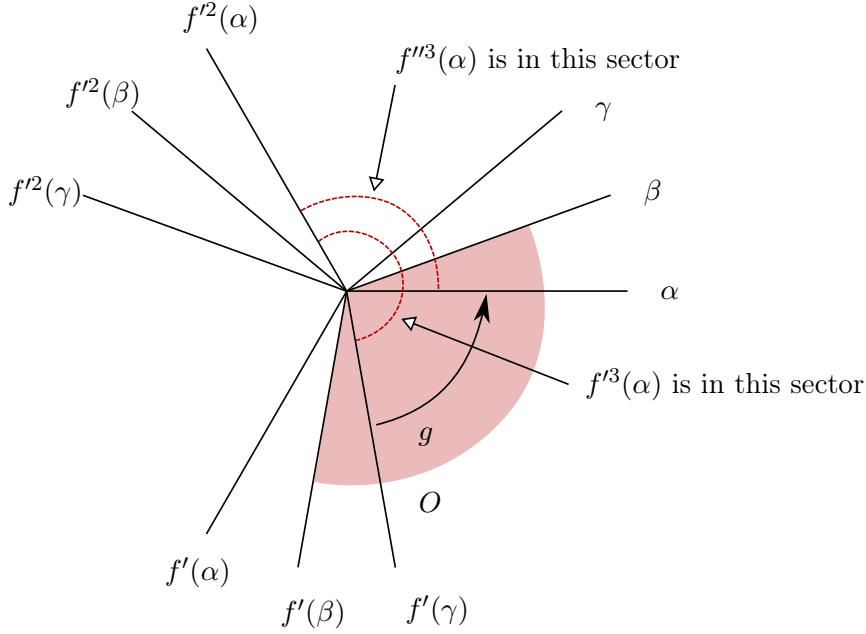
\caption{Proof of step 3}
\label{fig:index1-4iterates}
\end{figure}

\begin{proof}
We can find two curves $\beta,\gamma$ very close to $\alpha$, so that we get nine pairwise disjoint curves with the following cyclic order around $p_{0}$:
$$\alpha,\beta,\gamma,f'^2\alpha,f'^2\beta,f'^2\gamma, f'\alpha,f'\beta,f'\gamma.$$
(see Figure~\ref{fig:index1-4iterates}). Furthermore since $f'^2(\alpha)$ is included in the sector  $S(\gamma, f'(\alpha))$, we get $f'^3 \alpha \subset S(f'\gamma, f'^2\alpha)$. 
We consider coordinates in which all the above nine curves are straight rays.
Note that in these coordinates $I$ is dynamically transverse to the radial foliation.
 We consider some small parabolic sector $S = S(f'(\gamma),\alpha)$, and some small open set $O=S(f'(\beta),\beta)$ containing $S \setminus \{0\}$, so that the hypotheses of the parabolic pushing proposition~\ref{prop:parabolic-pushing} is satisfied. Now we push as before by some local isotopy $(J,V)$ for a local homeomorphism $(g,U)$ supported in $O$ and that locally sends $f'(\gamma)$ to $\alpha$. Near $p_{0}$ we have $g(f'^3(\alpha)) \subset g(S(f'\gamma, f'^2\alpha)) = S(\alpha,f'^2\alpha)$.
 By proposition~\ref{prop:parabolic-pushing} $f'g$ has no fixed point but $p_{0}$ in $U$, thus so does $f''=gf'$ since they are conjugate by $g$ which is supported on $U$. Since $g$ is the identity on $f'(\alpha)$ and $f'^2(\alpha)$ we get  $f''(\alpha)=f'(\alpha)$ and $f''^2(\alpha)=f'^2(\alpha)$, so that curves
 $$
 \alpha, f''^3(\alpha),  f''^2(\alpha)=f'^2(\alpha) , f''(\alpha)=f'(\alpha)
 $$
 are pairwise disjoint and in that cyclic order around $p_{0}$,
 as wanted. 
 
 Let $(I'',U)$ be a local isotopy from the identity to $f''$ that pushes $\alpha$ to $f''(\alpha)$ of less than one full turn in the negative direction, then it also pushes similarly 
  $f''(\alpha)$ to $f''^2(\alpha)$ and  $f''^2(\alpha)$ to $f''^3(\alpha)$. Since $\alpha$ is a positive arc for $I^{-1}$, the local rotation set $\rho(I'')$ is included in $[-\infty,0]$ by Claim~\ref{claim:positive-arc}. On the other hand $\alpha$ is also a positive arc for $R I^3$, and thus we also get  $\rho(I'') \subset [-\frac{1}{3},+\infty]$, and finally the local rotation set is included in $[-\frac{1}{3},0]$ as wanted.
 \end{proof}


\newcommand{\bbR}{\mathbb{R}}
\subsubsection{Proof of the ``iterated leaf lemma''}\label{sss.iterated-leaf-lemma}

In order to prove the iterated leaf lemma, our strategy is first to extend our local dynamics to the whole plane. The ``extension lemma'' below shows that this extension can be done in such a way that the global rotation set around the fixed point is not too large, barely larger than the local rotation set of the initial local homeomorphism. In this global setting, we work in the infinite annulus $A=\R^2 \setminus \{0\}$. The dynamics lifts to a map $\tilde f$ of the universal cover $\tilde A \simeq \R^2$. We denote by  $T$ the deck transformation. The second step in the proof of the iterated leaf lemma consists in playing with maps of the form $T^p\tilde f^{-q}$: we prove that if such a map $T'$  has positive (global) rotation set then it is conjugate to a plane translation; this is the content of the ``quotient lemma''. Then $\tilde f$ induces a dynamics in the annulus $\tilde A/T'$; the key point in the construction of the foliation with an iterated leaf will consist in applying Le Calvez's transverse foliation in this quotient. Note that in this section all foliations are just topological (non smooth) foliations.

We begin by recalling some definitions in this global setting. We work in the plane $\bbR^2$ which we compactify by adding the point $\infty$. The set $\bbR^2 \setminus \{0\}$ is identified with the open annulus $A = \mathbb{S}^1 \times \R$, whose ends $N,S$ correspond respectively to $0,\infty$ in this identification. We let $\tilde A$ be the universal cover of $A$, with deck transformation $T$. In the course of the proof we will encounter another transformation $T'$ of the plane $\tilde A$ which will be conjugate to $T$. Then the quotient $A'=\tilde A/T'$ is again an open annulus. We choose an invariant topological line for $T'$ and orient it in the sense in which $T'$ pushes points. In the quotient $A'$ the line becomes an oriented circle. We let $N',S'$ be the two ends of the annulus $A'$, so that $N'$ is on the lefthand side of this circle, and $S'$ is on the righthand side, and call them the North and South ends.

Let $f:\R^2 \to \R^2$ that fixes $0$, and $I$ an isotopy that fixes $0$, from the identity to $f$. This isotopy induces a lift $\tilde f:\tilde A \to \tilde A$. The (global) rotation set of $I$ in $A$ may be defined in this context. It is a closed subset of $\R \cup \{\pm\infty\}$ which, roughly speaking, contains the average speeds of rotation of long pieces of orbits that start and end not too close to one end of the annulus $A$.  We refer to~\cite{leroux13} or~\cite{conejeros16} for the precise definition. The definitions make it clear that this rotation set contains the local rotation set of $I$ around the fixed point $0$.

In this global setting we will apply Le Calvez's theorem, already mentioned in the local setting.
Assume for simplicity that $\tilde f$ has no fixed point; this will always be the case in what follows.
Then Le Calvez's theorem provides a topological oriented foliation $\cF$ of the annulus $A$ which is dynamically transverse to $I$: every trajectory of $I$ is homotopic to a curve which crosses each leaf of $\cF$ from left to right (this is made precise by using the charts of the foliation). As in the local setting, the property of being transverse is stable: any topological oriented foliation $\cF'$ of $A$ which is Whitney close to $\cF$ is again dynamically transverse to the isotopy $I$ (the proof is almost identical as the proof of the local stability, see~\cite{leroux13}, Section~4.4 b after Lemma 4.4.2).

Assume in addition that $f$ preserves the area of the plane $\R^2$. Then the transverse foliation is gradient-like near $N$ (see Section~\ref{section:transverse-foliations}): $N$ is a sink, a source or a saddle singularity for $\cF$. Moreover, $\cF$ has no closed leaf in $A$. However there may be some leaf going from $S$ to $S$ (``petals''), and some leaf $\phi$ whose $\alpha$-limit set is $\{N\}$ and whose $\omega$-limit set contains a petal leaf (the leaf spirals around and accumulate on a family of petals). 
This phenomenon is unstable: indeed in this situation, it is easy to prove that there exists an arbitrarily small perturbation of $\cF$ that produces a foliation $\cF'$ for which some leaf has $\alpha$-limit set $\{N\}$ and $\omega$-limit set $\{S\}$. Then,  by Poincar\'e-Bendixson's theory, every leaf of $\cF'$ whose $\alpha$-limit set is $\{N\}$ must have $\{S\}$ as its $\omega$-limit set. By the above mentioned stability property, $\cF'$ is still dynamically transverse to the isotopy $I$. To summarize, up to perturbing $\cF$, we may always assume that \emph{every leaf that comes from $N$ goes to $S$}, and likewise that \emph{every leaf that goes to $N$ comes from $S$}.

In particular, since $N$ is a sink, a source or a saddle, there is at least one  leaf going from $N$ to $S$ or from $S$ to $N$. By transversality, the first case implies that the rotation set of $I$ in $A$ in included in $[0,+\infty]$, and the second case that it is included in $[-\infty,0]$. A straightforward generalization of this argument shows that for every rational number $\frac{p}{q}$, if there exist no periodic orbit whose rotation number equals $\frac{p}{q}$, then the rotation set is included in $[-\infty,\frac{p}{q}]$ or in $[\frac{p}{q}, +\infty]$. In particular, the density of rational numbers implies that the rotation set is an interval.

\bigskip

The following lemma allows to extend a local homeomorphism into a homeomorphism of the plane, with a global rotation set equal or slightly larger than the local rotation set of the local homeomorphism. For related (unpublished) results, see~\cite{patou06}; the techniques are already used in~\cite{leroux13}, appendix A.

\begin{lemma}[Extension lemma]\label{lem:extension}
Let  $(f,U)$ be an area preserving local homeomorphism fixing $0$, where $U$ is a neighbourhood of $0$ in the plane. Let $(I,V)$ be a local isotopy associated to $f$. Let $[\alpha,\beta]$ be the local rotation set of $I$. 

Then there exists an area preserving homeomorphism $\overline{f}: \bbR^2 \to \bbR^2$, and an  isotopy $\overline{I}$ on $\bbR^2$ for $\overline{f}$,
such that $\overline{f}$ and $\overline{I}$ coincide respectively with $f$ and $I$ on a neighbourhood of $0$, and the rotation set of $\overline{I}$ in the annulus $A = \bbR^2 \setminus \{0\}$ is included in $[\alpha',\beta']$, where

\begin{itemize}
\item[(i)] if $\alpha > -\infty$ is a rational number different from $\beta$, and $0$ is not accumulated by periodic orbits of rotation number $\alpha$, then $\alpha'=\alpha$, furthermore in this case $\overline{f}$ has no periodic point of rotation number $\alpha$;
\item[(ii)] if $\alpha=-\infty$ then $\alpha'=\alpha$; in the other cases $\alpha'$ can be chosen to be any real number less than $\alpha$.
\end{itemize}
And $\beta'$ is defined by a symmetric set of conditions.
\end{lemma}
Note that since $\overline{I}$ extends $I$, the rotation set of $\overline{I}$ will automatically contain the local rotation set of $I$.
The following remark is added for completeness, though we will not need it here.
In the setting of the lemma, in the case when the local rotation set is a single rational number $\alpha$ and $0$ is not accumulated by periodic orbits of rotation number $\alpha$, a variation on the proof below produces an extension with rotation set included in an arbitrarily small interval $[\alpha-\varepsilon,\alpha+\varepsilon]$, one of whose endpoints is $\alpha$.


\begin{proof}
By a standard extension argument, we may assume $f$ is an area preserving homeomorphism of the plane (with no assumption on the global rotation set).

In the case when $\alpha'=\alpha$ (case (i) of the statement and when $\alpha=-\infty$) we set $\alpha''=\alpha'=\alpha$; otherwise choose any rational number $\alpha''$ such that $\alpha' < \alpha'' < \alpha$. Define $\beta''$ symmetrically. 
Let $F$ be the closure of the set of periodic points of $f$ in the annulus $A$ whose rotation number is $\alpha''$ or $\beta''$.
Clearly $F$ is a (maybe empty) closed subset of the annulus $A$.
By hypothesis (in particular the assumption on the local rotation set of $I$ in case (i)), $F$ is also a closed subset of the plane, \emph{i.e.}\ it does not accumulate $0$.

Let $M$ be the connected component of $\bbR^2 \setminus F$ that contains $0$, which is an open set invariant under $f$.
Let $\widetilde M$ be the universal cover of $M$, which is homeomorphic to the plane, and choose some  covering map $T'$ corresponding to a simple loop in the positive direction around $0$.
The annulus $A' = \widetilde M/T'$ projects to $M \setminus \{0\}$, and the projection $p$ restricts to a homeomorphism  between some neighborhood $V_{N'}$ of the end $N'$ of $A'$ and some neighborhood $V_{0}$ of $0$.
The map $f$ lifts to a homeomorphism $\tilde f$ of the plane $A' \cup \{N'\}$ which, near $N'$, is locally conjugate to $f$ near $0$. We may extend the restriction $p_{\mid V_{N'}}$ to a homeomorphism $\Phi$ between $A' \cup \{N'\}$ and $\bbR^2$. The map $\overline f' = \Phi \tilde f \Phi^{-1}$ is a homeomorphism of $\bbR^2$ that coincides with $f$ near $0$. Furthermore, $\overline f'$ preserves the image by $\Phi$ of the pull-back of the area by the covering map $p$, which is a measure that is positive on open sets and has no atoms. Using the Oxtoby-Ulam theorem, we see that this measure is the image of the area by a homeomorphism that is the identity near $0$. We finally conjugate $\overline f'$ by this homeomorphism to get an area preserving homeomorphisms $\overline f$ which coincide with $f$ near $0$.

Let $\overline{I}$ be an isotopy associated to $\overline{f}$ that extends (some restriction of) $I$. 
Let $\bar x$ be a periodic point of $\bar f$ with rational rotation number $\frac{p}{q}$ around $N$: this means that a lift of the trajectory $\overline{I}^q. \bar x$ joins a lift of $\bar x$ to its image under $\Phi T^p \Phi^{-1}$. Then $x = p\Phi^{-1}( \bar x)$ is a periodic point in $M$ with the same rotation number $\frac{p}{q}$ around $0$.
We conclude that $\overline{f}$ has no periodic point of rotation number $\alpha''$ nor $\beta''$.
Since  the rotation set of $\overline{I}$  relative to the annulus $\bbR^2 \setminus \{0\}$ is an interval, and because every rational number in the interior of this interval is  the rotation number of some periodic orbit, we deduce that
this interval does not contain $\alpha''$ nor $\beta''$ in its interior.
Furthermore it contains the local rotation set $[\alpha,\beta]$ of $I$ at $0$. Thus the rotation set of $\overline{I}$ is  included in $[\alpha'',\beta'']$.  In particular it is included in $[\alpha',\beta']$, as required.
\end{proof}

\bigskip

Let $A = \tilde A / T$ be an annulus as before, and  denote the quotient map by $\pi: \tilde A \to A$. Let $\mu$ be a measure on $\tilde A$ which is invariant under $T$. This measure induces a unique measure $\pi_\star \mu$ on $A$ with the property that for every topological disk $D$ in $\tilde A$  which projects injectively into $A$, $\pi_\star \mu(\pi(D)) = \mu(D)$. Conversely, given a measure $\mu_A$ on $A$, there is a unique measure $\pi^\star \mu_A$ on $\tilde A$ which is invariant under $T$ and such that $\mu_A$ is the measure induced by $\pi^\star \mu_A$; again we will say that $\pi^\star \mu_A$  is induced by $\mu_A$. Given a measure $\mu_A$ on $A$, we say that \emph{the end $N$ has finite measure} if some neighborhood of $N$ in $A$ has finite measure. 

A compact subset of the plane is said to be \emph{full} if its complement is connected. Compact connected full subsets of the plane are characterized as those subsets which can be written as a decreasing intersection of closed topological disks.

\begin{lemma}[Quotient lemma]~

\label{lemm:translation}
\begin{enumerate}
\item
Let  $g: \bbR^2 \to \bbR^2$ be  an area preserving homeomorphism of the plane fixing $0$, and ${J}$ be  an associated  isotopy fixing $0$.
Assume that the rotation set of ${J}$ in the annulus $A = \bbR^2 \setminus \{0\}$ is included in $(0,+\infty]$. Then the associated lift $T' = \tilde g: \tilde A \to \tilde A$ is conjugate to a translation.

\item Let $\mu$ denotes the area of the plane $\R^2$. Then $\mu$ induces a measure $\tilde \mu$ on $\tilde A$, and this measure in turn induces a measure $\mu'$ on $A' = \tilde A/T'$.
Assume furthermore that the local rotation set of $J$ around $0$ is bounded. Then the north end $N'$  of $A'$ has finite measure.

\item
Let $\pi: \tilde A \to A$ and $\pi': \tilde A' = \tilde A \to A'$ denote the quotient maps.
Let $(K_k)_{k \geq 0}$ be a decreasing sequence of compact connected full subsets of the plane $A \cup \{0\}$
whose intersection is $\{0\}$, each of which is invariant under $g$. Define $K'_k = \pi'(\pi^{-1}(K_k)) \cup \{N'\}$ for each $k$. Then $(K'_k)_{k \geq 0}$ is a decreasing sequence of compact connected full subsets of the plane $A' \cup \{N'\}$ whose intersection is ${N'}$.
\end{enumerate}
\end{lemma}

\begin{proof} 
To prove the first point, we try to follow the argument from Section 3 of~\cite{BCLR06}. The hypothesis there was that $g$ preserves a finite measure, so we have to adapt the proof to the present setting where the area is infinite. Note that here we may have some wandering points, so that even~\cite{guillou11} does not apply directly. (Also note that the definition of the rotation numbers in the annulus used in~\cite{BCLR06} involves only recurrent points and thus is not adapted to the present setting.)

Let  $R$ be the isotopy that makes one turn around $0$ in the positive direction, whose associated lift is the deck transformation $T$.
Let $g$ and $J$ be as in the statement.
Consider $q>0$ such that the rotation set is included in $(\frac{1}{q}, +\infty]$. 
Then the rotation set of the isotopy $J^q R^{-1}$ is included in $(0,+\infty]$.
Apply Le Calvez' theorem to get a singular foliation of the plane, with singularity $0$, transverse to the isotopy $J^q R^{-1}$. Since $g$ is area preserving, so is the time one map $h=g^q$ of this isotopy. As explained at the beginning of Section~\ref{sss.iterated-leaf-lemma}, up to perturbing $\cF$ we may assume that there is a leaf $\phi$ from $\infty$ to $0$.

Now the argument of~\cite{BCLR06} applies, let us recall it. 
Given two oriented proper lines $\beta, \beta'$, we will denote $\beta < \beta'$ if the two lines are disjoint, $\beta$ is on the left-hand side of $\beta'$ and $\beta'$ is on the right-hand side of $\beta$. Note that this is a transitive relation which is preserved when we apply an orientation preserving homeomorphism.
Let $\tilde h = \tilde g^q T^{-1}$, which is the lift of $h$ associated to the isotopy $J^q R^{-1}$.
Let $\tilde \phi$ be a lift of $\phi$, we have
$$
\tilde h^{-1}(\tilde \phi) < \tilde \phi < \tilde h(\tilde \phi)
$$
($\tilde \phi$ is said to be an \emph{oriented Brouwer line} for $\tilde h$).
From this we first get
$$
\tilde g^{-q}(\tilde \phi) < T^{-1}(\tilde \phi) \text{ and }  T(\tilde \phi) < \tilde g^{q}(\tilde \phi).
$$
Let $B'$ denote the closed strip bounded by $\tilde \phi$ and  $\tilde g^q(\tilde \phi)$. We claim that the union of the iterates of $B'$ under $\tilde g^q$ cover the plane:
indeed, from the previous inequalities we get, for every $n>0$,
$$
\tilde g^{-nq}(\tilde \phi) < T^{-n}(\tilde \phi) \text{ and }  T^n(\tilde \phi) < \tilde g^{nq}(\tilde \phi)
$$
and moreover the strip between $\tilde \phi$ and $T(\tilde \phi)$ is a fundamental domain for $T$.
We deduce that $\tilde g^q$ is conjugate to a translation. By a classical lemma, $\tilde g$ is also conjugate to a translation (indeed the quotient $\tilde A/g$ is a quotient of the annulus $\tilde A/g^q$ by a covering map of degree $q$, thus it is also an annulus).

\bigskip
Note that some orbit of $g$ may accumulates on both ends of $A$, which is why it is not so easy to compare properties in $A$ and in $A'$ in what follows.
\bigskip

Let us prove the second point. We assume that the local rotation set of $J$ around $N$ is bounded.
The annulus $\tilde A/{\tilde g}^q$ is a $q$-fold cover of the annulus $A'$, thus the end $N'$ of $A'$ has finite measure if and only if the corresponding end of $\tilde A/{\tilde g}^q$ has finite measure for the measure induced by $\tilde \mu$. Thus up to changing $g$ into $g^q$, we may assume that the rotation set of $J$ in $A$ is included in $(1, +\infty]$. Let $\phi$ be a topological line in $A$ as in the first point, whose lift $\tilde \phi$ satisfies
$$
\tilde \phi < T(\tilde \phi) < \tilde g(\tilde \phi).
$$
The band $B'$ between $\tilde \phi$ and $\tilde g (\tilde \phi)$ is a fundamental domain for $A'$.
Choose some curve $\tilde \alpha$ joining a point on $\tilde \phi$ to its image under $\tilde g$, and whose interior is included in the interior of the band $B'$.
Then $\alpha' = \pi'(\tilde \alpha)$ is a Jordan curve surrounding $N'$, we aim to prove that  the disk bounded by $\alpha'$ in $A'$ has finite $\mu'$-measure.
 Denote by $\tilde \phi_{+}$ the half-leaf of $\phi$ after the point $\tilde \alpha(0)$ on $\tilde \phi$. 
Let $B'_{+}$ be the connected component of $B' \setminus \tilde \alpha$ bounded by $\tilde \phi_{+}, \tilde \alpha, \tilde g(\tilde \phi_{+})$, which projects into this disk.
By perturbing $\phi$ if necessary we may assume $\phi$ has zero measure, and then the measure of this disk is also the $\tilde \mu$-measure of $B'_{+}$.

By modifying $g$ far from $0$, we may find an area-preserving homeomorphism $g_{c}$ of the plane $A\cup \{N\}$ which is compactly supported, and which has a lift $\tilde g_{c}$ which is the identity on a half-leaf of $\tilde \phi$, and coincides with $\tilde g$ on $\tilde \phi^+$. We denote by $J_{c}$ the corresponding isotopy on $A$ from the identity to $g_{c}$. To simplify the picture, up to modifying again $g_{c}$ on a compact band of $A$, we may assume that $\tilde g_{c}(\tilde \phi)$ is included in the closed half-plane on the right-hand side of $\tilde \phi$, and in particular the rotation set of $J_{c}$ in $A$ is still non negative.
Note that $g_{c}$ is the identity except on a set of finite measure, thus we may apply the Poincar\'e recurrence theorem: $\mu$ almost every point is recurrent under $g_{c}$.
Let $Z$ be the set of points which are on the right-hand side of $\tilde \phi$ and on the left-hand side of $\tilde g_{c} (\tilde \phi)$ (note that we do not care about measure zero sets, thus from now it is enough to define sets up to measure zero). Note that $Z$ equals $B'_{+}$ up to a finite $\tilde \mu$-measure set, and thus it is sufficient to prove that $Z$ has finite measure. Moreover  the local rotation set of $J_{c}$ around $N$ equals that of $J$, and thus is bounded by hypothesis. Since $g_{c}$ has compact support in $A \cup \{N\}$, we deduce that the (global) rotation set of $J_{c}$ in the annulus $A$ is also bounded, say it is included in $[0,M]$.
The idea of what follows is that the measure of $Z$ is the mean rotation number of $J_{c}$ in $A$, which is finite since the rotation set of $J_{c}$ is bounded.

Here are the details. Let $Q$ denote the band bounded by $\tilde \phi$ and $T(\tilde \phi$), $Q'=\tilde g_{c}(Q)$, $P_{i,j} = T^i(Q) \cap T^j(Q')$. The half-plane on the right-hand side of $\tilde \phi$ is $\cup_{i \geq 0} T^i(Q)$, the half-plane on the left-hand side of $\tilde g_{c} (\tilde \phi)$ is $\cup_{j<0} T^j(Q')$, and thus
$$
Z = \bigcup_{i\geq0, j<0} P_{i,j}.
$$
For each $i,j$ we have $T(P_{i,j})=P_{i+1, j+1}$ and since $\tilde \mu$ is invariant under $T$ we get
$$
\tilde \mu (Z) = \sum_{i \geq 0} i \tilde \mu(P_{i,0}).
$$
Now we establish the relation with the mean rotation number. For almost every point $x$ in $A$, we may define the algebraic intersection number $J_{c}.x \wedge \phi$ of the trajectory of $x$ under $J_{c}$ with the leaf $\phi$. Furthermore, since $\phi$ is a leaf of a foliation transverse to $J$, the function $x\mapsto J.x \wedge \phi$ is non-negative, and thus the function $x \mapsto J_{c}.x \wedge \phi$ is bounded below. In addition it vanishes except on a subset of $A$ which has finite area. In particular, its integral is well defined (maybe infinite).
Let $x \in A$, let $\tilde x$ be a lift of $x$ in the fundamental domain $Q$ of $A$, and let $i=J_{c}.x \wedge \phi$.
Then $\tilde g(\tilde x)$ belongs to $T^i(Q) \cap Q'= P_{i,0}$. Thus for every integer $i$, the points $\tilde x$ in $Q$ for which $J_{c}.x \wedge \phi=i$
are exactly the elements of $\tilde g_{c}^{-1}(P_{i,0})$. Denoting by $\mu$ the area of the plane $A \cup \{N\}$, since $\tilde \mu$ is preserved by $\tilde g_{c}$, we get
$$
\int_{A} \left( J_{c}.x \wedge \phi \right) d \mu = \sum_{i \geq 0} i \tilde \mu( \tilde g_{c}^{-1}(P_{i,0}) ) = \sum_{i \geq 0} i \tilde \mu(P_{i,0}) = \tilde \mu(Z).
$$
It remains to show that this integral is finite. Since  $x \mapsto J_{c}.x \wedge \phi$ is bounded below we may apply Birkhoff's ergodic theorem for positive functions: 
the function
$$
\rho(x) = \lim_{n \to +\infty} \frac{1}{n} J_{c}^n.x \wedge \phi
$$
is defined $\mu$-almost everywhere on $A$ (allowing the value $+\infty$), and its integral equals that of $x \mapsto J_{c}.x \wedge \phi$. Let $x$ be a recurrent point for which this limit exists, then it is easy to see that the number $\rho(x)$ belongs to the rotation set of $J_{c}$. Thus $\rho(x) \leq M$ almost everywhere, and since $\rho$ vanishes outside a finite measure subset, we get  
$\int_{A}\rho(x) d\mu < +\infty$. This completes the proof of the second point.

\bigskip

Let us prove the last point of the lemma. Let $K$ be a compact connected full subset of $A\cup \{N\}$ which strictly contains $N$ and is invariant under $g$. Denote $U_{K} = A \setminus K$, which is an essential invariant sub-annulus of $A$. Let $\tilde K = \pi^{-1}(K), \tilde U_{K} = \pi^{-1}(U_{K})$: $\tilde U_{K}$ is a topological plane which is invariant by both $T$ and $T'=\tilde g$. Since $T'$ is conjugate to a translation, so is its restriction to $\tilde U_{K}$, and thus $\pi'(\tilde U_{K})$ is an essential open sub-annulus of $A'$ whose complement in $A' \cup \{N'\}$ is $K':=\pi'(\tilde K) \cup \{N'\}$, and we want to show that this is a compact connected full subset of the plane $A' \cup \{N'\}$. We can write
$$
A' \cup \{S'\} = K'_{S'} \sqcup \pi'(\tilde U_{K}) \sqcup K'_{N'}
$$
where  $K'_{S'}, K'_{N'}$ are the connected components of the complement of $\pi'(\tilde U_{K})$ in the sphere $A' \cup \{S',N'\}$, respectively containing $S'$ and $N'$. By considering a nested family of topological disks bounded by simple closed curves included in $\pi'(\tilde U_{K})$, we see that $K'_{N'}$ is a compact connected full subset in the plane $A' \cup \{N'\}$, and likewise for $K'_{S'}$ in the plane  $A' \cup \{S'\}$. 

We will prove that $K'_{S'} = \{S'\}$. For this we consider an oriented topological line $\tilde \phi$ as in the proof of the first point, with $\tilde \phi < T(\tilde \phi)$ and $T(\tilde \phi) < T'^q(\tilde \phi)$.
We denote by $B$ the closed band between $\tilde \phi$ and $T(\tilde \phi)$. Likewise let $B'$ be the closed band between $\tilde \phi$ and $T'^q(\tilde \phi)$, which we compactify by adding two ends $\tilde N, \tilde S$, with $\tilde N$ on the left-hand side of any simple curve that goes from $\tilde \phi$ to $T'^q(\tilde \phi)$. Note that $B'$ contains $B$, and that $B \cup \{\tilde S, \tilde N\}$ is a compactification of the band $B$. The natural maps $B \cup \{\tilde S, \tilde N\} \to A\cup \{S,N\}$ and $B' \cup \{\tilde S, \tilde N\} \to A'\cup \{S',N'\}$ induced by the projection $\pi,\pi'$ are continuous: indeed $B$ is a fundamental domain for $A$; and $B'$ is a fundamental domain for the annulus $\tilde A/T'^q$ and $\pi'$ is the composition of the projection from $\tilde A$ to $\tilde A/T'^q$ and from $\tilde A/T'^q$ to $A'=\tilde A/T'$. We claim that the map from $B' \cup \{\tilde S, \tilde N\}$ to $A\cup \{S,N\}$ induced by $\pi$, that sends $\tilde S'$ to $S$ and $\tilde N'$ to $N$, is also continuous. To see this let us identify $A$ with $\mathbb{S}^1 \times \R$; since $\phi$ is a topological line from $S$ to $N$ in $A$, we may assume that in these coordinates $\phi$ is parametrized by $\phi(t) = (0,t)$. Since $g$ is continuous and fixes $N,S$, denoting $g^q(\phi(t)) = (x(t), y(t))$ we get that the function $y$ is proper, namely $y(t)$ tends to $\pm \infty$ when $t$ tends to $\pm \infty$. In these coordinates the projection $\pi$ writes as $\pi(\tilde x, \tilde y) = (\tilde x \text{ mod } 1, \tilde y)$, $\tilde \phi(t) = (0,t)$, and $\tilde g^k(\tilde \phi(t)) = (\tilde x(t),y(t))$. Let $M>0$, since the function $y$ is proper, we get that the intersection of the band $B'$ with the horizontal band $\R \times [-M,M]$ is compact. Let $(x_{k} \text{ mod } 1, y_{k})$ be a sequence in $B'$ that tends to $\tilde S$ (respectively $\tilde N$), then it has finitely many terms in each given compact subset, and $(y_{k})$ must tends to $-\infty$ (respectively $+\infty$). Thus the projection of the sequence in $A$ tends to $S$ (respectively $N$), which proves the claim.

Now assume by contradiction that $K'_{S'} \neq \{S'\}$. Then, since $K'_{S'}$ is connected, there exists a sequence $(x'_{k})$ in $K'_{S'} \neq \{S'\}$ that converges to $S'$. Let $(\tilde x_{k})$ be a pre-image under $\pi'$ of this sequence in $B'$, and note that it is included in $\tilde K$.
Since the natural map $B' \cup \{\tilde S, \tilde N\} \to A'\cup \{S',N'\}$ is continuous, the sequence  $(\tilde x_{k})$ converges to $\tilde S$. Let $x_{k}= \pi(\tilde x_{k})$; since the map  $B' \cup \{\tilde S, \tilde N\} \to A\cup \{S,N\}$ induced by $\pi$ is continuous we get that the sequence $(x_{k})$ converges to $S$. But this last sequence is included in $K$, and this contradicts the fact that $K$ is a compact subset of $A\cup \{N\}$.

Finally if $(K_{k})$ and $(K'_{k})$ are as in the hypothesis of the lemma, then clearly $(K'_{k})$ is a decreasing sequence of compact connected full subsets of $A' \cup \{N'\}$, and it is easy to see that its intersection is $\{N'\}$. The proof of the lemma is complete.
\end{proof}

\begin{proof}[Proof of the ``iterated leaf lemma'']
By the ``extension lemma''~\ref{lem:extension}, we may assume that $f$ is a homeomorphism of the plane that preserves the area, whose (global) rotation set is included in $[0,\frac{1}{q})$, and which has no fixed point of rotation number $0$ by the second point of the conclusion of this lemma (and thus no fixed point at all).

Let $\cF$ be a transverse oriented foliation on $\bbR^2 \setminus \{0\}$ for $I$. Recall the following properties from Theorem~\ref{theo.locally-transverse-foliation} and Section~\ref{sec:local-homeo-area-pres}.
Since $f$ preserves the area, $\cF$ has no circle leaf, and no petal. Since $L(f,0)=1$, the index of the singularity $0$ for $\cF$ must also be equal to 1. Thus $0$ is a sink or a source for $\cF$. If the local rotation set contains some positive number then $0$ has to be a sink. If the local rotation set is $\{0\}$ and $0$ is a source for $\cF$, then up to changing $f$ into its inverse we may assume again that $0$ is a sink for $\cF$, without changing the hypotheses of the lemma; note that if $f^{-1}$ satisfies the conclusion of the lemma then so does $f$. Up to modifying $\cF$ as in the proof of the quotient lemma, we may assume that there is a proper leaf that crosses the annulus, that is, its $\omega$-limit set is $0$ and its $\alpha$-limit set is the other end $S$ of the annulus.  Note that as a consequence of the Poincar\'e-Bendixson theorem, this is the case for every leaf whose $\omega$-limit set is $0$.

Consider the isotopy  $J = R I^{-q}$, and let $A, \tilde A, T$ be as in the ``quotient lemma''~\ref{lemm:translation}.
Let $\tilde f:\tilde A \to \tilde A$ be the lift of $f$ associated to $I$.
The lift associated to $J$ is $T' = T \tilde f^{-q}$. Its rotation set is included in $(0,1]$. According to the quotient lemma, $T'$ is conjugate to a translation. We consider the quotient annulus $A' = \bbR^2 / T'$, and the measure $\mu'$ induced by the area of the plane.

The map $\tilde f$ commutes with $T'$, thus it induces a homeomorphism $f'$ of the annulus $A'$ which fixes the end $N'$ of $A'$. Let $I'$ be an isotopy from the identity to $f'$ on $A'$, whose associated lift is $\tilde f$. Note that the trajectories of any point $\tilde x$ under $\tilde I$ and $\tilde I'$ are homotopic relative to their endpoints $\tilde x, \tilde f (\tilde x)$, and thus an oriented foliation is transverse to one if and only if it is transverse to the other one.

The homeomorphism $f'$ preserves the measure $\mu'$, and  according to the quotient lemma the end $N'$ of $A'$ has finite measure. Let $\cF'$ be a transverse oriented foliation for $I'$ on $A'$. As a consequence, $\cF'$ has no closed leaf and no petal at $N'$, thus $N'$ is either a source, a sink, or a saddle singularity for $\cF'$. Up to perturbing $\cF'$ we can assume that it has a leaf that crosses $A'$ either from $N'$ to $S'$ or from $S'$ to $N'$, and then this is the case for every leaf whose $\omega$ or $\alpha$-limit set is $\{N'\}$.
\begin{claim}\label{claim.source}
$N'$ is not a source for $\cF'$.
\end{claim}
Let us first admit the claim and complete the proof of the lemma. According to the claim, $N'$ is either a sink or a saddle singularity of $\cF'$; in both cases there is a leaf $\alpha$ whose  $\omega$-limit set is $\{N'\}$. Let $\tilde \alpha$ be a lift of $\alpha$ in $\bbR^2$.
The strip $B'$ between $\tilde \alpha':=T'^{-1}(\tilde \alpha) =  (T \tilde f^{-q})^{-1}(\tilde \alpha)$ and $\tilde \alpha$  is a fundamental domain for  $A'$, thus the foliation $\cF'$ lifts to a foliation $\tilde \cF'$ of this strip. 
Furthermore we have
$$
\tilde \alpha' < \tilde \alpha = T'(\tilde \alpha') < T(\tilde \alpha')
$$
and from this we deduce that the strip $B$ bounded by $\tilde \alpha'$ and $T(\tilde \alpha')$ is a fundamental domain for $T$ (the argument is entirely analogous to the one used in the proof of the first point of the quotient lemma). In particular the projection of $\tilde \alpha'$ in $A$ is a proper line that goes from $\infty$ to $0$.



Foliate the strip bounded by $\tilde \alpha$ and $ \tilde f(\tilde \alpha)$ with with any foliation $\cF_{0}$ homeomorphic to the foliation by parallel lines, then for every $p \in \{1, \dots, q-1\}$, foliate the strip bounded by $\tilde f^p(\tilde \alpha)$ and $\tilde f^{p+1}(\tilde \alpha)$ by the foliation $\tilde f^p(\cF_{0})$. By gluing those foliations together with the foliation $\tilde \cF'$ of $B'$ we  get a foliation $\tilde \cF''$ of the fundamental domain $B$. Furthermore, this foliation is easily seen to be transverse to $\tilde f$. The projection of $\tilde \cF''$ to our original plane yields a foliation $\cF''$ of $A$ which is transverse to $I$. Also note that since $\tilde \alpha' < T \tilde \alpha'$, the $\omega$-limit set of the projection of $\tilde \alpha'$ in $A$ is $\{0\}$. Thus $0$ is either a saddle or a sink for $\cF''$.
But since $L(f,0)=1$, by Theorem~\ref{theo.locally-transverse-foliation} $L(\cF'',0)=1$ and $0$ cannot be a saddle singularity of $\cF''$ .
Thus it is a sink and $\cF''$ satisfies the conclusion of the lemma (by the way we also see now that $N'$ was a sink of $\cF'$ and not a saddle). 

It remains to prove Claim~\ref{claim.source}. For this  we first relate the rotation set $\cal R$ of $I$ in the annulus $A$ and the rotation set $\cal R'$ of $I'$ in the annulus $A'$ (note that $\cal R$ and $\cal R'$ are global rotation sets, not local ones). First note that $I'$ has no contractible fixed point in $A'$, because this would correspond to a fixed point of $\tilde f$. 
Next, we claim that $\cal R'$ is the image of $\cal R$ under the map
$$
\rho \mapsto \frac{\rho}{1-q\rho}.
$$ 
Indeed, this can be checked directly by using the definitions of the rotation sets in $A$ and $A'$. (Note that in the case when $\cal R$ is a non trivial interval, every rational number in the interior of $\cal R$ is realized by a periodic orbit, and then it suffices to check that if a point $\tilde x$ projects to a periodic point in $A$ with rotation number $\rho$ then it  projects to a periodic orbit in $A'$ with rotation number $ \frac{\rho}{1-q\rho}$).


First assume that  the rotation set $\cal R$ is not $\{0\}$.
Since it is included in $[0,\frac{1}{q})$, it must contain some positive number.
Then by the above formula relating $\cal R$ and $\cal R'$, the rotation set $\cal R'$ also contains some positive number. Thus no leaf of $\cF'$ can cross $A'$ from $N'$ to the other end. This shows that $N'$ is not a source of $\cF'$ (nor a saddle).

It remains to address the case when the rotation set ${\cal R}=\{0\}$.  This case is conjectured not to happen, but in the absence of a proof we have to cope with it.
In particular, there is no periodic orbit: indeed since $I$ is transverse to the foliation $\cF$, a periodic orbit must have non zero rotation number, contrarily to the assumption on $\cal R$.

Since $I$ is transverse to a sink foliation $\cF$ on $A$, the isotopy $I$ has ``some positive rotation''. The idea of what follows is to try to see this positive rotation on the dynamics of some compact invariant set near the fixed point $0$, and then to track this positive rotation in the universal cover $\tilde A=\tilde A'$ where it will in turn imply some positive rotation for $I'$ in $A'$; this will prevent $I'$ to be transverse to a source foliation.

To implement this idea we will heavily use the material in~\cite{lecalvez03}. Following Le Calvez, we say that the point $0$ is indifferent for $f$ if for every small enough Jordan domain $U$ containing $0$, the connected component $K$ of the set $\cap_{k \in \mathbb{Z}} f^{-k}(\overline U)$ that contains $0$ meets the boundary of $U$. Let us prove that in our situation, the fixed point is indifferent for $f$. Assume by contradiction that $0$ is not indifferent. Since the area is preserved, it is neither a sink nor a source, and
the Le Calvez-Yoccoz theorem applies (\cite{LCY97}): there is some $q>0$ such that $L(f^q,0) < 1$. Then the local rotation set of $g=f^q$ is $0$ modulo $1$ (see for instance~\cite{leroux13}). There is a local isotopy $J$ from the identity to $g$ whose local rotation set is $\{0\}$ ; this local isotopy is also characterized, up to homotopy, by having nonpositive index (\cite{leroux04}). On the other hand the isotopy $I^q$ has index one since it is transverse to the sink foliation $\cF$. Thus $I^q$ is not homotopic to $J$, and its local rotation set is a non zero integer. On the other hand its local rotation set should be $q$ times the local rotation set of $I$, which is $\{0\}$, a contradiction.

Thus the fixed point is indifferent. Let $U$ be a small Jordan domain containing $0$ in $\bbR^2$, such that $K$ meets the boundary of $U$.
Let $U_{K}$ be the unbounded connected component of $\bbR^2 \setminus K$. This open set is an annulus, and the end corresponding to $K$ admits the \emph{prime ends compactification}, which is a topology on the disjoint union $U_{K} \sqcup \mathbb{S}^1$ that is homeomorphic to the half-infinite annulus $\mathbb{S}^1 \times [0,+\infty)$ (see~\cite{lecalvez03}, Section~4). Since $f(K) = K$ the restriction of $f$ to $U_{K}$ extends continuously to a circle homeomorphism $\hat f: \mathbb{S}^1 \to \mathbb{S}^1$.

The universal cover of $U_{K} \sqcup \mathbb{S}^1 \simeq \mathbb{S}^1 \times [0,+\infty)$ identifies with $\tilde U_{K} \sqcup \Delta \simeq \bbR \times [0,+\infty)$ where $\tilde U_{K}$ is included in $\tilde A$ and $\Delta$ is the universal cover of the circle compactification $\mathbb{S}^1$.
The lift $\tilde f$ of $f$ to $\tilde U_{K}$ induces on $\bbR$ a homeomorphism which is the natural lift ${\bar f}$ of $\hat f$. Since the rotation set of $I$ in $A$ is $\{0\}$,the natural lift has translation number $0$ for the covering transformation $\bar T$ induced by $T$ on $\Delta$. We orient $\Delta$ according to $T$.
A curve $\gamma$ is called an \emph{access arc} of $K$  if it is included in $U_{K}$ except for one end-point $ \gamma(0) = p$ which belongs to $K$. A classical result from prime-ends theory is that $\gamma \setminus\{p\}$ has a limit in the prime-end compactification, \emph{i. e.} there exists a point $\hat p \in \mathbb{S}^1$ such that $\gamma \setminus\{p\} \cup \{\hat p\}$ is a continuous curve in $U_{K} \sqcup \mathbb{S}^1$. The following claim uses access arcs to detect to sign of the rotation on $\mathbb{S}^1$.
\begin{claim}
 Let $\Gamma : (-\infty, 0] \to \tilde A$ be an injective curve whose projection under the quotient map $\pi : \tilde A \to A$ is proper, with $\pi \circ \Gamma(t)$ converging to the end opposite to $0$ when $t$ tends to $-\infty$, such that $\pi \circ \Gamma((-\infty,0) ) \subset U_{K}$ and $\pi \circ \Gamma(0) \in K \setminus \{0\}$.
Assume that $\tilde f(\Gamma)$ and $T(\Gamma)$ are disjoint from $\Gamma$. 
Finally let $\bar p$ be the point in $\Delta$ to which $\Gamma(t)$ converges when $t$ tends to $0$.
Then $ {\bar f}(\bar p) > \bar p$  if and only if $\tilde f(\Gamma)$ and $T(\Gamma)$ are in the same connected component of $\tilde U_{K} \setminus \Gamma$.
\end{claim}
This claim follows from the considerations in~\cite[Section 5.1]{lecalvez03}.

Let $\alpha$ be any leaf of $\cF$ that goes from the other end of $A$ to $0$ and that meets $K \setminus \{0\}$, let $\tilde \alpha$ be some lift of $\alpha$ in $\tilde A$. Then let $\Gamma \subset \tilde \alpha$ be obtained from $\tilde \alpha$ by running along $\alpha$ from $-\infty$ till the first point on $K$. Since $\cF$ is positively transverse to $I$, $\tilde f(\Gamma)$ and $T(\Gamma)$ are both on the right-hand side of $\tilde \alpha$. Thus they are included in the same connected component of $\tilde U_{K} \setminus \Gamma$. The criterion applies and provides a point $\bar p$ in $\Delta$ such that ${\bar f}(\bar p) > \bar p$.

Let $\tilde K$ be the inverse image of $K$ in $\tilde A$.
Let $K'$ be the projection in $A'$ of $\tilde K$, to which we add the fixed point $N'$.
According to point (3) of the quotient lemma, $K'$ is a compact connected full set which is invariant under $f'$ and included in some small neighbourhood of $N'$.
Note that $\tilde U_{K} \sqcup \Delta$ also identifies with the universal cover of the prime end compactification of the complement $U_{K'}$ of $K'$ in $A'$. Also note that the rotation number defined by Proposition~4.1 in~\cite{lecalvez03}  coincides with the element of our local rotation set by~\cite[Corollaire 3.14]{leroux13}, and thus is equal to $0$.

 Remember that $T' = T \tilde f^{-q}$. Let $\bar T'$ be induced by $T'$ on $\Delta$, we get $\bar T' = \bar T \bar f^{-q}$. Since the translation number of $\bar f$ on $\Delta$ with respect to $T$ is zero, there is a point $\bar p_{0}$ of $\Delta$ which is fixed by $\bar f$, and thus $\bar T'(\bar p_{0})=\bar T(\bar p_{0})$. Thus the orientations of $\Delta$  according to $T$ and $T'$ coincide. Now to finish the proof of Claim~\ref{claim.source}, let us assume by contradiction that  $N'$ is a source for $\cF'$.
The above argument applies symmetrically to any leaf of $\cF'$ that crosses the annulus $A'$, and provides a point $\bar p'$ in $\Delta$ such that ${\bar f}(\bar p') < \bar p'$.

Now we will get a contradiction following an argument of~\cite{lecalvez03}. The point $0$ is indifferent and non accumulated. Theorem 9.4 of~\cite{lecalvez03} applies and, since the rotation number is zero and $L(f,0)=1$ in our case, says that the sequence of indices $L(f^k,0)$ is constantly equal to $1$.
Then Proposition 9.5 of~\cite{lecalvez03} applies and says that, given that the rotation number is zero, either every point $p$ satisfies ${\bar{f}}(p) > p$ or every point $p$ satisfies ${\bar{f}}(p) < p$. This contradicts the existence of $p$ and $p'$ and completes the proof of Claim~\ref{claim.source}.

Alternatively, we may argue as in point (c) of the proof of Proposition 9.5 in~\cite{lecalvez03} (this argument is more precise, but it makes use of many definitions of~\cite{lecalvez03}). Let $O_+$ be the projection in $\mathbb{S}^1$ of the connected component containing $p$ of the complement in $\Delta$ of the set of fixed points of ${\bar f}$. Likewise, $p'$ defines a component $O_{-}$. By identifying the connected components of the complement of $O_{+} \cup O_{-}$ into two points, we get a quotient of $\mathbb{S}^1$ on which $\hat f$ induces a North-South map.
 According to Lemma 7.1 and Proposition 7.2 of~\cite{lecalvez03}, there is a finer factor $F$ of $\hat f$ which is finite and not indifferent, and its restriction to $\mathbb{S}^1$ has at least one sink and one source with respect to the dynamics on $\mathbb{S}^1$. The sinks and sources are saddle points of $F$, and the formula of Proposition 8.1 says that $L(f,0)<1$, which is a contradiction.
\end{proof}

%
%

\bibliographystyle{abbrv}
\bibliography{biblio}

\end{document}